\documentclass{lmcs}
\pdfoutput=1
\usepackage[T2A,T1]{fontenc}
\usepackage[utf8]{inputenc}
\usepackage[russian,english]{babel}

\usepackage{tikz}
\usepackage{tikz-cd}
\definecolor{nordred}{HTML}{bf616a}
\definecolor{bordeaux}{HTML}{821529}
\definecolor{bluelink}{HTML}{003399}
\definecolor{nordred}{HTML}{bf616a}
\definecolor{nordblue}{HTML}{81a1c1}
\definecolor{norddarkblue}{HTML}{5e81ac}
\definecolor{nordgreen}{HTML}{a3be8c}
\definecolor{nordnight}{HTML}{4c566a}

\AtEndPreamble{
  \RequirePackage{hyperref}
  \hypersetup{
    breaklinks = true,
    linktocpage,
    colorlinks = true,
    linkcolor = nordnight,
    citecolor = nordgreen,
    urlcolor = nordblue
  }
} %
\usepackage{amssymb}
\usepackage{amsmath}
\usepackage{amsthm}
\usepackage{stmaryrd}
\usepackage{mathpartir}
\usepackage{bbold}

\usepackage{stdnamesmacros}
\usepackage{probabilitymacros}
\usepackage{monoidmacros}

\usepackage{stringdiagrams}
\usepackage{probabilityfigures}
\usepackage{monoidfigures}

\usepackage[hypertexnames=false]{hyperref} %
\usepackage{cleveref} %

\definecolor{med0}{HTML}{1C1B1B} %
\definecolor{med1}{HTML}{261D1D} %
\definecolor{med2}{HTML}{362B2B} %
\definecolor{med3}{HTML}{5E5757} %
\definecolor{med4}{HTML}{FFE983} %
\definecolor{med5}{HTML}{FFF4C2} %
\definecolor{med6}{HTML}{FDF6E3} %
\definecolor{med7}{HTML}{DB7842} %
\definecolor{med8}{HTML}{B32E39} %
\definecolor{med9}{HTML}{821529} %
\definecolor{medA}{HTML}{602E51} %
\definecolor{medB}{HTML}{FFC929} %
\definecolor{medC}{HTML}{60C37E} %
\definecolor{medD}{HTML}{89D7D0} %
\definecolor{medE}{HTML}{3E75DA} %
\definecolor{medF}{HTML}{D0ADE1} %
\colorlet{medBlack}{med0}
\colorlet{medWhite}{med6}
\colorlet{medRed}{med8}
\colorlet{medBlue}{medE}
\definecolor{nordred}{HTML}{bf616a}
\definecolor{bordeaux}{HTML}{4b1121}
\definecolor{darkyellow}{HTML}{FFC20A}
\definecolor{nicered}{HTML}{9C0D38}
\definecolor{niceblue}{HTML}{0C7BDC}
\colorlet{localblack}{black}
\colorlet{localwhite}{white}
\colorlet{localcolor}{medB}
\colorlet{localgray}{med3}
\colorlet{localblue}{medE}
\colorlet{localred}{med8}
\colorlet{addcolor}{white}
\colorlet{copycolor}{black}

\crefname{defi}{Definition}{Definitions}
\crefname{thm}{Theorem}{Theorems}
\crefname{lem}{Lemma}{Lemmas}
\crefname{prop}{Proposition}{Propositions}
\crefname{cor}{Corollary}{Corollaries}
\crefname{exa}{Example}{Examples}
\crefname{rem}{Remark}{Remarks}

\allowdisplaybreaks{}

\newcommand{\hint}[2]{\smash{\overset{(#1)}{#2}}}

\newcommand\m[1]{\({#1}\)}
\newcommand\mm[1]{\[{#1}\]} %

\usepackage[gen]{eurosym}
\usepackage[utf8]{inputenc}
\usepackage{cmll}

\newcommand\scalemath[3]{\scalebox{#1}[#2]{\mbox{\ensuremath{\displaystyle #3}}}}

\newcommand{\leftarrowtip}{\ensuremath{\tikz\draw[line width=0.5pt,->] (10pt,0) -- (0,0);}}
\newcommand{\leftarrowtailnotip}{\ensuremath{\tikz\draw[line width=0.5pt,-<] (0,0) -- (10pt,0);}}

\newcommand{\unicodeStar}{\ensuremath{\star}}
\DeclareUnicodeCharacter{2605}{\unicodeStar}

\DeclareUnicodeCharacter{21D2}{\ensuremath{\Rightarrow}}
\DeclareUnicodeCharacter{2218}{\ensuremath{\circ}}
\DeclareUnicodeCharacter{2022}{\ensuremath{\bullet}}
\DeclareUnicodeCharacter{2219}{\ensuremath{\bullet}}
\DeclareUnicodeCharacter{2026}{\ensuremath{\dots}}
\DeclareUnicodeCharacter{2208}{\ensuremath{\in}}
\DeclareUnicodeCharacter{2192}{\ensuremath{\to}}
\DeclareUnicodeCharacter{2190}{\ensuremath{\leftarrowtip}}
\DeclareUnicodeCharacter{2919}{\ensuremath{\leftarrowtailnotip}}

\DeclareUnicodeCharacter{00D7}{\ensuremath{\times}}
\DeclareUnicodeCharacter{00B7}{\ensuremath{\cdot}}
\DeclareUnicodeCharacter{222B}{\ensuremath{\int}}
\DeclareUnicodeCharacter{22A4}{\ensuremath{\top}}
\DeclareUnicodeCharacter{22A5}{\ensuremath{\bot}}
\DeclareUnicodeCharacter{2264}{\ensuremath{\leq}}

\newcommand{\unicodecolon}{\ensuremath{\colon}}
\DeclareUnicodeCharacter{FE55}{\unicodecolon}
\newcommand{\unicodeleftpar}{\ensuremath{\left(}}
\DeclareUnicodeCharacter{27EE}{\unicodeleftpar}
\newcommand{\unicoderightpar}{\ensuremath{\right)}}
\DeclareUnicodeCharacter{27EF}{\unicoderightpar}
\DeclareUnicodeCharacter{2260}{\neq}
\DeclareUnicodeCharacter{22A9}{\Vdash}
\DeclareUnicodeCharacter{2237}{\proportion}
\DeclareUnicodeCharacter{2124}{\mathbb{Z}}
\DeclareUnicodeCharacter{27E8}{\langle}
\DeclareUnicodeCharacter{27E9}{\rangle}
\DeclareUnicodeCharacter{21A6}{\mapsto}
\DeclareUnicodeCharacter{22A2}{\vdash}
\DeclareUnicodeCharacter{2090}{\ensuremath{{}_a}}
\DeclareUnicodeCharacter{A71B}{{}^\uparrow}
\DeclareUnicodeCharacter{A71C}{{}^\downarrow}
\DeclareUnicodeCharacter{27E6}{\llbracket}
\DeclareUnicodeCharacter{27E7}{\rrbracket}

\newcommand{\unicoderightcircle}{\ensuremath{\RIGHTcircle}}
\DeclareUnicodeCharacter{25D1}{\unicoderightcircle}
\newcommand{\unicodeleftcircle}{\ensuremath{\LEFTcircle}}
\DeclareUnicodeCharacter{25D0}{\unicodeleftcircle}
\DeclareUnicodeCharacter{229B}{\circledast}

\newcommand{\unicodebbA}{\ensuremath{\mathbb{A}}}
\DeclareUnicodeCharacter{1D538}{\unicodebbA}
\newcommand{\unicodebbB}{\ensuremath{\mathbb{B}}}
\DeclareUnicodeCharacter{1D539}{\unicodebbB}
\newcommand{\unicodebbC}{\ensuremath{\mathbb{C}}}
\DeclareUnicodeCharacter{2102}{\unicodebbC}
\DeclareUnicodeCharacter{1D53B}{\ensuremath{\mathbb{D}}}
\DeclareUnicodeCharacter{2115}{\ensuremath{\mathbb{N}}}
\DeclareUnicodeCharacter{211D}{\ensuremath{\mathbb{R}}}
\DeclareUnicodeCharacter{1D543}{\ensuremath{\mathbb{L}}}
\newcommand\UnicodeBlackboardP{\ensuremath{\mathbf{P}}} \DeclareUnicodeCharacter{2119}{\UnicodeBlackboardP}
\DeclareUnicodeCharacter{211A}{\ensuremath{\mathbb{Q}}}
\DeclareUnicodeCharacter{1D544}{\ensuremath{\mathbb{M}}}
\DeclareUnicodeCharacter{1D54C}{\ensuremath{\mathbb{U}}}
\DeclareUnicodeCharacter{1D54D}{\ensuremath{\mathbf{V}}}
\DeclareUnicodeCharacter{1D54E}{\ensuremath{\mathbb{W}}}
\DeclareUnicodeCharacter{1D546}{\ensuremath{\mathbb{O}}}
\DeclareUnicodeCharacter{1D540}{\ensuremath{\mathbb{I}}}
\DeclareUnicodeCharacter{1D54A}{\ensuremath{\mathbb{S}}}
\DeclareUnicodeCharacter{1D53C}{\ensuremath{\mathbb{E}}}

\newcommand{\unicodecalS}{\ensuremath{\mathcal{S}}}
\newcommand{\unicodecalT}{\ensuremath{\mathcal{T}}}
\newcommand{\unicodecalC}{\ensuremath{\mathcal{C}}}
\newcommand{\unicodecalD}{\ensuremath{\mathcal{D}}}
\newcommand{\unicodecalX}{\ensuremath{\mathcal{X}}}
\newcommand{\unicodecalN}{\ensuremath{\mathcal{N}}}
\newcommand{\unicodecalE}{\ensuremath{\mathcal{E}}}
\DeclareUnicodeCharacter{1D4D4}{\unicodecalE}
\DeclareUnicodeCharacter{1D4D2}{\unicodecalC}
\DeclareUnicodeCharacter{1D4D3}{\unicodecalD}
\DeclareUnicodeCharacter{1D4DE}{\mathcal{O}}
\DeclareUnicodeCharacter{1D4E2}{\unicodecalS}
\DeclareUnicodeCharacter{1D4E3}{\unicodecalT}
\DeclareUnicodeCharacter{1D4E7}{\unicodecalX}
\DeclareUnicodeCharacter{1D4D0}{\ensuremath{\mathcal{A}}}
\DeclareUnicodeCharacter{1D4D1}{\ensuremath{\mathcal{B}}}
\DeclareUnicodeCharacter{1D4D6}{\ensuremath{\mathcal{G}}}
\DeclareUnicodeCharacter{1D4D7}{\ensuremath{\mathcal{H}}}
\DeclareUnicodeCharacter{1D4DB}{\ensuremath{\mathcal{L}}}
\DeclareUnicodeCharacter{1D4DC}{\ensuremath{\mathcal{M}}}
\DeclareUnicodeCharacter{1D4DD}{\unicodecalN}
\DeclareUnicodeCharacter{1D4B1}{\ensuremath{\mathcal{V}}}
\DeclareUnicodeCharacter{1D4E5}{\ensuremath{\mathcal{V}}}
\DeclareUnicodeCharacter{1D4E6}{\ensuremath{\mathcal{W}}}
\DeclareUnicodeCharacter{1D4E4}{\ensuremath{\mathcal{U}}}

\DeclareUnicodeCharacter{03B1}{\alpha}
\DeclareUnicodeCharacter{03B2}{\beta}
\DeclareUnicodeCharacter{03BC}{\mu}
\DeclareUnicodeCharacter{03B4}{\delta}
\DeclareUnicodeCharacter{03B5}{\discard}
\DeclareUnicodeCharacter{03B7}{\eta}
\DeclareUnicodeCharacter{03BB}{\lambda}
\DeclareUnicodeCharacter{03C1}{\rho}
\DeclareUnicodeCharacter{03C8}{\psi}
\DeclareUnicodeCharacter{03C4}{\tau}
\DeclareUnicodeCharacter{03A8}{\Psi}
\DeclareUnicodeCharacter{03C3}{\sigma}
\DeclareUnicodeCharacter{03C6}{\varphi}
\DeclareUnicodeCharacter{03A6}{\Phi}
\DeclareUnicodeCharacter{03A3}{\Sigma}
\DeclareUnicodeCharacter{03D5}{\phi}
\DeclareUnicodeCharacter{03B8}{\theta}
\DeclareUnicodeCharacter{03C0}{\ensuremath{\pi}}
\DeclareUnicodeCharacter{0393}{\Gamma}
\DeclareUnicodeCharacter{0394}{\Delta}
\DeclareUnicodeCharacter{03BA}{\kappa}
\DeclareUnicodeCharacter{03BD}{\cp}
\DeclareUnicodeCharacter{25A0}{\blacksquare}
\DeclareUnicodeCharacter{25AA}{\blacksquare}

\newcommand{\hirayo}{\scaleobj{0.9}{\text{\usefont{U}{min}{m}{n}\symbol{'210}}}}
\DeclareUnicodeCharacter{3088}{\hirayo}
\DeclareFontFamily{U}{min}{}
\DeclareFontShape{U}{min}{m}{n}{<-> udmj30}{}

\newcommand\UnicodeWhiteRightPointingSmallTriangle{\triangleright}
\DeclareUnicodeCharacter{25B9}{\mathbin{\UnicodeWhiteRightPointingSmallTriangle}}
\newcommand\UnicodeWhiteDownPointingSmallTriangle{\triangledown}
\DeclareUnicodeCharacter{25BF}{\mathbin{\UnicodeWhiteDownPointingSmallTriangle}}
\newcommand\UnicodeWhiteUpPointingSmallTriangle{\scalemath{1}{-1}{{}^{\triangledown}}}
\DeclareUnicodeCharacter{25B5}{\mathbin{\UnicodeWhiteUpPointingSmallTriangle}}

\DeclareUnicodeCharacter{2080}{\ensuremath{{}_0}}
\DeclareUnicodeCharacter{2081}{\ensuremath{{}_1}}
\DeclareUnicodeCharacter{2082}{\ensuremath{{}_2}}
\DeclareUnicodeCharacter{2083}{\ensuremath{{}_3}}

\DeclareUnicodeCharacter{1D62}{\ensuremath{{}_i}}
\DeclareUnicodeCharacter{2C7C}{\ensuremath{{}_j}}
\DeclareUnicodeCharacter{02B3}{\ensuremath{{}^r}}
\DeclareUnicodeCharacter{02E1}{\ensuremath{{}^\ell}}
\DeclareUnicodeCharacter{1D48}{\ensuremath{{}^d}}
\DeclareUnicodeCharacter{1D50}{\ensuremath{{}^m}}
\DeclareUnicodeCharacter{1D58}{\ensuremath{{}^u}}
\DeclareUnicodeCharacter{209A}{\ensuremath{{}_p}}
\DeclareUnicodeCharacter{2096}{\ensuremath{{}_k}}

\DeclareUnicodeCharacter{2245}{\ensuremath{\cong}}
\DeclareUnicodeCharacter{2286}{\subseteq}

\DeclareUnicodeCharacter{22C5}{\cdot}
\DeclareUnicodeCharacter{25C3}{\ensuremath{\triangleleft}}
\DeclareUnicodeCharacter{25B9}{\ensuremath{\triangleright}}

\newcommand\smallmath[2]{#1{\raisebox{\dimexpr \fontdimen 22 \textfont 2
      - \fontdimen 22 \scriptscriptfont 2 \relax}{$\scriptscriptstyle #2$}}}
\newcommand\smalloplus{\smallmath\mathbin\oplus}

\DeclareUnicodeCharacter{2295}{\smalloplus}
\DeclareUnicodeCharacter{2297}{\otimes}
\DeclareUnicodeCharacter{214B}{\parr}
\DeclareUnicodeCharacter{2298}{\oslash}
\DeclareUnicodeCharacter{25C0}{\mathbin{\blacktriangleleft}}
\DeclareUnicodeCharacter{25C1}{\mathbin{\vartriangleleft}}
\DeclareUnicodeCharacter{22B3}{\mathbin{\triangleright}}
\DeclareUnicodeCharacter{22B2}{\condcomp}
\DeclareUnicodeCharacter{FF5C}{\given}
\DeclareUnicodeCharacter{227A}{\mathbin{\prec}}
\DeclareUnicodeCharacter{227B}{\mathbin{\succ}}
\DeclareUnicodeCharacter{22A3}{\mathbin{\dashv}}
\DeclareUnicodeCharacter{219D}{\ensuremath{\leadsto}}
\DeclareUnicodeCharacter{1361}{\colon}

\DeclareUnicodeCharacter{29D1}{\mathrel{\multimapdotbothB}}
\DeclareUnicodeCharacter{29D2}{\mathrel{\multimapdotbothA}}
\DeclareUnicodeCharacter{22C4}{\mathbin{\diamond}}
\DeclareUnicodeCharacter{226B}{\mathrel{\gg}}
\DeclareUnicodeCharacter{25A1}{\Box}
\DeclareUnicodeCharacter{266F}{\sharp}

\DeclareUnicodeCharacter{2099}{_n}
\DeclareUnicodeCharacter{2098}{_m}
\DeclareUnicodeCharacter{1D5AD}{\ensuremath{\mathsf{N}}}

\DeclareUnicodeCharacter{1D4DF}{\mathcal{P}}

\newcommand\mydots{\makebox[0.6em][c]{.\hfil.\hfil.}}
\DeclareUnicodeCharacter{2026}{\mydots}
\DeclareUnicodeCharacter{226B}{\gg}
\DeclareUnicodeCharacter{2113}{\ensuremath{\mathfrak{l}}}%
\DeclareUnicodeCharacter{1D530}{\ensuremath{\mathfrak{s}}}

\usepackage{stmaryrd}
\newcommand{\unicodeRelationalComposition}{\mathbin{\fatsemi}}
\DeclareUnicodeCharacter{2A3E}{\unicodeRelationalComposition}

\DeclareUnicodeCharacter{2014}{\,---\,}
\newcommand{\sneakylink}[2]{\protect\hyperlink{#1}{#2}}
\newcommand{\removelink}[1]{\begin{NoHyper}#1\end{NoHyper}} %
\makeatletter
\newcommand{\nicelinktarget}[1]{\Hy@raisedlink{\hypertarget{#1}{}}}
\makeatother

\newcommand\linkdef[2]{\nicelinktarget{#1}{\color{black}{#2}}}
\newcommand\defining[1]{\nicelinktarget{#1}{}}

\usepackage[xcolor,no patch,hyperref,quotation,composition]{knowledge}
\knowledgeconfigure{notion}
\knowledgestyle{notion}{color=nordnight}
\knowledgestyle{intro notion}{emphasize,color=nordnight}

\knowledge{notion}
| Markov category
| Markov categories

\knowledge{notion}
| distributive Markov category
| distributive Markov categories
| Distributive Markov category
| Distributive Markov categories

\knowledge{notion}
| uniformity
| Uniformity
| uniformity axiom
| Uniformity axiom
| uniformly traced distributive multicategory
| uniformly traced distributive multicategories
| Uniformly traced distributive multicategory
| Uniformly traced distributive multicategories

\knowledge{notion}
| comparator
| comparators
| compare morphism
| compare morphisms
| Comparator
| Comparators
| Compare morphism
| Compare morphisms
| compare 
| Compare

\knowledge{notion}
| term
| terms
| Term
| Terms

\knowledge{notion}
| interchange
| Interchange
| interchange axiom
| Interchange axiom

\knowledge{notion}
| correctness triple
| Correctness triple
| correctness triples
| Correctness triples

\knowledge{notion}
| posetal imperative category
| Posetal imperative category
| posetal imperative categories
| Posetal imperative categories

\knowledge{notion}
| posetal imperative multicategory
| Posetal imperative multicategory
| posetal imperative multicategories
| Posetal imperative multicategories

\knowledge{notion}
| guard combinators
| Guard combinators
| guard combinator
| Guard combinator

\knowledge{notion}
| command combinators
| Command combinators
| command combinator
| Command combinator

\knowledge{notion}
| monoidal category
| monoidal categories
| Monoidal category
| Monoidal categories

\knowledge{notion}
| sesquifunctor
| Sesquifunctor
| sesquifunctors
| Sesquifunctors

\knowledge{notion}
| premonoidal category
| Premonoidal category
| premonoidal categories
| Premonoidal categories
| premonoidality
| Premonoidality
| premonoidal
| Premonoidal

\knowledge{notion}
| central
| Central
| central morphism
| Central morphism
| central morphisms
| Central morphisms

\knowledge{notion}
| Symmetric premonoidal category
| Symmetric premonoidal categories
| symmetric premonoidal category
| symmetric premonoidal categories

\knowledge{notion}
| deterministic
| Deterministic
| deterministic morphism
| Deterministic morphism
| deterministic morphisms
| Deterministic morphisms

\knowledge{notion}
| total
| total morphism
| total morphisms
| Total
| Total morphism
| Total morphisms
| totality
| Totality

\knowledge{notion}
| cocartesian multicategory
| Cocartesian multicategory
| cocartesian multicategories
| Cocartesian multicategories

\knowledge{notion}
| predistributive multicategory
| Predistributive multicategory
| predistributive multicategories
| Predistributive multicategories

\knowledge{notion}
| posetal distributive copy-discard multicategory
| Posetal distributive copy-discard multicategory
| posetal distributive copy-discard multicategories
| Posetal distributive copy-discard multicategories

\knowledge{notion}
| exact observation
| exact observations
| Exact observation
| Exact observations
| process with exact observations
| processes with exact observations
| category of processes with exact observations
\newcommand{\exactObservations}{\kl{exact observations}}
\newcommand{\ExactObservations}{\kl{Exact observations}}
\newcommand{\obsv}[1]{{#1}^{\kl[exact observation]{\circ}}}
\newcommand{\pur}[1]{{#1}^{\kl[exact observation]{•}}}

\knowledge{notion}
| index
| indices
| Index
| Indices

\knowledge{notion}
| state
| states
| State
| States

\knowledge{notion}
| state combinator
| state combinators
| State combinator
| State combinators

\knowledge{notion}
| multicategory
| Multicategory
| Multicategories
| multicategories

\knowledge{notion}
| distributive multicategory
| distributive multicategories
| Distributive multicategory
| Distributive multicategories

\knowledge{notion}
| substitution
| Substitution
| substitutions
| Substitutions
| substituting
| Substituting

\knowledge{notion}
| guard
| guards
| Guard
| Guards

\knowledge{notion}
| predicate
| predicates
| Predicates
| Predicate

\knowledge{notion}
| sharp predicate
| Sharp predicate
| sharp predicates
| Sharp predicates
| sharp
| Sharp

\knowledge{notion}
| predicate combinators
| Predicate combinators

\knowledge{notion}
| command
| commands
| Commands
| Command

\knowledge{notion}
| fresh
| Fresh

\knowledge{notion}
| alpha-equivalent
| alpha-equivalence
| Alpha-equivalent
| Alpha-equivalence

\knowledge{notion}
| variable
| variables
| Variable
| Variables

\knowledge{notion}
| traced distributive copy-discard multicategory
| Traced distributive copy-discard multicategory
| traced distributive copy-discard multicategories
| Traced distributive copy-discard multicategories

\knowledge{notion}
| traced distributive multicategory
| Traced distributive multicategory
| traced distributive multicategories
| Traced distributive multicategories

\knowledge{notion}
| traced predistributive copy-discard multicategory
| Traced predistributive copy-discard multicategory
| traced predistributive copy-discard multicategories
| Traced predistributive copy-discard multicategories

\knowledge{notion}
| predistributive copy-discard multicategory
| Predistributive copy-discard multicategory
| predistributive copy-discard multicategories
| Predistributive copy-discard multicategories

\knowledge{notion}
| distributive copy-discard multicategory
| Distributive copy-discard multicategory
| distributive copy-discard multicategories
| Distributive copy-discard multicategories

\knowledge{notion}
| variable substitution
| variable substitutions
| Variable substitution
| Variable substitutions

\knowledge{notion}
| posetal distributive signature
| posetal distributive signatures
| Posetal distributive signature
| Posetal distributive signatures

\knowledge{notion}
| posetal uniformity
| Posetal uniformity

\knowledge{notion}
| posetal distributive copy-discard category
| posetal distributive copy-discard categories
| Posetal distributive copy-discard category
| Posetal distributive copy-discard categories

\knowledge{notion}
| posetal uniform trace
| posetal uniform traces
| Posetal uniform traces
| Posetal uniform trace
| posetal uniform traced monoidal category
| Posetal uniform traced monoidal category
| posetal uniform traced monoidal categories
| Posetal uniform traced monoidal categories

\knowledge{notion}
| affine copy-discard category
| Affine copy-discard category
| affine copy-discard categories
| Affine copy-discard categories

\knowledge{notion}
| label substitution
| label substitutions
| Label substitution
| Label substitutions

\knowledge{notion}
| conditional
| conditionals
| Conditionals
| Conditional
\newcommand{\conditional}{\kl{conditional}}
\newcommand{\conditionals}{\kl{conditionals}}
\newcommand{\Conditional}{\kl{Conditional}}
\newcommand{\Conditionals}{\kl{Conditionals}}

\knowledge{notion}
| conditional composition
| Conditional composition

\knowledge{notion}
| deterministic domain
| deterministic domains
| quasi-total morphism
| quasi-total morphisms
| quasi-total
| Quasi-total morphism
| Quasi-total morphisms
| Quasi-total

\knowledge{notion}
| anchor
| anchors
| Anchor
| Anchors
| label
| Label
| Labels
| labels

\knowledge{notion}
| context
| Context
| contexts
| Contexts

\knowledge{notion}
| copy-discard category
| Copy-discard category
| copy-discard categories
| Copy-discard categories
| copy-discard
| Copy-discard
\newcommand{\copyDiscardCategory}{\kl{copy-discard category}}
\newcommand{\copyDiscardCategories}{\kl{copy-discard categories}}

\newcommand{\CopyDiscardCategories}{\kl{Copy-discard categories}}

\knowledge{notion}
| copy-discard monoidal category
| Copy-discard monoidal category
| copy-discard monoidal categories
| Copy-discard monoidal categories

\knowledge{notion}
| copy-discard premonoidal category
| Copy-discard premonoidal category
| copy-discard premonoidal categories
| Copy-discard premonoidal categories
| premonoidal copy-discard category
| Premonoidal copy-discard category
| premonoidal copy-discard categories
| Premonoidal copy-discard categories

\knowledge{notion}
| commutative imperative category
| Commutative imperative category
| commutative imperative categories
| Commutative imperative categories

\knowledge{notion}
| imperative category
| Imperative category
| imperative categories
| Imperative categories

\knowledge{notion}
| imperative multicategory
| Imperative multicategory
| imperative multicategories
| Imperative multicategories

\knowledge{notion}
| bimonoidally strict
| Bimonoidally strict
| strict
| Strict

\knowledge{notion}
| distributive signature
| distributive signatures
| Distributive signature
| Distributive signatures

\knowledge{notion}
| distributive category
| distributive categories
| Distributive category
| Distributive categories
| distributive copy-discard category
| distributive copy-discard categories
| Distributive copy-discard category
| Distributive copy-discard categories

\knowledge{notion}
| basic type
| basic types
| Basic type
| Basic types

\knowledge{notion}
| generator
| generators
| Generator
| Generators

\knowledge{notion}
| effectful triple
| effectful triples
| Effectful triple
| Effectful triples

\knowledge{notion}
| double signature
| double signatures
| Double signature
| Double signatures

\knowledge{notion}
| forgetful double signature
| forgetful double signatures
| Forgetful double signature
| Forgetful double signatures

\knowledge{notion}
| double signature morphism
| double signature morphisms
| Double signature morphism
| Double signature morphisms

\knowledge{notion}
| pinwheel double category
| pinwheel double categories
| Pinwheel double category
| Pinwheel double categories
| double category with pinwheels
| double categories with pinwheels
| Double category with pinwheels
| Double categories with pinwheels

\knowledge{notion}
| bigraph
| Bigraph
| bigraphs
| Bigraphs
| 2-graph
| 2-graphs
| 2-Graph
| 2-Graphs

\knowledge{notion}
| 2-graph morphism
| 2-graph morphisms

\knowledge{notion}
| 2-graph category
| 2-graph categories

\knowledge{notion}
| double category
| double categories
| Double category
| Double categories

\knowledge{notion}
| weighted polygraph
| weighted polygraphs
| Weighted polygraph
| Weighted polygraphs
| timed polygraph
| timed polygraphs
| Timed polygraph
| Timed polygraphs

\knowledge{notion}
| morphism of timed polygraphs
| Morphism of timed polygraphs
| timed polygraph morphism
| Timed polygraph morphism

\knowledge{notion}
| tilted bicategory signature
| tilted bicategory signatures
| Tilted bicategory signature
| Tilted bicategory signatures
| tilted bigraph
| Tilted bigraph
| tilted bigraphs
| Tilted bigraphs
| tilted 2-graph
| Tilted 2-graph
| tilted 2-graphs
| Tilted 2-graphs

\knowledge{notion}
| path
| paths
| Path
| Paths
| composable path
| composable paths

\knowledge{notion}
| signature
| signatures
| Signature
| Signatures

\knowledge{notion}
| congruence
| congruences
| Congruence
| Congruences

\knowledge{notion}
| symmetry
| symmetries
| Symmetry
| Symmetries

\knowledge{notion}
| profunctor
| profunctors
| Profunctor
| Profunctors

\knowledge{notion}
| left inclusion
| Left inclusion

\knowledge{notion}
| right inclusion
| Right inclusion

\knowledge{notion}
| left whiskering
| Left whiskering

\knowledge{notion}
| right whiskering
| Right whiskering

\knowledge{notion}
| inclusion and whiskering
| Inclusion and whiskering
| whiskering interacts with inclusion
| Whiskering interacts with inclusion

\knowledge{notion}
| whiskering
| Whiskering

\knowledge{notion}
| rewiring preserves whiskering
| Rewiring preserves whiskering
| whiskering preserves rewiring
| Whiskering preserves rewiring

\knowledge{notion}
| rewiring preserves interchange
| Rewiring preserves interchange

\knowledge{notion}
| rewiring preserves composition
| Rewiring preserves composition

\knowledge{notion}
| rewiring
| Rewiring 
| definition of rewiring

\knowledge{rewiring is an action}[]{notion}

\knowledge{notion}
| composition
| Composition
| term composition
| Term composition

\knowledge{notion}
| term composition is associative
| Term composition is associative

\knowledge{notion}
| term composition is unital
| Term composition is unital

\knowledge{notion}
| tensoring
| Tensoring

\knowledge{notion}
| arrow notation preterm
| Arrow notation preterm
| arrow notation preterms
| Arrow notation preterms

\knowledge{notion}
| arrow notation term
| Arrow notation term
| arrow notation terms
| Arrow notation terms

\knowledge{notion}
| observe-arrow notation
| Observe-arrow notation
| observe-arrow notation term
| Observe-arrow notation term
| observe-arrow notation terms
| Observe-arrow notation terms

\knowledge{notion}
| subdistribution
| Subdistribution
| subdistributions
| Subdistributions
| subdistributional

\knowledge{notion}
| rescaling
| Rescaling

\knowledge{notion}
| normalisation
| normalisations
| Normalisation
| Normalisations
| normalization
| normalizations
| Normalization
| Normalizations
| normalised
| Normalised
| normalise
| normalization

\knowledge{notion}
| observe-arrow notation copy-discard-compare category of terms

\knowledge{notion}
| category of arrow notation terms

\knowledge{notion}
| copy-discard-compare categories
| copy-discard-compare category
| Copy-discard-compare categories
| Copy-discard-compare category

\knowledge{notion}
| Kleisli extension
| Kleisli extensions

\knowledge{notion}
| Kleisli category
| Kleisli categories

\knowledge{notion}
| partially additive
| partially additive monad
| Partially additive
| Partially additive monad

\knowledge{notion}
| sets
| category of sets

\knowledge{notion}
| powerset
| powerset monad

\newcommand{\nePowerset}{\mathcal{P}_{ne}}

\knowledge{notion}
| category of relations
| Category of relations

\knowledge{notion}
| maybe
| maybe monad

\knowledge{notion}
| marginal
| marginals
| Marginal
| Marginals

\knowledge{notion}
| partial Markov category
| partial Markov categories
| Partial Markov category
| Partial Markov categories
\newcommand{\partialMarkovCategory}{\hyperlink{linkPartialMarkov}{partial Markov category}}
\newcommand{\partialMarkovCategories}{\hyperlink{linkPartialMarkov}{partial Markov categories}}
\newcommand{\PartialMarkovCategory}{\hyperlink{linkPartialMarkov}{Partial Markov category}}
\newcommand{\PartialMarkovCategories}{\hyperlink{linkPartialMarkov}{Partial Markov categories}}

\knowledge{notion}
| category of partial functions
| Category of partial functions

\knowledge{notion}
| discrete subdistribution
| discrete subdistributions
| discrete subdistribution monad

\knowledge{notion}
| discrete partial Markov category
| discrete partial Markov categories
| Discrete partial Markov category
| Discrete partial Markov categories

\knowledge{notion}
| category of discrete stochastic channels
| Category of discrete stochastic channels

\knowledge{notion}
| category of finite stochastic channels
| Category of finite stochastic channels
| finstoch
\newcommand{\FinStoch}{\kl[finstoch]{\mathsf{FinStoch}}}

\knowledge{notion}
| support

\knowledge{notion}
| Dirac distribution

\knowledge{notion}
| measurable subdistribution
| measurable subdistributions
| measurable subdistribution monad

\knowledge{notion}
| standard Borel spaces
| Standard Borel spaces

\knowledge{notion}
| quantale-valued sets

\knowledge{notion}
| assertion-correctness triple
| assertion-correctness triples
| Assertion-correctness triple
| Assertion-correctness triples

\knowledge{notion}
| state-correctness triple
| state-correctness triples
| State-correctness triple
| State-correctness triples

\knowledge{notion}
| predicate-correctness triple
| predicate-correctness triples
| Predicate-correctness triple
| Predicate-correctness triples

\knowledge{notion}
| assertion-incorrectness triple
| assertion-incorrectness triples
| Assertion-incorrectness triple
| Assertion-incorrectness triples

\knowledge{notion}
| state-incorrectness triple
| state-incorrectness triples
| State-incorrectness triple
| State-incorrectness triples

\knowledge{notion}
| predicate-incorrectness triple
| predicate-incorrectness triples
| Predicate-incorrectness triple
| Predicate-incorrectness triples

\knowledge{notion}
| relational assertion-correctness triple
| relational assertion-correctness triples
| Relational assertion-correctness triple
| Relational assertion-correctness triples

\knowledge{notion}
| relational state-correctness triple
| relational state-correctness triples
| Relational state-correctness triple
| Relational state-correctness triples

\knowledge{notion}
| relational predicate-correctness triple
| relational predicate-correctness triples
| Relational predicate-correctness triple
| Relational predicate-correctness triples

\knowledge{notion}
| relational assertion-incorrectness triple
| relational assertion-incorrectness triples
| Relational assertion-incorrectness triple
| Relational assertion-incorrectness triples

\knowledge{notion}
| relational state-incorrectness triple
| relational state-incorrectness triples
| Relational state-incorrectness triple
| Relational state-incorrectness triples

\knowledge{notion}
| relational predicate-incorrectness triple
| relational predicate-incorrectness triples
| Relational predicate-incorrectness triple
| Relational predicate-incorrectness triples

\knowledge{notion}
| branch combinator

\knowledge{notion}
| coupling synchronization

\knowledge{notion}
| domain of definition
| Domain of definition

\newcommand{\domain}{\kl[domain of definition]{domain}}

\knowledge{notion}
| deterministic-domain
| Deterministic-domain
| deterministic-domain morphism
| Deterministic-domain morphism

\knowledge{notion}
| affine maps with Gaussian noise
| Gauss

\newcommand{\Gauss}{\kl[Gauss]{\cat{Gauss}}} %

\knowledge{notion}
| affine maps with partial Gaussian noise
| subGauss

\newcommand{\parGauss}{\kl[subGauss]{\cat{Gauss}_{\leq 1}}} %

\knowledge{notion}
| cartesian restriction category
| Cartesian restriction category
| cartesian restriction categories
| Cartesian restriction categories
| relevant copy-discard category
| relevant copy-discard categories
| Relevant copy-discard category
| Relevant copy-discard categories

\knowledge{notion}
| discrete cartesian restriction category
| Discrete cartesian restriction category
| discrete cartesian restriction categories
| Discrete cartesian restriction categories

\knowledge{notion}
| Par
| category of partial functions

\newcommand{\Par}{\kl[Par]{\cat{Par}}} %
\NewDocumentEnvironment{reviewer}{+b}{\begin{quote}\itshape #1}{\end{quote}}
\NewDocumentEnvironment{response}{+b}{\noindent #1}{\medskip}

\newcommand{\bayesianInversion}{\hyperlink{linkbayesinv}{Bayesian inversion}}
\newcommand{\BayesianInversion}{\hyperlink{linkbayesinv}{Bayesian inversion}}

\newcommand{\BayesianInversions}{\hyperlink{linkbayesinv}{Bayesian inversions}}

\newcommand{\Stoch}{\hyperlink{linkStoch}{\ensuremath{\mathsf{Stoch}}}}
\newcommand{\SubStoch}{\hyperlink{linkSubStoch}{\ensuremath{\mathsf{SubStoch}}}}
\newcommand{\subStoch}{\hyperlink{linkSubStoch}{\ensuremath{\mathsf{subStoch}}}}
\newcommand{\BorelStoch}{\hyperlink{linkborelstoch}{\ensuremath{\cat{BorelStoch}}}} %
\newcommand{\subBorelStoch}{\hyperlink{linksubborelstoch}{\ensuremath{\cat{BorelSubstoch}}}}
\newcommand{\Giry}{\hyperlink{linkgiry}{\ensuremath{\fun{G}}}} %
\newcommand{\subGiry}{\hyperlink{linksubgiry}{\begin{NoHyper}{\Giry}\end{NoHyper}_{\leq 1}}} %
\newcommand{\BMeas}{\hyperlink{linkBmeas}{\cat{Borel}}}
\newcommand{\ket}[1]{|#1\rangle} %

\newcommand{\MarkovCategory}{\kl{Markov category}}
\newcommand{\MarkovCategories}{\kl{Markov categories}}
\newcommand{\discretePartialMarkovCategory}{\kl{discrete partial Markov category}}
\newcommand{\discretePartialMarkovCategories}{\kl{discrete partial Markov categories}}
\newcommand{\DiscretePartialMarkovCategory}{\kl{Discrete partial Markov category}}
\newcommand{\DiscretePartialMarkovCategories}{\kl{Discrete partial Markov categories}}

\newcommand{\comparators}{\kl{compare morphisms}}

\newcommand{\normalisation}{\hyperlink{linkNormalisation}{normalisation}}
\newcommand{\normalisations}{\hyperlink{linkNormalisation}{normalisations}}
\newcommand{\Normalisation}{\hyperlink{linkNormalisation}{Normalisation}}

\newcommand{\bayesinv}[2]{\smash{{#1}_{\sneakylink{linkbayesinv}{\dagger}}({#2})}} %

\newcommand{\Pearls}{\hyperlink{linkPearls}{Pearl's}}
\newcommand{\PearlsUpdate}{\hyperlink{linkPearls}{Pearl's update}}
\newcommand{\Jeffreys}{\hyperlink{linkJeffreys}{Jeffrey's}}
\newcommand{\JeffreysUpdate}{\hyperlink{linkJeffreys}{Jeffrey's update}}

\newcommand{\almostSurely}{\hyperlink{linkAlmostSureEquality}{almost-surely}}
\newcommand{\almostSurelyEqual}{\hyperlink{linkAlmostSureEquality}{almost-surely equal}}
\newcommand{\almostSureEquality}{\hyperlink{linkAlmostSureEquality}{almost-sure equality}}

\newcommand{\total}{\hyperlink{linkTotalMorphism}{total}}
\newcommand{\deterministic}{\hyperlink{linkDeterministicMorphism}{deterministic}}

\newcommand{\Deterministic}{\hyperlink{linkDeterministicMorphism}{Deterministic}}

\newcommand{\subdistributions}{\hyperlink{linkSubdistribution}{subdistributions}}

\newcommand{\subprobabilityMeasure}{\hyperlink{linkSubprobabilityMeasure}{subprobability measure}}

\newcommand{\deterministicdomain}{\hyperlink{linkDomainOfDefinition}{deterministic-domain}}

\newcommand{\marginal}{\hyperlink{linkMarginal}{marginal}}
\newcommand{\marginals}{\hyperlink{linkMarginal}{marginals}}

\newcommand{\MarginalComposition}{\hyperlink{linkMarginalComposition}{Conditional composition}}
\newcommand{\conditionalComposition}{\hyperlink{linkMarginalComposition}{conditional composition}}

\newcommand{\BayesTheorem}{\hyperlink{linkBayesTheorem}{Bayes' theorem}}

\newcommand{\exact}{\ensuremath{\mathsf{exact}}}

\newcommand{\distributiveMarkovCategory}{\kl{distributive Markov category}}

\newcommand{\distributiveMarkovCategories}{\kl{distributive Markov categories}}

\newcommand{\cond}{\sneakylink{linkcond}{\operatorname{\pmb{c}}}}

\newcommand{\norm}{\operatorname{norm}}

\newcommand{\pdf}[1]{#1^{\mathit{pdf}}}
\usepackage{subfigure}
\usepackage{scalerel}
\usepackage{cancel}

\title{Partial Markov Categories}
\author[E. Di Lavore]{Elena Di Lavore\lmcsorcid{0000-0002-7783-5079}}[a,b,c]
\author[M. Román]{Mario Román\lmcsorcid{0000-0003-3158-1226}}[b,c]
\author[P. Sobociński]{Paweł Sobociński\lmcsorcid{0000-0003-3158-1226}}[c]
\address{Department of Computer Science, University of Pisa, Italy}
\address{Department of Computer Science, University of Oxford, United Kingdom}
\address{Department of Software Science, Tallinn University of Technology, Estonia}
\thanks{Elena Di Lavore, Mario Román, and Paweł Sobociński were supported by the
Safeguarded AI programme of the Advanced Research and Invention Agency and the Estonian Research Council grant PRG3215. Paweł Sobociński was additionally supported by the Estonian Research Council grant TEM-TA5 and the Estonian Center of Excellence in Artificial Intelligence (EXAI), as well as by the European Union under Grant Agreement No. 101087529. 
This article is based on work from
COST Action EuroProofNet, CA20111 supported by COST (European Cooperation in
Science and Technology). Authorship statement: no part of this manuscript was conceived, written, or edited by generative artificial intelligence.}

\begin{document}
\begin{abstract}
  We introduce partial Markov categories: a synthetic framework for
  probabilistic inference blending the work of Cho and Jacobs, Fritz, and
  Golubtsov on Markov categories with the work of Cockett and Lack on cartesian
  restriction categories. We describe exact observations, Bayes' theorem,
  normalisation, and both Pearl's and Jeffrey's updates in abstract
  categorical terms.
\end{abstract}
\maketitle

\setcounter{tocdepth}{1}
\tableofcontents

\section{Introduction}%
\label{sec:introduction}

\emph{Probabilistic inference} is the problem of computing the probability of an outcome, given a statistical model and specific observations on its variables \cite{jaynes2003probability}. Probabilistic syntax is often ambiguous for this task: all probabilistic channels arising from a probability measure are written as ``$P$'', while their type and underlying model remain implicit. This ambiguity is the source of the many apparent paradoxes of probabilistic inference (e.g., the \emph{Monty Hall problem} \cite{selvin75:montyhall}, or \emph{Newcomb's paradox} \cite{nozick_1969}): these problems have different solutions depending on how one interprets implicit information.

As a result, \emph{probabilistic programming}—the declarative programming paradigm for probabilistic inference—does not employ the more ambiguous traditional syntax; instead, it adds probabilistic primitives to the already vast syntax of programming constructs~\cite{GMRBT12:church,toplinWood16:anglican,narayanan16:hakaru,stan17}. Probabilistic programming semantics—both operational \cite{P03:acpl,park08,dLZ12:poslc,borgstrom16,culpeppercobb18} and denotational \cite{kozen81,panangaden1999,ehrhard2017measurable,ept11:coherencespaces,staton2017commutative,dash23}—are forced to tackle this complexity.

Could we devise more specific, minimalistic syntaxes for probabilistic inference? A good syntax should allow us to formally and explicitly represent \emph{(i)} derived channels, such as marginals and conditionals; \emph{(ii)} specific observations on variables; and \emph{(iii)} and the shape of the statistical models.

We introduce \partialMarkovCategories{}: a denotational and categorical framework for probabilistic processes constrained by observations. We propose \partialMarkovCategories{} as a minimalistic syntax for probabilistic inference: it expresses the types and the context of each process, and it forces the underlying modelling assumptions—both about probability and partiality—to be made explicit.
\PartialMarkovCategories{} blend \MarkovCategories{} (an algebra of probabilistic processes \cite{cho_2019,fritz_2020}) with \kl{cartesian restriction categories} (an algebra of partiality \cite{cockett2007restriction}); and inherit their string diagrammatic syntax.

\subsection{String diagrams}
\label{sec:string-diagrams}

\emph{String diagrams}~\cite{joyal91} offer an intuitive, rigorous, and often succinct syntax that illustrates the causal and temporal dependencies within a process. In this notation, \emph{processes} (morphisms of a symmetric monoidal category) are represented by nodes and \emph{resources} (objects of a symmetric monoidal category) are represented by wires. Additionally, we assume the existence of ``copy'' and ``discard'' nodes, represented with black dots and satisfying the axioms in \Cref{intro:diagram-copy-discard}. Categories endowed with this structure are called \kl{copy-discard categories}.
\begin{figure}[h]

\tikzset{every picture/.style={line width=0.75pt}} %

\begin{tikzpicture}[x=0.75pt,y=0.75pt,yscale=-1,xscale=1]
\draw    (115,60) -- (115,45) ;
\draw  [draw opacity=0] (160,25) -- (190,25) -- (190,50) -- (160,50) -- cycle ;
\draw  [draw opacity=0] (235,25) -- (265,25) -- (265,50) -- (235,50) -- cycle ;
\draw    (290,60) -- (290,50) ;
\draw  [draw opacity=0] (325,25) -- (355,25) -- (355,50) -- (325,50) -- cycle ;
\draw  [draw opacity=0] (370,25) -- (400,25) -- (400,50) -- (370,50) -- cycle ;
\draw    (130,30) -- (130,20) ;
\draw    (130,30) .. controls (120.17,34.43) and (115,34.47) .. (115,45) ;
\draw  [fill={rgb, 255:red, 0; green, 0; blue, 0 }  ,fill opacity=1 ] (127,30) .. controls (127,28.34) and (128.34,27) .. (130,27) .. controls (131.66,27) and (133,28.34) .. (133,30) .. controls (133,31.66) and (131.66,33) .. (130,33) .. controls (128.34,33) and (127,31.66) .. (127,30) -- cycle ;
\draw    (130,30) .. controls (139.88,34.47) and (145,34.72) .. (145,45) ;
\draw    (145,45) .. controls (135.17,49.43) and (130,49.47) .. (130,60) ;
\draw  [fill={rgb, 255:red, 0; green, 0; blue, 0 }  ,fill opacity=1 ] (142,45) .. controls (142,43.34) and (143.34,42) .. (145,42) .. controls (146.66,42) and (148,43.34) .. (148,45) .. controls (148,46.66) and (146.66,48) .. (145,48) .. controls (143.34,48) and (142,46.66) .. (142,45) -- cycle ;
\draw    (145,45) .. controls (154.88,49.47) and (160,49.72) .. (160,60) ;
\draw    (220,30) -- (220,20) ;
\draw    (220,30) .. controls (210.17,34.43) and (205,34.47) .. (205,45) ;
\draw  [fill={rgb, 255:red, 0; green, 0; blue, 0 }  ,fill opacity=1 ] (217,30) .. controls (217,28.34) and (218.34,27) .. (220,27) .. controls (221.66,27) and (223,28.34) .. (223,30) .. controls (223,31.66) and (221.66,33) .. (220,33) .. controls (218.34,33) and (217,31.66) .. (217,30) -- cycle ;
\draw    (220,30) .. controls (229.88,34.47) and (235,34.72) .. (235,45) ;
\draw    (235,60) -- (235,45) ;
\draw    (205,45) .. controls (195.17,49.43) and (190,49.47) .. (190,60) ;
\draw  [fill={rgb, 255:red, 0; green, 0; blue, 0 }  ,fill opacity=1 ] (202,45) .. controls (202,43.34) and (203.34,42) .. (205,42) .. controls (206.66,42) and (208,43.34) .. (208,45) .. controls (208,46.66) and (206.66,48) .. (205,48) .. controls (203.34,48) and (202,46.66) .. (202,45) -- cycle ;
\draw    (205,45) .. controls (214.88,49.47) and (220,49.72) .. (220,60) ;
\draw    (305,35) -- (305,20) ;
\draw    (305,35) .. controls (295.17,39.43) and (290,39.47) .. (290,50) ;
\draw  [fill={rgb, 255:red, 0; green, 0; blue, 0 }  ,fill opacity=1 ] (302,35) .. controls (302,33.34) and (303.34,32) .. (305,32) .. controls (306.66,32) and (308,33.34) .. (308,35) .. controls (308,36.66) and (306.66,38) .. (305,38) .. controls (303.34,38) and (302,36.66) .. (302,35) -- cycle ;
\draw    (305,35) .. controls (314.88,39.47) and (320,39.72) .. (320,50) ;
\draw  [fill={rgb, 255:red, 0; green, 0; blue, 0 }  ,fill opacity=1 ] (317,50) .. controls (317,48.34) and (318.34,47) .. (320,47) .. controls (321.66,47) and (323,48.34) .. (323,50) .. controls (323,51.66) and (321.66,53) .. (320,53) .. controls (318.34,53) and (317,51.66) .. (317,50) -- cycle ;
\draw    (365,60) -- (365,20) ;
\draw  [draw opacity=0] (460,25) -- (490,25) -- (490,50) -- (460,50) -- cycle ;
\draw  [draw opacity=0] (520,25) -- (550,25) -- (550,50) -- (520,50) -- cycle ;
\draw    (440,30) -- (440,20) ;
\draw    (440,30) .. controls (430.17,34.43) and (425.29,34.31) .. (425,40) ;
\draw  [fill={rgb, 255:red, 0; green, 0; blue, 0 }  ,fill opacity=1 ] (437,30) .. controls (437,28.34) and (438.34,27) .. (440,27) .. controls (441.66,27) and (443,28.34) .. (443,30) .. controls (443,31.66) and (441.66,33) .. (440,33) .. controls (438.34,33) and (437,31.66) .. (437,30) -- cycle ;
\draw    (440,30) .. controls (449.88,34.47) and (454.86,34.6) .. (455,40) ;
\draw    (425,40) .. controls (425.5,50.43) and (454.83,49.77) .. (455,60) ;
\draw    (455,40) .. controls (455.5,50.43) and (424.83,49.77) .. (425,60) ;
\draw    (505,35) -- (505,20) ;
\draw    (505,35) .. controls (495.17,39.43) and (490,39.47) .. (490,50) ;
\draw  [fill={rgb, 255:red, 0; green, 0; blue, 0 }  ,fill opacity=1 ] (502,35) .. controls (502,33.34) and (503.34,32) .. (505,32) .. controls (506.66,32) and (508,33.34) .. (508,35) .. controls (508,36.66) and (506.66,38) .. (505,38) .. controls (503.34,38) and (502,36.66) .. (502,35) -- cycle ;
\draw    (505,35) .. controls (514.88,39.47) and (520,39.72) .. (520,50) ;
\draw    (490,60) -- (490,50) ;
\draw    (520,60) -- (520,50) ;

\draw (175,37.5) node    {$=$};
\draw (250,37.5) node    {$;$};
\draw (340,37.5) node    {$=$};
\draw (385,37.5) node    {$;$};
\draw (475,37.5) node    {$=$};
\draw (535,37.5) node    {$;$};
\draw (130,16.6) node [anchor=south] [inner sep=0.75pt]  [font=\tiny]  {$X$};
\draw (220,16.6) node [anchor=south] [inner sep=0.75pt]  [font=\tiny]  {$X$};
\draw (305,16.6) node [anchor=south] [inner sep=0.75pt]  [font=\tiny]  {$X$};
\draw (365,16.6) node [anchor=south] [inner sep=0.75pt]  [font=\tiny]  {$X$};
\draw (440,16.6) node [anchor=south] [inner sep=0.75pt]  [font=\tiny]  {$X$};
\draw (505,16.6) node [anchor=south] [inner sep=0.75pt]  [font=\tiny]  {$X$};
\draw (520,63.4) node [anchor=north] [inner sep=0.75pt]  [font=\tiny]  {$X$};
\draw (490,63.4) node [anchor=north] [inner sep=0.75pt]  [font=\tiny]  {$X$};
\draw (455,63.4) node [anchor=north] [inner sep=0.75pt]  [font=\tiny]  {$X$};
\draw (425,63.4) node [anchor=north] [inner sep=0.75pt]  [font=\tiny]  {$X$};
\draw (365,63.4) node [anchor=north] [inner sep=0.75pt]  [font=\tiny]  {$X$};
\draw (290,63.4) node [anchor=north] [inner sep=0.75pt]  [font=\tiny]  {$X$};
\draw (235,63.4) node [anchor=north] [inner sep=0.75pt]  [font=\tiny]  {$X$};
\draw (220,63.4) node [anchor=north] [inner sep=0.75pt]  [font=\tiny]  {$X$};
\draw (190,63.4) node [anchor=north] [inner sep=0.75pt]  [font=\tiny]  {$X$};
\draw (115,63.4) node [anchor=north] [inner sep=0.75pt]  [font=\tiny]  {$X$};
\draw (130,63.4) node [anchor=north] [inner sep=0.75pt]  [font=\tiny]  {$X$};
\draw (160,63.4) node [anchor=north] [inner sep=0.75pt]  [font=\tiny]  {$X$};

\end{tikzpicture}
   \caption{Axioms for copy and discard: associativity, unitality, and commutativity.\label{intro:diagram-copy-discard}}
\end{figure}

For instance, the equations in \Cref{fig:copy-discard-markov}—which we read from top to bottom—state that some process, $f ፡ X → Y$, can be copied (left) and that some  procees, $g ፡ X → Y$, can be discarded (right).
The leftmost equation fails for stochastic processes (e.g.~throwing a coin twice is not the same as throwing it once and copying the result) but holds
for partial functions: morphisms satisfying this equation are called \kl{deterministic} (\Cref{def:total_deterministic}). Conversely, the rightmost equation fails for partial functions (e.g.~partiality is not discardable) but holds for probabilistic processes: morphisms satisfying this equation are called \kl{total} (\Cref{def:total_deterministic}). Thus, we need different axioms for probabilistic and partial processes: these are given by \MarkovCategories{} and \kl{cartesian restriction categories}, respectively.

\begin{figure}[h]

\tikzset{every picture/.style={line width=0.75pt}} %

\begin{tikzpicture}[x=0.75pt,y=0.75pt,yscale=-1,xscale=1]
\draw   (140,30.5) -- (160,30.5) -- (160,45.5) -- (140,45.5) -- cycle ;
\draw    (150,30) -- (150,20) ;
\draw    (150,55) -- (150,45) ;
\draw    (150,55) .. controls (140.17,59.43) and (135,59.47) .. (135,70) ;
\draw  [fill={rgb, 255:red, 0; green, 0; blue, 0 }  ,fill opacity=1 ] (147,55) .. controls (147,53.34) and (148.34,52) .. (150,52) .. controls (151.66,52) and (153,53.34) .. (153,55) .. controls (153,56.66) and (151.66,58) .. (150,58) .. controls (148.34,58) and (147,56.66) .. (147,55) -- cycle ;
\draw    (150,55) .. controls (159.88,59.47) and (165,59.72) .. (165,70) ;
\draw   (195,45) -- (215,45) -- (215,60) -- (195,60) -- cycle ;
\draw   (225,45) -- (245,45) -- (245,60) -- (225,60) -- cycle ;
\draw    (220,30) -- (220,20) ;
\draw    (220,30) .. controls (210.17,34.43) and (205,34.47) .. (205,45) ;
\draw  [fill={rgb, 255:red, 0; green, 0; blue, 0 }  ,fill opacity=1 ] (217,30) .. controls (217,28.34) and (218.34,27) .. (220,27) .. controls (221.66,27) and (223,28.34) .. (223,30) .. controls (223,31.66) and (221.66,33) .. (220,33) .. controls (218.34,33) and (217,31.66) .. (217,30) -- cycle ;
\draw    (220,30) .. controls (229.88,34.47) and (235,34.72) .. (235,45) ;
\draw    (205,70) -- (205,60) ;
\draw    (235,70) -- (235,60) ;
\draw  [draw opacity=0] (165,30) -- (195,30) -- (195,55) -- (165,55) -- cycle ;
\draw  [draw opacity=0] (245,30) -- (275,30) -- (275,55) -- (245,55) -- cycle ;
\draw   (330,30.5) -- (350,30.5) -- (350,45.5) -- (330,45.5) -- cycle ;
\draw    (340,30) -- (340,20) ;
\draw    (340,55) -- (340,45) ;
\draw  [fill={rgb, 255:red, 0; green, 0; blue, 0 }  ,fill opacity=1 ] (337,55) .. controls (337,53.34) and (338.34,52) .. (340,52) .. controls (341.66,52) and (343,53.34) .. (343,55) .. controls (343,56.66) and (341.66,58) .. (340,58) .. controls (338.34,58) and (337,56.66) .. (337,55) -- cycle ;
\draw    (395,40) -- (395,20) ;
\draw  [fill={rgb, 255:red, 0; green, 0; blue, 0 }  ,fill opacity=1 ] (392,40) .. controls (392,38.34) and (393.34,37) .. (395,37) .. controls (396.66,37) and (398,38.34) .. (398,40) .. controls (398,41.66) and (396.66,43) .. (395,43) .. controls (393.34,43) and (392,41.66) .. (392,40) -- cycle ;
\draw  [draw opacity=0] (355,30) -- (385,30) -- (385,55) -- (355,55) -- cycle ;
\draw  [draw opacity=0] (405,30) -- (435,30) -- (435,55) -- (405,55) -- cycle ;

\draw (180,42.5) node    {$=$};
\draw (150,38) node  [font=\footnotesize]  {$f$};
\draw (260,42.5) node    {$;$};
\draw (370,42.5) node    {$=$};
\draw (420,42.5) node    {$;$};
\draw (205,52.5) node  [font=\footnotesize]  {$f$};
\draw (235,52.5) node  [font=\footnotesize]  {$f$};
\draw (340,38) node  [font=\footnotesize]  {$g$};
\draw (150,16.6) node [anchor=south] [inner sep=0.75pt]  [font=\tiny]  {$X$};
\draw (220,16.6) node [anchor=south] [inner sep=0.75pt]  [font=\tiny]  {$X$};
\draw (340,16.6) node [anchor=south] [inner sep=0.75pt]  [font=\tiny]  {$X$};
\draw (395,16.6) node [anchor=south] [inner sep=0.75pt]  [font=\tiny]  {$X$};
\draw (135,73.4) node [anchor=north] [inner sep=0.75pt]  [font=\tiny]  {$Y$};
\draw (165,73.4) node [anchor=north] [inner sep=0.75pt]  [font=\tiny]  {$Y$};
\draw (205,73.4) node [anchor=north] [inner sep=0.75pt]  [font=\tiny]  {$Y$};
\draw (235,73.4) node [anchor=north] [inner sep=0.75pt]  [font=\tiny]  {$Y$};

\end{tikzpicture}
     \caption{Definition of deterministic ($f$) and total ($g$) morphism.\label{fig:copy-discard-markov}}%
\end{figure}

\subsection{Markov categories}
\MarkovCategories{}~\cite{fritz_2020} are usually interpreted as an algebra of probabilistic processes that adequately expresses \emph{independence}, \emph{marginals}, and \emph{Bayesian networks}~\cite{Gol99,fong_2013,cho_2019,fritz_2020,jacobs2021causal,jacobs2020logical}. Let us first introduce the standard definition of a Markov category (\Cref{def:affine-copy-discard}). 

\begin{defi}[Affine copy-discard category \cite{cho_2019}]
  \label{def:affine-copy-discard}
  \AP An \intro{affine copy-discard category}—also known as \emph{Markov category} \cite{fritz_2020}—is a \kl{copy-discard category} where all morphisms are \kl{total}.
\end{defi}

This text, however, reserves the name \emph{Markov} for those that satisfy an additional axiom: the axiom of \kl{conditionals}.\footnote{See \Cref{subsec:markov-categories}, which discusses this nomenclature.}
The axiom of \conditionals{} ensures that every joint distribution, $f ፡ X → Y ⊗ Z$, can be split into a marginal distribution and a conditional distribution, $c ፡ X ⊗ Y → Z$. While not strictly needed for probabilistic reasoning, it is needed for probabilistic inference. \kl{Partial Markov categories} will inherit the axiom of \kl{conditionals}.

\begin{figure}[h!]

\tikzset{every picture/.style={line width=0.75pt}} %

\begin{tikzpicture}[x=0.75pt,y=0.75pt,yscale=-1,xscale=1]
\draw    (235,20) .. controls (244.88,24.47) and (250,19.72) .. (250,30) ;
\draw   (140,50) -- (160,50) -- (160,65) -- (140,65) -- cycle ;
\draw    (150,50) -- (150,40) ;
\draw    (235,20) -- (235,10) ;
\draw    (250,100) -- (250,95) ;
\draw  [draw opacity=0] (180,40) -- (210,40) -- (210,65) -- (180,65) -- cycle ;
\draw  [fill={rgb, 255:red, 255; green, 255; blue, 255 }  ,fill opacity=1 ] (260,30) -- (240,30) -- (240,45) -- (260,45) -- cycle ;
\draw  [fill={rgb, 255:red, 255; green, 255; blue, 255 }  ,fill opacity=1 ] (260,80) -- (240,80) -- (240,95) -- (260,95) -- cycle ;
\draw    (235,20) .. controls (225.17,24.43) and (220,19.47) .. (220,30) ;
\draw  [fill={rgb, 255:red, 0; green, 0; blue, 0 }  ,fill opacity=1 ] (232,20) .. controls (232,18.34) and (233.34,17) .. (235,17) .. controls (236.66,17) and (238,18.34) .. (238,20) .. controls (238,21.66) and (236.66,23) .. (235,23) .. controls (233.34,23) and (232,21.66) .. (232,20) -- cycle ;
\draw    (245,60) -- (245,45) ;
\draw    (245,60) .. controls (235.17,64.43) and (230,59.47) .. (230,70) ;
\draw  [fill={rgb, 255:red, 0; green, 0; blue, 0 }  ,fill opacity=1 ] (242,60) .. controls (242,58.34) and (243.34,57) .. (245,57) .. controls (246.66,57) and (248,58.34) .. (248,60) .. controls (248,61.66) and (246.66,63) .. (245,63) .. controls (243.34,63) and (242,61.66) .. (242,60) -- cycle ;
\draw    (245,60) .. controls (254.88,64.47) and (260,61.93) .. (260,70) ;
\draw    (220,65) .. controls (219.75,74.22) and (245,69.97) .. (245,80) ;
\draw    (220,45) -- (220,30) ;
\draw    (220,65) -- (220,45) ;
\draw    (230,70) .. controls (229.75,79.22) and (220,74.97) .. (220,85) ;
\draw    (260,70) .. controls (259.75,79.22) and (254.71,74.03) .. (255,80) ;
\draw    (220,85) -- (220,100) ;
\draw    (145,75) -- (145,65) ;
\draw    (155,75) -- (155,65) ;
\draw    (255,55) -- (255,45) ;
\draw  [fill={rgb, 255:red, 0; green, 0; blue, 0 }  ,fill opacity=1 ] (252,55) .. controls (252,53.34) and (253.34,52) .. (255,52) .. controls (256.66,52) and (258,53.34) .. (258,55) .. controls (258,56.66) and (256.66,58) .. (255,58) .. controls (253.34,58) and (252,56.66) .. (252,55) -- cycle ;

\draw (195,52.5) node    {$=$};
\draw (150,57.5) node  [font=\footnotesize]  {$f$};
\draw (250,37.5) node  [font=\footnotesize]  {$f$};
\draw (150,36.6) node [anchor=south] [inner sep=0.75pt]  [font=\tiny]  {$X$};
\draw (235,6.6) node [anchor=south] [inner sep=0.75pt]  [font=\tiny]  {$X$};
\draw (145,78.4) node [anchor=north] [inner sep=0.75pt]  [font=\tiny]  {$Y$};
\draw (155,78.4) node [anchor=north] [inner sep=0.75pt]  [font=\tiny]  {$Z$};
\draw (248.5,87.5) node  [font=\footnotesize]  {$c$};
\draw (220,103.4) node [anchor=north] [inner sep=0.75pt]  [font=\tiny]  {$Y$};
\draw (250,103.4) node [anchor=north] [inner sep=0.75pt]  [font=\tiny]  {$Z$};

\end{tikzpicture}
   \caption{\protect\Conditionals{} split joint channels $(f)$
  into their \protect\marginal{} and a \protect\conditional{} $(c)$.}%
  \label{fig:conditionals}
\end{figure}

\begin{defi}[Markov category, see \Cref{def:markovCategory}]
\AP A Markov category is an \kl{affine copy-discard category} such that, for every morphism $f ፡ X → Y ⊗ Z$, there exists a morphism $c ፡ X ⊗ Y → Z$ satisfying the equation in  \Cref{fig:conditionals}. We say that this morphism ($c$) is a \kl{conditional} (of $f$).
\end{defi}

However—even with \kl{conditionals}—\MarkovCategories{} cannot express constraints; i.e., we cannot impose \emph{a posteriori} that a sampling from a channel must coincide with some observation.
This limitation is structural.
\MarkovCategories{} axiomatically impose that every morphism is \kl{total} (rightmost equation in \Cref{fig:copy-discard-markov}); for distributions, this is equivalent to considering only (proper) probability distributions, i.e.\ those where the individual probabilities sum up to $1$.
When updating, this becomes a limitation: updates yield not proper distributions, but only subdistributions. We need the notion of partiality afforded by \kl{cartesian restriction categories}.

\subsection{Cartesian restriction categories} 
\kl{Cartesian restriction categories}~\cite{cockett2007restriction, curien_obtulowicz_1989} are an algebra of processes that adequately expresses partiality and \emph{constraints}.
All morphisms in a \kl{cartesian restriction category} are \kl{deterministic} (leftmost equation in \Cref{fig:copy-discard-markov}).
This equation rules out nondeterministic behaviour: the stochastic kernels that commute with copying are the ones that encode functions.

Let us introduce not the usual definition of \kl{cartesian restriction categories} but their characterization as \kl{copy-discard categories} of \kl{deterministic} morphisms \cite[Theorem 5.2]{cockett2007restriction}.

\begin{defi}[Relevant copy-discard category]
  A \intro{relevant copy-discard category}—also known as \intro{cartesian restriction category}—is a \kl{copy-discard category} where all morphisms are \kl{deterministic}.
\end{defi}

Some \kl{cartesian restriction categories} have an additional \emph{compare} morphism for each object that checks if two resources are equal: \emph{compare} morphisms interact with copy and discard satisfying the axioms in \Cref{fig:comparator-axioms}: these are known as \emph{discrete cartesian restriction categories}~\cite{cockett2012range2, di_liberti_nester_2021}.

\begin{figure}[h!]

\tikzset{every picture/.style={line width=0.75pt}} %

\begin{tikzpicture}[x=0.75pt,y=0.75pt,yscale=-1,xscale=1]
\draw    (70,85) -- (70,100) ;
\draw    (85,115) -- (85,125) ;
\draw    (85,115) .. controls (75.17,110.57) and (70,110.53) .. (70,100) ;
\draw  [fill={rgb, 255:red, 0; green, 0; blue, 0 }  ,fill opacity=1 ] (82,115) .. controls (82,116.66) and (83.34,118) .. (85,118) .. controls (86.66,118) and (88,116.66) .. (88,115) .. controls (88,113.34) and (86.66,112) .. (85,112) .. controls (83.34,112) and (82,113.34) .. (82,115) -- cycle ;
\draw    (85,115) .. controls (94.88,110.53) and (100,110.28) .. (100,100) ;
\draw    (100,100) .. controls (90.17,95.57) and (85,95.53) .. (85,85) ;
\draw  [fill={rgb, 255:red, 0; green, 0; blue, 0 }  ,fill opacity=1 ] (97,100) .. controls (97,101.66) and (98.34,103) .. (100,103) .. controls (101.66,103) and (103,101.66) .. (103,100) .. controls (103,98.34) and (101.66,97) .. (100,97) .. controls (98.34,97) and (97,98.34) .. (97,100) -- cycle ;
\draw    (100,100) .. controls (109.88,95.53) and (115,95.28) .. (115,85) ;
\draw    (175,115) -- (175,125) ;
\draw    (175,115) .. controls (165.17,110.57) and (160,110.53) .. (160,100) ;
\draw  [fill={rgb, 255:red, 0; green, 0; blue, 0 }  ,fill opacity=1 ] (172,115) .. controls (172,116.66) and (173.34,118) .. (175,118) .. controls (176.66,118) and (178,116.66) .. (178,115) .. controls (178,113.34) and (176.66,112) .. (175,112) .. controls (173.34,112) and (172,113.34) .. (172,115) -- cycle ;
\draw    (175,115) .. controls (184.88,110.53) and (190,110.28) .. (190,100) ;
\draw    (190,85) -- (190,100) ;
\draw    (160,100) .. controls (150.17,95.57) and (145,95.53) .. (145,85) ;
\draw  [fill={rgb, 255:red, 0; green, 0; blue, 0 }  ,fill opacity=1 ] (157,100) .. controls (157,101.66) and (158.34,103) .. (160,103) .. controls (161.66,103) and (163,101.66) .. (163,100) .. controls (163,98.34) and (161.66,97) .. (160,97) .. controls (158.34,97) and (157,98.34) .. (157,100) -- cycle ;
\draw    (160,100) .. controls (169.88,95.53) and (175,95.28) .. (175,85) ;
\draw    (235,115) -- (235,125) ;
\draw    (235,115) .. controls (225.17,110.57) and (220.29,110.69) .. (220,105) ;
\draw  [fill={rgb, 255:red, 0; green, 0; blue, 0 }  ,fill opacity=1 ] (232,115) .. controls (232,116.66) and (233.34,118) .. (235,118) .. controls (236.66,118) and (238,116.66) .. (238,115) .. controls (238,113.34) and (236.66,112) .. (235,112) .. controls (233.34,112) and (232,113.34) .. (232,115) -- cycle ;
\draw    (235,115) .. controls (244.88,110.53) and (249.86,110.4) .. (250,105) ;
\draw    (220,105) .. controls (220.5,94.57) and (249.83,95.23) .. (250,85) ;
\draw    (250,105) .. controls (250.5,94.57) and (219.83,95.23) .. (220,85) ;
\draw  [draw opacity=0] (250.01,115) -- (280.01,115) -- (280.01,90) -- (250.01,90) -- cycle ;
\draw    (295.01,110) -- (295.01,125) ;
\draw    (295.01,110) .. controls (285.17,105.57) and (280.01,105.53) .. (280.01,95) ;
\draw  [fill={rgb, 255:red, 0; green, 0; blue, 0 }  ,fill opacity=1 ] (292.01,110) .. controls (292.01,111.66) and (293.35,113) .. (295.01,113) .. controls (296.66,113) and (298.01,111.66) .. (298.01,110) .. controls (298.01,108.34) and (296.66,107) .. (295.01,107) .. controls (293.35,107) and (292.01,108.34) .. (292.01,110) -- cycle ;
\draw    (295.01,110) .. controls (304.88,105.53) and (310.01,105.28) .. (310.01,95) ;
\draw    (280.01,85) -- (280.01,95) ;
\draw    (310.01,85) -- (310.01,95) ;
\draw    (580,95) -- (580,85) ;
\draw    (580,95) .. controls (570.17,99.43) and (570,99.47) .. (570,110) ;
\draw  [fill={rgb, 255:red, 0; green, 0; blue, 0 }  ,fill opacity=1 ] (577,95) .. controls (577,93.34) and (578.34,92) .. (580,92) .. controls (581.66,92) and (583,93.34) .. (583,95) .. controls (583,96.66) and (581.66,98) .. (580,98) .. controls (578.34,98) and (577,96.66) .. (577,95) -- cycle ;
\draw    (580,95) .. controls (589.88,99.47) and (590,94.72) .. (590,105) ;
\draw    (600,115) .. controls (590.17,110.57) and (590,115.53) .. (590,105) ;
\draw  [fill={rgb, 255:red, 0; green, 0; blue, 0 }  ,fill opacity=1 ] (597,115) .. controls (597,116.66) and (598.34,118) .. (600,118) .. controls (601.66,118) and (603,116.66) .. (603,115) .. controls (603,113.34) and (601.66,112) .. (600,112) .. controls (598.34,112) and (597,113.34) .. (597,115) -- cycle ;
\draw    (600,115) .. controls (609.88,110.53) and (610,110.28) .. (610,100) ;
\draw    (610,100) -- (610,85) ;
\draw    (570,125) -- (570,110) ;
\draw    (600,125) -- (600,115) ;
\draw    (470,95) -- (470,85) ;
\draw    (470,95) .. controls (479.83,99.43) and (480,99.47) .. (480,110) ;
\draw  [fill={rgb, 255:red, 0; green, 0; blue, 0 }  ,fill opacity=1 ] (473,95) .. controls (473,93.34) and (471.66,92) .. (470,92) .. controls (468.34,92) and (467,93.34) .. (467,95) .. controls (467,96.66) and (468.34,98) .. (470,98) .. controls (471.66,98) and (473,96.66) .. (473,95) -- cycle ;
\draw    (470,95) .. controls (460.13,99.47) and (460,94.72) .. (460,105) ;
\draw    (450,115) .. controls (459.83,110.57) and (460,115.53) .. (460,105) ;
\draw  [fill={rgb, 255:red, 0; green, 0; blue, 0 }  ,fill opacity=1 ] (453,115) .. controls (453,116.66) and (451.66,118) .. (450,118) .. controls (448.34,118) and (447,116.66) .. (447,115) .. controls (447,113.34) and (448.34,112) .. (450,112) .. controls (451.66,112) and (453,113.34) .. (453,115) -- cycle ;
\draw    (450,115) .. controls (440.13,110.53) and (440,110.28) .. (440,100) ;
\draw    (440,100) -- (440,85) ;
\draw    (480,125) -- (480,110) ;
\draw    (450,125) -- (450,115) ;
\draw    (525,100) .. controls (515.17,95.57) and (510,95.53) .. (510,85) ;
\draw  [fill={rgb, 255:red, 0; green, 0; blue, 0 }  ,fill opacity=1 ] (522,100) .. controls (522,101.66) and (523.34,103) .. (525,103) .. controls (526.66,103) and (528,101.66) .. (528,100) .. controls (528,98.34) and (526.66,97) .. (525,97) .. controls (523.34,97) and (522,98.34) .. (522,100) -- cycle ;
\draw    (525,100) .. controls (534.88,95.53) and (540,95.28) .. (540,85) ;
\draw    (525,110) .. controls (515.17,114.43) and (510,114.47) .. (510,125) ;
\draw  [fill={rgb, 255:red, 0; green, 0; blue, 0 }  ,fill opacity=1 ] (522,110) .. controls (522,108.34) and (523.34,107) .. (525,107) .. controls (526.66,107) and (528,108.34) .. (528,110) .. controls (528,111.66) and (526.66,113) .. (525,113) .. controls (523.34,113) and (522,111.66) .. (522,110) -- cycle ;
\draw    (525,110) .. controls (534.88,114.47) and (540,114.72) .. (540,125) ;
\draw    (525,110) -- (525,100) ;
\draw  [draw opacity=0] (480,95) -- (510,95) -- (510,115) -- (480,115) -- cycle ;
\draw  [draw opacity=0] (540,95) -- (570,95) -- (570,115) -- (540,115) -- cycle ;
\draw  [draw opacity=0] (115,115) -- (145,115) -- (145,90) -- (115,90) -- cycle ;
\draw    (355,115) .. controls (345.17,110.57) and (340,115.53) .. (340,105) ;
\draw  [fill={rgb, 255:red, 0; green, 0; blue, 0 }  ,fill opacity=1 ] (352,115) .. controls (352,116.66) and (353.34,118) .. (355,118) .. controls (356.66,118) and (358,116.66) .. (358,115) .. controls (358,113.34) and (356.66,112) .. (355,112) .. controls (353.34,112) and (352,113.34) .. (352,115) -- cycle ;
\draw    (355,115) .. controls (364.88,110.53) and (370,115.28) .. (370,105) ;
\draw    (355,95) .. controls (345.17,99.43) and (340,94.47) .. (340,105) ;
\draw  [fill={rgb, 255:red, 0; green, 0; blue, 0 }  ,fill opacity=1 ] (352,95) .. controls (352,93.34) and (353.34,92) .. (355,92) .. controls (356.66,92) and (358,93.34) .. (358,95) .. controls (358,96.66) and (356.66,98) .. (355,98) .. controls (353.34,98) and (352,96.66) .. (352,95) -- cycle ;
\draw    (355,95) .. controls (364.88,99.47) and (370,94.72) .. (370,105) ;
\draw    (355,125) -- (355,115) ;
\draw    (355,85) -- (355,95) ;
\draw  [draw opacity=0] (370,115) -- (400,115) -- (400,90) -- (370,90) -- cycle ;
\draw    (405,85) -- (405,125) ;
\draw  [draw opacity=0] (310,120) -- (340,120) -- (340,90) -- (310,90) -- cycle ;
\draw  [draw opacity=0] (410,120) -- (440,120) -- (440,90) -- (410,90) -- cycle ;

\draw (440,81.6) node [anchor=south] [inner sep=0.75pt]  [font=\tiny]  {$X$};
\draw (470,81.6) node [anchor=south] [inner sep=0.75pt]  [font=\tiny]  {$X$};
\draw (510,81.6) node [anchor=south] [inner sep=0.75pt]  [font=\tiny]  {$X$};
\draw (265.01,102.5) node    {$=$};
\draw (450,128.4) node [anchor=north] [inner sep=0.75pt]  [font=\tiny]  {$X$};
\draw (480,128.4) node [anchor=north] [inner sep=0.75pt]  [font=\tiny]  {$X$};
\draw (495,105) node    {$=$};
\draw (555,105) node    {$=$};
\draw (540,81.6) node [anchor=south] [inner sep=0.75pt]  [font=\tiny]  {$X$};
\draw (580,81.6) node [anchor=south] [inner sep=0.75pt]  [font=\tiny]  {$X$};
\draw (610,81.6) node [anchor=south] [inner sep=0.75pt]  [font=\tiny]  {$X$};
\draw (540,128.4) node [anchor=north] [inner sep=0.75pt]  [font=\tiny]  {$X$};
\draw (510,128.4) node [anchor=north] [inner sep=0.75pt]  [font=\tiny]  {$X$};
\draw (570,128.4) node [anchor=north] [inner sep=0.75pt]  [font=\tiny]  {$X$};
\draw (600,128.4) node [anchor=north] [inner sep=0.75pt]  [font=\tiny]  {$X$};
\draw (295.01,128.4) node [anchor=north] [inner sep=0.75pt]  [font=\tiny]  {$X$};
\draw (235,128.4) node [anchor=north] [inner sep=0.75pt]  [font=\tiny]  {$X$};
\draw (220,81.6) node [anchor=south] [inner sep=0.75pt]  [font=\tiny]  {$X$};
\draw (250,81.6) node [anchor=south] [inner sep=0.75pt]  [font=\tiny]  {$X$};
\draw (280.01,81.6) node [anchor=south] [inner sep=0.75pt]  [font=\tiny]  {$X$};
\draw (310.01,81.6) node [anchor=south] [inner sep=0.75pt]  [font=\tiny]  {$X$};
\draw (130,102.5) node    {$=$};
\draw (70,81.6) node [anchor=south] [inner sep=0.75pt]  [font=\tiny]  {$X$};
\draw (85,81.6) node [anchor=south] [inner sep=0.75pt]  [font=\tiny]  {$X$};
\draw (115,81.6) node [anchor=south] [inner sep=0.75pt]  [font=\tiny]  {$X$};
\draw (145,81.6) node [anchor=south] [inner sep=0.75pt]  [font=\tiny]  {$X$};
\draw (175,81.6) node [anchor=south] [inner sep=0.75pt]  [font=\tiny]  {$X$};
\draw (190,81.6) node [anchor=south] [inner sep=0.75pt]  [font=\tiny]  {$X$};
\draw (175,128.4) node [anchor=north] [inner sep=0.75pt]  [font=\tiny]  {$X$};
\draw (85,128.4) node [anchor=north] [inner sep=0.75pt]  [font=\tiny]  {$X$};
\draw (385,102.5) node    {$=$};
\draw (405,128.4) node [anchor=north] [inner sep=0.75pt]  [font=\tiny]  {$X$};
\draw (355,128.4) node [anchor=north] [inner sep=0.75pt]  [font=\tiny]  {$X$};
\draw (355,81.6) node [anchor=south] [inner sep=0.75pt]  [font=\tiny]  {$X$};
\draw (405,81.6) node [anchor=south] [inner sep=0.75pt]  [font=\tiny]  {$X$};

\end{tikzpicture}
   \caption{Axioms of compare morphisms in a discrete cartesian restriction category.}%
  \label{fig:comparator-axioms}
\end{figure}

\begin{defi}[Discrete cartesian restriction category]
  A \intro{discrete cartesian restriction category} is a \kl{cartesian restriction category} endowed with \kl{compare} morphisms.
\end{defi}

\begin{rem}
  Some \partialMarkovCategories{} inherit compare morphims that behave as in \Cref{fig:comparator-axioms}:
we call these \discretePartialMarkovCategories{}.
As we shall see in the next example, \kl{compare} morphisms are useful for expressing constraints coming from observations.
\end{rem}

\PartialMarkovCategories{} blend \MarkovCategories{} and \kl{cartesian restriction categories} while imposing neither of the axioms in \Cref{fig:copy-discard-markov}. They retain the ability to copy and discard resources, but its processes can be probabilistic and also partial.
In other words, \partialMarkovCategories{} have a (not necessarily natural) comonoid structure on every object that makes them \emph{\copyDiscardCategories}~\cite{corradini_gadducci_1999,cho_2019,heunenvicary2019}.
\medskip

\subsection{Example — the Three Prisoners Problem}%
\label{sec:threeprisoners}

To illustrate both the power of diagrammatic reasoning in \kl{Markov categories}
and their limitations, let us consider the well-known \emph{Three Prisoners
Problem}\footnote{As in Casella and Berger's book~\cite[Example~1.3.4]{casella2024statistical}, or Jacobs' recent treatment \cite{jacobs}. Consider the Monty Hall problem for a similar example \cite{selvin75:montyhall}.} as a first example.

\begin{quote}
Three prisoners, A, B and C are on death row. The governor has decided to pardon
one at random. This information is shared only with the warden. Prisoner A asks
the warden who is to be pardoned. The warden refuses to answer. Prisoner A then
follows by asking who is to be executed. The warden thinks for a while, and
answers that Prisoner B is to be executed. The warden thinks that no useful
information has been given to Prisoner A, since clearly one of Prisoner B or
Prisoner C will be executed. However, prisoner A is happy, thinking that his
chances of being pardoned are now 50\%. What has really been learned as a
result of this interaction?
\end{quote}

Some underlying implicit assumptions ought be clarified: \emph{(i)} the warden is assumed not to lie if not necessary, but \emph{(ii)} will avoid telling the asking prisoner that they are to be executed. Explicitly, \emph{(iii)} if the asking prisoner is the one to be pardoned, the warden randomly chooses between the two prisoners to be executed when answering.

To begin modelling this problem, we work in the \kl{Markov category} of finitary
stochastic channels $\Stoch$ (\Cref{def:finitary-distributions}). Objects are sets, and morphisms $X → Y$ are (stochastic)
channels in the usual sense \cite{fritz_2020,jacobs}: functions $X → \distr(Y)$ where
$\distr(Y)$ is the set of finitely supported probability distributions on $Y$.

Let $\mathsf{Prisoners} = \{\mathsf{A}, \mathsf{B}, \mathsf{C}\}$ be the set of
prisoners. The governor's decision is modelled by the uniform distribution on
$\mathsf{Prisoners}$: that is, by the morphism $\mathsf{Pardoned}: 1\to \mathsf{Prisoners}$
mapping the singleton $1$ to the uniform distribution $1/3\ket{A}+ 1/3 \ket{B}+
1/3\ket{C}$. Prisoner A—the protagonist of the story—is modelled by an ordinary
function, $\mathsf{A}: 1\to \mathsf{Prisoners}$, corresponding to the
distribution $1\ket{A}$.
The warden's answering strategy, implied by our implicit assumptions, is modelled
by a morphism
$\mathsf{Answer}:\mathsf{Prisoners}\times\mathsf{Prisoners}\to\mathsf{Prisoners}$,
where the first argument is the prisoner asking the question and the second is
the prisoner to be pardoned. Concretely, it is defined by
\[
\mathsf{Answer}(x,y) = 
  \begin{cases}
    \frac{1}{2}\ket{p_1} + \frac{1}{2}\ket{p_2} & 
    \text{if $x=y$, where $\{p_1,p_2\}=\mathsf{Prisoners}\backslash \{x\}$;} \\
     1\ket{p} &
     \text{otherwise, where $\{p\}= \mathsf{Prisoners}\backslash \{x,y\}.$} \end{cases}
\]
The set-up for the Three Prisoners Problem can then be captured by the string diagram in \Cref{fig:threeprisoners}.
\begin{figure}[!ht]

\tikzset{every picture/.style={line width=0.75pt}} %

\begin{tikzpicture}[x=0.75pt,y=0.75pt,yscale=-1,xscale=1]
\draw    (90,35) -- (90,25) ;
\draw  [fill={rgb, 255:red, 0; green, 0; blue, 0 }  ,fill opacity=1 ] (87,35) .. controls (87,33.34) and (88.34,32) .. (90,32) .. controls (91.66,32) and (93,33.34) .. (93,35) .. controls (93,36.66) and (91.66,38) .. (90,38) .. controls (88.34,38) and (87,36.66) .. (87,35) -- cycle ;
\draw  [draw opacity=0] (120,30) -- (150,30) -- (150,55) -- (120,55) -- cycle ;
\draw   (65,10) -- (115,10) -- (115,25) -- (65,25) -- cycle ;
\draw   (30,45) -- (80,45) -- (80,60) -- (30,60) -- cycle ;
\draw   (30,10) -- (50,10) -- (50,25) -- (30,25) -- cycle ;
\draw    (90,35) .. controls (80.17,39.43) and (65,34.47) .. (65,45) ;
\draw    (90,35) .. controls (100.18,40.04) and (115,34.47) .. (115,45) ;
\draw    (40,25) .. controls (40.91,39.66) and (44.91,29.26) .. (45,45) ;
\draw    (55,60) -- (55,70) ;
\draw    (115,45) -- (115,70) ;

\draw (135,42.5) node    {$;$};
\draw (90,17.5) node  [font=\scriptsize]  {$\mathsf{Pardoned}$};
\draw (55,52.5) node  [font=\scriptsize]  {$\mathsf{Answer}$};
\draw (40,17.5) node  [font=\footnotesize]  {$A$};

\end{tikzpicture}
   \caption{Set-up for the Three Prisoners Problem.}
  \label{fig:threeprisoners}
\end{figure}
The result of this set-up is a map of type $1 → \mathsf{Prisoners} ×
\mathsf{Prisoners}$, which is the same thing as a (total) joint probability
distribution on $\mathsf{Prisoners} × \mathsf{Prisoners}$. In fact, a simple
calculation confirms that this distribution is
\begin{equation}%
  \label{eq:setup}
  \frac{1}{6}\ket{B,A}+\frac{1}{6}\ket{C,A}+\frac{1}{3}\ket{C,B}+\frac{1}{3}\ket{B,C}.
\end{equation}
This does not yet cover the crucial part of the Three Prisoners Problem—the
effect on this joint probability distribution obtained from learning the
warden's answer to Prisoner A's question. To do this satisfactorily, we need our
theory of probabilistic processes to deal with \emph{partial} distributions.

Let us illustrate the power of
\discretePartialMarkovCategories{} by reasoning about the Three Prisoners Problem.
We know that the warden answers ``Prisoner B'' in response to Prisoner A's
question. We can include this information in our model by introducing this as a
partial equality check, as in \Cref{fig:threePrisoners2}.
\begin{figure}[!h]

\tikzset{every picture/.style={line width=0.75pt}} %

\begin{tikzpicture}[x=0.75pt,y=0.75pt,yscale=-1,xscale=1]
\draw    (250,35) -- (250,25) ;
\draw  [fill={rgb, 255:red, 0; green, 0; blue, 0 }  ,fill opacity=1 ] (247,35) .. controls (247,33.34) and (248.34,32) .. (250,32) .. controls (251.66,32) and (253,33.34) .. (253,35) .. controls (253,36.66) and (251.66,38) .. (250,38) .. controls (248.34,38) and (247,36.66) .. (247,35) -- cycle ;
\draw  [draw opacity=0] (280,40) -- (310,40) -- (310,65) -- (280,65) -- cycle ;
\draw   (225,10) -- (275,10) -- (275,25) -- (225,25) -- cycle ;
\draw   (190,45) -- (240,45) -- (240,60) -- (190,60) -- cycle ;
\draw   (190,10) -- (210,10) -- (210,25) -- (190,25) -- cycle ;
\draw    (250,35) .. controls (240.17,39.43) and (225,34.47) .. (225,45) ;
\draw    (250,35) .. controls (260.18,40.04) and (275,34.47) .. (275,45) ;
\draw    (200,25) .. controls (200.91,39.66) and (204.91,29.26) .. (205,45) ;
\draw    (275,45) -- (275,80) ;
\draw   (155,45) -- (175,45) -- (175,60) -- (155,60) -- cycle ;
\draw    (190,70) -- (190,80) ;
\draw  [fill={rgb, 255:red, 0; green, 0; blue, 0 }  ,fill opacity=1 ] (187,70) .. controls (187,71.66) and (188.34,73) .. (190,73) .. controls (191.66,73) and (193,71.66) .. (193,70) .. controls (193,68.34) and (191.66,67) .. (190,67) .. controls (188.34,67) and (187,68.34) .. (187,70) -- cycle ;
\draw    (190,70) .. controls (180.17,65.57) and (165,70.53) .. (165,60) ;
\draw    (190,70) .. controls (200.18,64.96) and (215,70.53) .. (215,60) ;

\draw (295,52.5) node    {$;$};
\draw (250,17.5) node  [font=\scriptsize]  {$\mathsf{Pardoned}$};
\draw (215,52.5) node  [font=\scriptsize]  {$\mathsf{Answer}$};
\draw (200,17.5) node  [font=\footnotesize]  {$\mathsf{A}$};
\draw (165,52.5) node  [font=\footnotesize]  {$\mathsf{B}$};

\end{tikzpicture}
   \caption{Set-up for the Three Prisoners Problem, after the equality check.}
  \label{fig:threePrisoners2}
\end{figure}
The result is a kind of ``action at a distance'' collapse of the joint
probability distribution. One way to understand this is that the equality check
deletes all of the subterms of~\eqref{eq:setup} that do not agree with what has been
observed:
\begin{equation}
\label{eq:step2}
\frac{1}{6}\ket{B,A} + \frac{1}{6}\cancel{\ket{C,A}}+\frac{1}{3}\cancel{\ket{C,B}}+\frac{1}{3}\ket{B,C}
=
\frac{1}{6}\ket{B,A} + \frac{1}{3}\ket{B,C}.
\end{equation}
Note that the result is not anymore a proper probability distribution but a
\kl{subdistribution}—$1/6 + 1/3 ≠ 1$—which puts us outside the realm of \kl{Markov
categories}. 

Next, given that we are only interested in the probability of who has been
pardoned, we can marginalise on the warden's answer, obtaining the final set-up
in \Cref{fig:threePrisoners3}, which evaluates to the \kl{subdistribution}
$\frac{1}{6}\ket{A} + \frac{1}{3}\ket{C}$.
\begin{figure}[!h]

\tikzset{every picture/.style={line width=0.75pt}} %

\begin{tikzpicture}[x=0.75pt,y=0.75pt,yscale=-1,xscale=1]
\draw    (250,35) -- (250,25) ;
\draw  [fill={rgb, 255:red, 0; green, 0; blue, 0 }  ,fill opacity=1 ] (247,35) .. controls (247,33.34) and (248.34,32) .. (250,32) .. controls (251.66,32) and (253,33.34) .. (253,35) .. controls (253,36.66) and (251.66,38) .. (250,38) .. controls (248.34,38) and (247,36.66) .. (247,35) -- cycle ;
\draw  [draw opacity=0] (280,40) -- (310,40) -- (310,65) -- (280,65) -- cycle ;
\draw   (225,10) -- (275,10) -- (275,25) -- (225,25) -- cycle ;
\draw   (190,45) -- (240,45) -- (240,60) -- (190,60) -- cycle ;
\draw   (190,10) -- (210,10) -- (210,25) -- (190,25) -- cycle ;
\draw    (250,35) .. controls (240.17,39.43) and (225,34.47) .. (225,45) ;
\draw    (250,35) .. controls (260.18,40.04) and (275,34.47) .. (275,45) ;
\draw    (200,25) .. controls (200.91,39.66) and (204.91,29.26) .. (205,45) ;
\draw    (275,45) -- (275,90) ;
\draw   (155,45) -- (175,45) -- (175,60) -- (155,60) -- cycle ;
\draw    (190,70) -- (190,80) ;
\draw  [fill={rgb, 255:red, 0; green, 0; blue, 0 }  ,fill opacity=1 ] (187,70) .. controls (187,71.66) and (188.34,73) .. (190,73) .. controls (191.66,73) and (193,71.66) .. (193,70) .. controls (193,68.34) and (191.66,67) .. (190,67) .. controls (188.34,67) and (187,68.34) .. (187,70) -- cycle ;
\draw    (190,70) .. controls (180.17,65.57) and (165,70.53) .. (165,60) ;
\draw    (190,70) .. controls (200.18,64.96) and (215,70.53) .. (215,60) ;
\draw  [fill={rgb, 255:red, 0; green, 0; blue, 0 }  ,fill opacity=1 ] (187,80) .. controls (187,78.34) and (188.34,77) .. (190,77) .. controls (191.66,77) and (193,78.34) .. (193,80) .. controls (193,81.66) and (191.66,83) .. (190,83) .. controls (188.34,83) and (187,81.66) .. (187,80) -- cycle ;

\draw (295,52.5) node    {$;$};
\draw (250,17.5) node  [font=\scriptsize]  {$\mathsf{Pardoned}$};
\draw (215,52.5) node  [font=\scriptsize]  {$\mathsf{Answer}$};
\draw (200,17.5) node  [font=\footnotesize]  {$\mathsf{A}$};
\draw (165,52.5) node  [font=\footnotesize]  {$\mathsf{B}$};

\end{tikzpicture}
   \caption{Set-up for the Three Prisoners Problem, after marginalisation.}
  \label{fig:threePrisoners3}
\end{figure}

As a final step, we can \kl{normalise}: the axioms of \kl{partial Markov
categories} imply that each channel factors into \emph{(i)} a scalar obtained by
discarding its output—called its \emph{validity}—and \emph{(ii)} a new
channel—called its \emph{normalization}—that we represent with a blue box around
the original morphism. The resulting validity can be computed to be $1/2$. The
\kl{normalised} (proper) probability distribution is $\frac{1}{3}\ket{A} +
\frac{2}{3}\ket{C}$.
\begin{figure}[!h]

\tikzset{every picture/.style={line width=0.75pt}} %

\begin{tikzpicture}[x=0.75pt,y=0.75pt,yscale=-1,xscale=1]
\draw    (150,35) -- (150,25) ;
\draw  [fill={rgb, 255:red, 0; green, 0; blue, 0 }  ,fill opacity=1 ] (147,35) .. controls (147,33.34) and (148.34,32) .. (150,32) .. controls (151.66,32) and (153,33.34) .. (153,35) .. controls (153,36.66) and (151.66,38) .. (150,38) .. controls (148.34,38) and (147,36.66) .. (147,35) -- cycle ;
\draw  [draw opacity=0] (180,40) -- (210,40) -- (210,65) -- (180,65) -- cycle ;
\draw   (125,10) -- (175,10) -- (175,25) -- (125,25) -- cycle ;
\draw   (90,45) -- (140,45) -- (140,60) -- (90,60) -- cycle ;
\draw   (90,10) -- (110,10) -- (110,25) -- (90,25) -- cycle ;
\draw    (150,35) .. controls (140.17,39.43) and (125,34.47) .. (125,45) ;
\draw    (150,35) .. controls (160.18,40.04) and (175,34.47) .. (175,45) ;
\draw    (100,25) .. controls (100.91,39.66) and (104.91,29.26) .. (105,45) ;
\draw    (175,45) -- (175,90) ;
\draw   (55,45) -- (75,45) -- (75,60) -- (55,60) -- cycle ;
\draw    (90,70) -- (90,80) ;
\draw  [fill={rgb, 255:red, 0; green, 0; blue, 0 }  ,fill opacity=1 ] (87,70) .. controls (87,71.66) and (88.34,73) .. (90,73) .. controls (91.66,73) and (93,71.66) .. (93,70) .. controls (93,68.34) and (91.66,67) .. (90,67) .. controls (88.34,67) and (87,68.34) .. (87,70) -- cycle ;
\draw    (90,70) .. controls (80.17,65.57) and (65,70.53) .. (65,60) ;
\draw    (90,70) .. controls (100.18,64.96) and (115,70.53) .. (115,60) ;
\draw  [fill={rgb, 255:red, 0; green, 0; blue, 0 }  ,fill opacity=1 ] (87,80) .. controls (87,78.34) and (88.34,77) .. (90,77) .. controls (91.66,77) and (93,78.34) .. (93,80) .. controls (93,81.66) and (91.66,83) .. (90,83) .. controls (88.34,83) and (87,81.66) .. (87,80) -- cycle ;
\draw  [color={rgb, 255:red, 129; green, 161; blue, 193 }  ,draw opacity=0.75 ][fill={rgb, 255:red, 129; green, 161; blue, 193 }  ,fill opacity=0.25 ] (350,12.84) .. controls (350,8.51) and (353.51,5) .. (357.84,5) -- (472.16,5) .. controls (476.49,5) and (480,8.51) .. (480,12.84) -- (480,77.16) .. controls (480,81.49) and (476.49,85) .. (472.16,85) -- (357.84,85) .. controls (353.51,85) and (350,81.49) .. (350,77.16) -- cycle ;
\draw    (450,35.5) -- (450,25.5) ;
\draw  [fill={rgb, 255:red, 0; green, 0; blue, 0 }  ,fill opacity=1 ] (447,35.5) .. controls (447,33.84) and (448.34,32.5) .. (450,32.5) .. controls (451.66,32.5) and (453,33.84) .. (453,35.5) .. controls (453,37.16) and (451.66,38.5) .. (450,38.5) .. controls (448.34,38.5) and (447,37.16) .. (447,35.5) -- cycle ;
\draw   (425,10.5) -- (475,10.5) -- (475,25.5) -- (425,25.5) -- cycle ;
\draw   (390,45.5) -- (440,45.5) -- (440,60.5) -- (390,60.5) -- cycle ;
\draw   (390,10.5) -- (410,10.5) -- (410,25.5) -- (390,25.5) -- cycle ;
\draw    (450,35.5) .. controls (440.17,39.93) and (425,34.97) .. (425,45.5) ;
\draw    (450,35.5) .. controls (460.18,40.54) and (475,34.97) .. (475,45.5) ;
\draw    (400,25.5) .. controls (400.91,40.16) and (404.91,29.76) .. (405,45.5) ;
\draw   (355,45.5) -- (375,45.5) -- (375,60.5) -- (355,60.5) -- cycle ;
\draw    (390,70.5) -- (390,80.5) ;
\draw  [fill={rgb, 255:red, 0; green, 0; blue, 0 }  ,fill opacity=1 ] (387,70.5) .. controls (387,72.16) and (388.34,73.5) .. (390,73.5) .. controls (391.66,73.5) and (393,72.16) .. (393,70.5) .. controls (393,68.84) and (391.66,67.5) .. (390,67.5) .. controls (388.34,67.5) and (387,68.84) .. (387,70.5) -- cycle ;
\draw    (390,70.5) .. controls (380.17,66.07) and (365,71.03) .. (365,60.5) ;
\draw    (390,70.5) .. controls (400.18,65.46) and (415,71.03) .. (415,60.5) ;
\draw  [fill={rgb, 255:red, 0; green, 0; blue, 0 }  ,fill opacity=1 ] (387,80.5) .. controls (387,78.84) and (388.34,77.5) .. (390,77.5) .. controls (391.66,77.5) and (393,78.84) .. (393,80.5) .. controls (393,82.16) and (391.66,83.5) .. (390,83.5) .. controls (388.34,83.5) and (387,82.16) .. (387,80.5) -- cycle ;
\draw    (475,45) .. controls (474.32,64.37) and (459.65,70.03) .. (460,95) ;
\draw  [draw opacity=0] (480,35) -- (510,35) -- (510,60) -- (480,60) -- cycle ;
\draw    (310,35.5) -- (310,25.5) ;
\draw  [fill={rgb, 255:red, 0; green, 0; blue, 0 }  ,fill opacity=1 ] (307,35.5) .. controls (307,33.84) and (308.34,32.5) .. (310,32.5) .. controls (311.66,32.5) and (313,33.84) .. (313,35.5) .. controls (313,37.16) and (311.66,38.5) .. (310,38.5) .. controls (308.34,38.5) and (307,37.16) .. (307,35.5) -- cycle ;
\draw   (285,10.5) -- (335,10.5) -- (335,25.5) -- (285,25.5) -- cycle ;
\draw   (250,45.5) -- (300,45.5) -- (300,60.5) -- (250,60.5) -- cycle ;
\draw   (250,10.5) -- (270,10.5) -- (270,25.5) -- (250,25.5) -- cycle ;
\draw    (310,35.5) .. controls (300.17,39.93) and (285,34.97) .. (285,45.5) ;
\draw    (310,35.5) .. controls (320.18,40.54) and (335,34.97) .. (335,45.5) ;
\draw    (260,25.5) .. controls (260.91,40.16) and (264.91,29.76) .. (265,45.5) ;
\draw    (335,45.5) -- (335,55) ;
\draw   (215,45.5) -- (235,45.5) -- (235,60.5) -- (215,60.5) -- cycle ;
\draw    (250,70.5) -- (250,80.5) ;
\draw  [fill={rgb, 255:red, 0; green, 0; blue, 0 }  ,fill opacity=1 ] (247,70.5) .. controls (247,72.16) and (248.34,73.5) .. (250,73.5) .. controls (251.66,73.5) and (253,72.16) .. (253,70.5) .. controls (253,68.84) and (251.66,67.5) .. (250,67.5) .. controls (248.34,67.5) and (247,68.84) .. (247,70.5) -- cycle ;
\draw    (250,70.5) .. controls (240.17,66.07) and (225,71.03) .. (225,60.5) ;
\draw    (250,70.5) .. controls (260.18,65.46) and (275,71.03) .. (275,60.5) ;
\draw  [fill={rgb, 255:red, 0; green, 0; blue, 0 }  ,fill opacity=1 ] (247,80.5) .. controls (247,78.84) and (248.34,77.5) .. (250,77.5) .. controls (251.66,77.5) and (253,78.84) .. (253,80.5) .. controls (253,82.16) and (251.66,83.5) .. (250,83.5) .. controls (248.34,83.5) and (247,82.16) .. (247,80.5) -- cycle ;
\draw  [fill={rgb, 255:red, 0; green, 0; blue, 0 }  ,fill opacity=1 ] (332,55) .. controls (332,53.34) and (333.34,52) .. (335,52) .. controls (336.66,52) and (338,53.34) .. (338,55) .. controls (338,56.66) and (336.66,58) .. (335,58) .. controls (333.34,58) and (332,56.66) .. (332,55) -- cycle ;

\draw (195,52.5) node    {$=$};
\draw (150,17.5) node  [font=\scriptsize]  {$\mathsf{Pardoned}$};
\draw (115,52.5) node  [font=\scriptsize]  {$\mathsf{Answer}$};
\draw (100,17.5) node  [font=\footnotesize]  {$\mathsf{A}$};
\draw (65,52.5) node  [font=\footnotesize]  {$\mathsf{B}$};
\draw (450,18) node  [font=\scriptsize]  {$\mathsf{Pardoned}$};
\draw (415,53) node  [font=\scriptsize]  {$\mathsf{Answer}$};
\draw (400,18) node  [font=\footnotesize]  {$\mathsf{A}$};
\draw (365,53) node  [font=\footnotesize]  {$\mathsf{B}$};
\draw (495,47.5) node    {$;$};
\draw (310,18) node  [font=\scriptsize]  {$\mathsf{Pardoned}$};
\draw (275,53) node  [font=\scriptsize]  {$\mathsf{Answer}$};
\draw (260,18) node  [font=\footnotesize]  {$\mathsf{A}$};
\draw (225,53) node  [font=\footnotesize]  {$\mathsf{B}$};

\end{tikzpicture}
   \caption{Set-up for the Three Prisoners Problem, after normalization.}
  \label{fig:threePrisoners4}
\end{figure}

Perhaps unexpectedly, we conclude that Prisoner A still (only) has a $1/3$
probability of a pardon; instead, the computation shows how the chances of
prisoner $C$ have increased to $2/3$. The formal probabilistic calculus produces
this solution without relying on intuition.

\subsection{Contributions}%
\label{sec:Contributions}

Our main contribution is the algebra of \partialMarkovCategories{}
(\Cref{def:partialMarkov}) and \discretePartialMarkovCategories{}
(\Cref{def:discrete-partial-markov-cat}).
We introduce a synthetic description of \normalisation{} ($\normal{-}$, \Cref{def:normalisation}), and we prove its compositional properties, up to
\almostSureEquality{}, in the abstract setting of \partialMarkovCategories{}:
for instance, $\normal{f ⨾ g} = \normal{\normal{f} ⨾ g}$
(\Cref{prop:normalisationPrecomposes}) and, $\normal{\normal{f}} = \normal{f}$
(\Cref{prop:normalisationsIdempotent}).
We prove a synthetic version of \BayesTheorem{} that holds in any
\discretePartialMarkovCategory{} (\Cref{th:bayes}). We apply this framework to
compare \Pearls{} and \JeffreysUpdate{} rules
(\Cref{def:pearl-update,def:jeffrey-update,prop:pearlandjeffrey}). We prove that
the Kleisli category of the \emph{Maybe} monad over a \MarkovCategory{} is a
\partialMarkovCategory{} (\Cref{thm:partialMarkovMaybeMonad}).

Finally, in \Cref{sec:exactObservations}, we provide an algebra of \exactObservations{} (\Cref{def:exact-conditioning-cat}). Our main result is the construction of a \partialMarkovCategory{} on top of any \MarkovCategory{} where a synthetic Bayes' Theorem holds (\Cref{th:bayes-exact-observations}); in this construction, \exactObservations{} are expressed via \kl{conditionals} of the base \MarkovCategory{} (\Cref{prop:normalisation-of-programs,prop:conditionals-of-programs}). The \partialMarkovCategory{} of \exactObservations{} expresses deterministic observations even in nondiscrete cases, like that of continuous probabilistic processes (\Cref{prop:normalisation-of-programs,prop:conditionals-of-programs}). We illustrate these with an example.

\subsection{Related work}

\subsubsection{Extended version}
This manuscript is an extended version of ``Evidential Decision Theory via Partial Markov Categories'' \cite{dr:evidentialDecision23}, presented at the Fortieth Annual Symposium on Logic in Computer Science (LICS 2025) and authored by the two first-named authors.

This extended version employs a slightly different definition of conditional (\Cref{def:conditional}) to avoid the requirement that conditionals are quasitotal in \kl{partial Markov categories} (\Cref{sec:markovCategories,sec:PartialMarkovCategories}), essentially without affecting the rest of the mathematical development. Moreover, this extended version expands Kleisli categories of Maybe monads over distributive Markov categories (\Cref{sec:KleisliMaybe}) and adds an entirely new section on the construction of exact observations (\Cref{sec:exactObservations}).

\subsubsection{Markov categories} The categorical approach to probability theory
based on \MarkovCategories{}~\cite{fritz_2020} has led to the abstraction of
various
results~\cite{fritz2020infinite,fritz2021probability,fritz2019probability,fitz2021deFinetti}.
\MarkovCategories{} have also been applied in the formalisation of Bayes
networks and other kinds of probabilistic reasoning in categorical
terms~\cite{fong_2013, jacobs2020logical, jacobs2021causal}. The breadth of
results and applications of \MarkovCategories{} suggest that there can be an
equally rich landscape for their partial counterpart.
Simpson’s probability sheaves
constitute another approach to synthetic probability theory \cite{simpson17,simpson24}; for which Stein recently proposed a comparison \cite{stein25}.

\subsubsection{Categories of partial maps} Partiality has long been studied in
Computer Science, and even categorical approaches to it date back to the works
of Carboni~\cite{carboni1987bicategories}, Di Paola and
Heller~\cite{di1987dominical}, Robinson and
Rosolini~\cite{robinson1988categories}, and Curien and
Obtu{\l}owicz~\cite{curien_obtulowicz_1989}. %
However, our categorical structures are related to the more recent work on \emph{restriction categories} by Cockett and Lack~\cite{cockett02,cockett2003restriction}, and, in particular, \emph{cartesian restriction categories}~\cite{cockett2007restriction} and \emph{discrete cartesian restriction categories}~\cite{cockett2012range2,di_liberti_nester_2021}. 
Explicitly, we blend the string diagrammatic axiomatization of \kl{cartesian restriction categories} \cite{di_liberti_nester_2021} with that of \MarkovCategories{} \cite{cho2015introduction,fritz_2020}; in general, the result is neither a \MarkovCategory{} nor a \kl{cartesian restriction category}.

\subsubsection{Copy-discard categories}
\CopyDiscardCategories{} originally appear in graph
rewriting~\cite{corradini_gadducci_1999}, under the name of \emph{GS-monoidal}
categories. \CopyDiscardCategories{} of partial probabilistic processes—and even
the \emph{compare} morphism—have been previously
considered~\cite{panangaden1999, jacobs2018probability, cho_2019,
stein_thesis_2021}, but no comprehensive explicit presentation was given.
\BayesianInversion{} for compact closed \copyDiscardCategories{} appears in the
work of Coecke and Spekkens~\cite{coeckespekkens2012picturing}; the explicit
relationship between that definition and \Cref{def:bayes-inversion} might
involve \normalisation{}, but we leave it for further work.

\subsubsection{Categorical semantics of probabilistic programming}
There exists a vast literature on categorical semantics for probabilistic
programming languages (for the most closely related, see e.g.~\cite{hasegawa97,
stay2013bicategorical, staton_et_al_2016, heunen_kammar_staton_yang_2017,
ehrhard2017measurable, dahlqvist19semantics, vakar2019domain}). However, while
the internal language of \MarkovCategories{} has been studied
\cite{fritz_liang_2022}, the notion of \partialMarkovCategory{} and its
diagrammatic syntax have remained unexplored. Stein and
Staton~\cite{stein2021compositional, stein_thesis_2021} recently presented the
$\mathbf{Cond}$ construction for exact conditioning: in a similar fashion, we
construct a \partialMarkovCategory{} of \exactObservations{} in
\Cref{def:exact-conditioning-cat}, the relation between both again
seems to involve \normalisation{} and warrants further work.

\section{Background on Markov Categories}%
\label{sec:markovCategories}%

This section recalls background on categorical probability. Readers familiar with Markov categories are encouraged to proceed to Section 3 and refer back when needed. 

Categorical probability starts from a weakening of cartesian monoidal categories:
\kl{copy-discard categories}. \kl{Copy-discard categories} allow us to use the
same resource more than once or not at all, exactly as cartesian monoidal
categories; unlike them, they do not allow processes themselves to be
copied: in fact, copyable processes must be \kl{deterministic}. Because probabilistic
processes fail to be \kl{deterministic}, we are led to consider
\kl{copy-discard categories}—instead of cartesian categories—as a semantic basis
for probability theory.

\kl{Copy-discard categories} are an algebra of processes that compose
sequentially and in parallel, with nodes that ``copy and discard'' resources—the
inputs and outputs of processes. They posseyss a convenient sound and complete
syntax in terms of string diagrams~\cite{joyal91,fritz_liang_2022}. A particular
family of \kl{copy-discard categories}—the more specialized \kl{Markov
categories}—allow us to reason about probabilistic processes.

\subsection{Copy-Discard Categories}%
\label{sec:copyDiscardCategories}%
\defining{linkCopyDiscard}{}%
\AP \intro{Copy-discard categories} are symmetric monoidal categories where every object \m{X} has a \emph{copy} morphism, \m{ν_X ፡ X → X ⊗ X}, and a \emph{discard} morphism, \m{ε_X ፡ X → I}, forming a commutative comonoid (\Cref{diagram-copy-discard})\footnote{\CopyDiscardCategories{} have been called \emph{GS-monoidal categories} when applied to graph rewriting~\cite{corradini_gadducci_1999} and \emph{CD-categories} when applied to unnormalised probabilistic processes~\cite{cho_2019}.  In categorical quantum mechanics, they are \emph{copy and delete}~\cite{heunenvicary2019}. Fritz and Liang~\cite[Remark 2.2]{fritz_liang_2022} expand on the history of this structure.}. The comonoid on the tensor of two objects, $X ⊗ Y$, must be the
tensor of their comonoid structures: \m{ν_{X ⊗ Y} = (ν_X ⊗ ν_Y) ⨾ (\id{} ⊗ σ ⊗
\id{})} and $ε_{X ⊗ Y} = ε_X ⊗ ε_Y$. The comonoid on the unit object must be the
identity, $ν_I = \id{I}$ and $ε_I = \id{I}$.
\begin{figure}[!ht]

\tikzset{every picture/.style={line width=0.75pt}} %

\begin{tikzpicture}[x=0.75pt,y=0.75pt,yscale=-1,xscale=1]
\draw    (115,60) -- (115,45) ;
\draw  [draw opacity=0] (160,25) -- (190,25) -- (190,50) -- (160,50) -- cycle ;
\draw  [draw opacity=0] (235,25) -- (265,25) -- (265,50) -- (235,50) -- cycle ;
\draw    (290,60) -- (290,50) ;
\draw  [draw opacity=0] (325,25) -- (355,25) -- (355,50) -- (325,50) -- cycle ;
\draw  [draw opacity=0] (370,25) -- (400,25) -- (400,50) -- (370,50) -- cycle ;
\draw    (130,30) -- (130,20) ;
\draw    (130,30) .. controls (120.17,34.43) and (115,34.47) .. (115,45) ;
\draw  [fill={rgb, 255:red, 0; green, 0; blue, 0 }  ,fill opacity=1 ] (127,30) .. controls (127,28.34) and (128.34,27) .. (130,27) .. controls (131.66,27) and (133,28.34) .. (133,30) .. controls (133,31.66) and (131.66,33) .. (130,33) .. controls (128.34,33) and (127,31.66) .. (127,30) -- cycle ;
\draw    (130,30) .. controls (139.88,34.47) and (145,34.72) .. (145,45) ;
\draw    (145,45) .. controls (135.17,49.43) and (130,49.47) .. (130,60) ;
\draw  [fill={rgb, 255:red, 0; green, 0; blue, 0 }  ,fill opacity=1 ] (142,45) .. controls (142,43.34) and (143.34,42) .. (145,42) .. controls (146.66,42) and (148,43.34) .. (148,45) .. controls (148,46.66) and (146.66,48) .. (145,48) .. controls (143.34,48) and (142,46.66) .. (142,45) -- cycle ;
\draw    (145,45) .. controls (154.88,49.47) and (160,49.72) .. (160,60) ;
\draw    (220,30) -- (220,20) ;
\draw    (220,30) .. controls (210.17,34.43) and (205,34.47) .. (205,45) ;
\draw  [fill={rgb, 255:red, 0; green, 0; blue, 0 }  ,fill opacity=1 ] (217,30) .. controls (217,28.34) and (218.34,27) .. (220,27) .. controls (221.66,27) and (223,28.34) .. (223,30) .. controls (223,31.66) and (221.66,33) .. (220,33) .. controls (218.34,33) and (217,31.66) .. (217,30) -- cycle ;
\draw    (220,30) .. controls (229.88,34.47) and (235,34.72) .. (235,45) ;
\draw    (235,60) -- (235,45) ;
\draw    (205,45) .. controls (195.17,49.43) and (190,49.47) .. (190,60) ;
\draw  [fill={rgb, 255:red, 0; green, 0; blue, 0 }  ,fill opacity=1 ] (202,45) .. controls (202,43.34) and (203.34,42) .. (205,42) .. controls (206.66,42) and (208,43.34) .. (208,45) .. controls (208,46.66) and (206.66,48) .. (205,48) .. controls (203.34,48) and (202,46.66) .. (202,45) -- cycle ;
\draw    (205,45) .. controls (214.88,49.47) and (220,49.72) .. (220,60) ;
\draw    (305,35) -- (305,20) ;
\draw    (305,35) .. controls (295.17,39.43) and (290,39.47) .. (290,50) ;
\draw  [fill={rgb, 255:red, 0; green, 0; blue, 0 }  ,fill opacity=1 ] (302,35) .. controls (302,33.34) and (303.34,32) .. (305,32) .. controls (306.66,32) and (308,33.34) .. (308,35) .. controls (308,36.66) and (306.66,38) .. (305,38) .. controls (303.34,38) and (302,36.66) .. (302,35) -- cycle ;
\draw    (305,35) .. controls (314.88,39.47) and (320,39.72) .. (320,50) ;
\draw  [fill={rgb, 255:red, 0; green, 0; blue, 0 }  ,fill opacity=1 ] (317,50) .. controls (317,48.34) and (318.34,47) .. (320,47) .. controls (321.66,47) and (323,48.34) .. (323,50) .. controls (323,51.66) and (321.66,53) .. (320,53) .. controls (318.34,53) and (317,51.66) .. (317,50) -- cycle ;
\draw    (365,60) -- (365,20) ;
\draw  [draw opacity=0] (460,25) -- (490,25) -- (490,50) -- (460,50) -- cycle ;
\draw  [draw opacity=0] (520,25) -- (550,25) -- (550,50) -- (520,50) -- cycle ;
\draw    (440,30) -- (440,20) ;
\draw    (440,30) .. controls (430.17,34.43) and (425.29,34.31) .. (425,40) ;
\draw  [fill={rgb, 255:red, 0; green, 0; blue, 0 }  ,fill opacity=1 ] (437,30) .. controls (437,28.34) and (438.34,27) .. (440,27) .. controls (441.66,27) and (443,28.34) .. (443,30) .. controls (443,31.66) and (441.66,33) .. (440,33) .. controls (438.34,33) and (437,31.66) .. (437,30) -- cycle ;
\draw    (440,30) .. controls (449.88,34.47) and (454.86,34.6) .. (455,40) ;
\draw    (425,40) .. controls (425.5,50.43) and (454.83,49.77) .. (455,60) ;
\draw    (455,40) .. controls (455.5,50.43) and (424.83,49.77) .. (425,60) ;
\draw    (505,35) -- (505,20) ;
\draw    (505,35) .. controls (495.17,39.43) and (490,39.47) .. (490,50) ;
\draw  [fill={rgb, 255:red, 0; green, 0; blue, 0 }  ,fill opacity=1 ] (502,35) .. controls (502,33.34) and (503.34,32) .. (505,32) .. controls (506.66,32) and (508,33.34) .. (508,35) .. controls (508,36.66) and (506.66,38) .. (505,38) .. controls (503.34,38) and (502,36.66) .. (502,35) -- cycle ;
\draw    (505,35) .. controls (514.88,39.47) and (520,39.72) .. (520,50) ;
\draw    (490,60) -- (490,50) ;
\draw    (520,60) -- (520,50) ;

\draw (175,37.5) node    {$=$};
\draw (250,37.5) node    {$;$};
\draw (340,37.5) node    {$=$};
\draw (385,37.5) node    {$;$};
\draw (475,37.5) node    {$=$};
\draw (535,37.5) node    {$;$};
\draw (130,16.6) node [anchor=south] [inner sep=0.75pt]  [font=\tiny]  {$X$};
\draw (220,16.6) node [anchor=south] [inner sep=0.75pt]  [font=\tiny]  {$X$};
\draw (305,16.6) node [anchor=south] [inner sep=0.75pt]  [font=\tiny]  {$X$};
\draw (365,16.6) node [anchor=south] [inner sep=0.75pt]  [font=\tiny]  {$X$};
\draw (440,16.6) node [anchor=south] [inner sep=0.75pt]  [font=\tiny]  {$X$};
\draw (505,16.6) node [anchor=south] [inner sep=0.75pt]  [font=\tiny]  {$X$};
\draw (520,63.4) node [anchor=north] [inner sep=0.75pt]  [font=\tiny]  {$X$};
\draw (490,63.4) node [anchor=north] [inner sep=0.75pt]  [font=\tiny]  {$X$};
\draw (455,63.4) node [anchor=north] [inner sep=0.75pt]  [font=\tiny]  {$X$};
\draw (425,63.4) node [anchor=north] [inner sep=0.75pt]  [font=\tiny]  {$X$};
\draw (365,63.4) node [anchor=north] [inner sep=0.75pt]  [font=\tiny]  {$X$};
\draw (290,63.4) node [anchor=north] [inner sep=0.75pt]  [font=\tiny]  {$X$};
\draw (235,63.4) node [anchor=north] [inner sep=0.75pt]  [font=\tiny]  {$X$};
\draw (220,63.4) node [anchor=north] [inner sep=0.75pt]  [font=\tiny]  {$X$};
\draw (190,63.4) node [anchor=north] [inner sep=0.75pt]  [font=\tiny]  {$X$};
\draw (115,63.4) node [anchor=north] [inner sep=0.75pt]  [font=\tiny]  {$X$};
\draw (130,63.4) node [anchor=north] [inner sep=0.75pt]  [font=\tiny]  {$X$};
\draw (160,63.4) node [anchor=north] [inner sep=0.75pt]  [font=\tiny]  {$X$};

\end{tikzpicture}
   \caption{Axioms of a commutative comonoid.}
  \label{diagram-copy-discard}
\end{figure}

Some morphisms will be copied and discarded by the \emph{copy} and \emph{discard} morphisms: we call these \emph{\deterministic{}} and \emph{\total{}}, respectively. However, contrary to what happens in cartesian categories \cite{fox76}, \emph{copy} and \emph{discard} are not necessarily natural transformations: this corresponds to the fact that morphisms are not always \deterministic{} nor \total{}.

\begin{defi}[Deterministic and total morphisms]
\label{def:total_deterministic}%
\defining{linkTotalMorphism}%
\defining{linkDeterministicMorphism}%
  A morphism \m{f ፡ X → Y} in a \copyDiscardCategory{} is called
  \emph{deterministic} if $f ⨾ ν_{Y} = ν_{X} ⨾ (f ⊗ f)$
  (\Cref{diagram-total-deterministic}, left); a morphism $g ፡ X → Y$ is called
  \emph{total} if $g ⨾ ε_{Y} = ε_{X}$ (\Cref{diagram-total-deterministic},
  right).
  \begin{figure}[h!]

\tikzset{every picture/.style={line width=0.75pt}} %

\begin{tikzpicture}[x=0.75pt,y=0.75pt,yscale=-1,xscale=1]
\draw   (140,30.5) -- (160,30.5) -- (160,45.5) -- (140,45.5) -- cycle ;
\draw    (150,30) -- (150,20) ;
\draw    (150,55) -- (150,45) ;
\draw    (150,55) .. controls (140.17,59.43) and (135,59.47) .. (135,70) ;
\draw  [fill={rgb, 255:red, 0; green, 0; blue, 0 }  ,fill opacity=1 ] (147,55) .. controls (147,53.34) and (148.34,52) .. (150,52) .. controls (151.66,52) and (153,53.34) .. (153,55) .. controls (153,56.66) and (151.66,58) .. (150,58) .. controls (148.34,58) and (147,56.66) .. (147,55) -- cycle ;
\draw    (150,55) .. controls (159.88,59.47) and (165,59.72) .. (165,70) ;
\draw   (195,45) -- (215,45) -- (215,60) -- (195,60) -- cycle ;
\draw   (225,45) -- (245,45) -- (245,60) -- (225,60) -- cycle ;
\draw    (220,30) -- (220,20) ;
\draw    (220,30) .. controls (210.17,34.43) and (205,34.47) .. (205,45) ;
\draw  [fill={rgb, 255:red, 0; green, 0; blue, 0 }  ,fill opacity=1 ] (217,30) .. controls (217,28.34) and (218.34,27) .. (220,27) .. controls (221.66,27) and (223,28.34) .. (223,30) .. controls (223,31.66) and (221.66,33) .. (220,33) .. controls (218.34,33) and (217,31.66) .. (217,30) -- cycle ;
\draw    (220,30) .. controls (229.88,34.47) and (235,34.72) .. (235,45) ;
\draw    (205,70) -- (205,60) ;
\draw    (235,70) -- (235,60) ;
\draw  [draw opacity=0] (165,30) -- (195,30) -- (195,55) -- (165,55) -- cycle ;
\draw  [draw opacity=0] (245,30) -- (275,30) -- (275,55) -- (245,55) -- cycle ;
\draw   (330,30.5) -- (350,30.5) -- (350,45.5) -- (330,45.5) -- cycle ;
\draw    (340,30) -- (340,20) ;
\draw    (340,55) -- (340,45) ;
\draw  [fill={rgb, 255:red, 0; green, 0; blue, 0 }  ,fill opacity=1 ] (337,55) .. controls (337,53.34) and (338.34,52) .. (340,52) .. controls (341.66,52) and (343,53.34) .. (343,55) .. controls (343,56.66) and (341.66,58) .. (340,58) .. controls (338.34,58) and (337,56.66) .. (337,55) -- cycle ;
\draw    (395,40) -- (395,20) ;
\draw  [fill={rgb, 255:red, 0; green, 0; blue, 0 }  ,fill opacity=1 ] (392,40) .. controls (392,38.34) and (393.34,37) .. (395,37) .. controls (396.66,37) and (398,38.34) .. (398,40) .. controls (398,41.66) and (396.66,43) .. (395,43) .. controls (393.34,43) and (392,41.66) .. (392,40) -- cycle ;
\draw  [draw opacity=0] (355,30) -- (385,30) -- (385,55) -- (355,55) -- cycle ;
\draw  [draw opacity=0] (405,30) -- (435,30) -- (435,55) -- (405,55) -- cycle ;

\draw (180,42.5) node    {$=$};
\draw (150,38) node  [font=\footnotesize]  {$f$};
\draw (260,42.5) node    {$;$};
\draw (370,42.5) node    {$=$};
\draw (420,42.5) node    {$;$};
\draw (205,52.5) node  [font=\footnotesize]  {$f$};
\draw (235,52.5) node  [font=\footnotesize]  {$f$};
\draw (340,38) node  [font=\footnotesize]  {$g$};
\draw (150,16.6) node [anchor=south] [inner sep=0.75pt]  [font=\tiny]  {$X$};
\draw (220,16.6) node [anchor=south] [inner sep=0.75pt]  [font=\tiny]  {$X$};
\draw (340,16.6) node [anchor=south] [inner sep=0.75pt]  [font=\tiny]  {$X$};
\draw (395,16.6) node [anchor=south] [inner sep=0.75pt]  [font=\tiny]  {$X$};
\draw (135,73.4) node [anchor=north] [inner sep=0.75pt]  [font=\tiny]  {$Y$};
\draw (165,73.4) node [anchor=north] [inner sep=0.75pt]  [font=\tiny]  {$Y$};
\draw (205,73.4) node [anchor=north] [inner sep=0.75pt]  [font=\tiny]  {$Y$};
\draw (235,73.4) node [anchor=north] [inner sep=0.75pt]  [font=\tiny]  {$Y$};

\end{tikzpicture}
     \caption{\protect\Deterministic{} morphism (left) and \protect\total{}
    morphism (right).}%
    \label{diagram-total-deterministic}%
  \end{figure}
\end{defi}

The algebra of \copyDiscardCategories{} is already enough to discuss the
first basic concepts from probability theory: \emph{\kl{marginals}} (\Cref{def:marginal}),
\emph{\kl{conditionals}}, and \emph{\almostSureEquality{}}.

\subsection{Marginals and conditionals}

This section recasts \conditionals{} and their properties (\Cref{def:conditional}) in the partial setting; their definition needs an auxiliary operation—\kl{conditional composition}—that extends a \kl{marginal} with a given morphism (\Cref{def:conditionalExtension} and \Cref{diagram-conditionals}).

\begin{defi}[Marginals]
\defining{linkProjection}{}%
\defining{linkMarginal}{}%
\label{def:marginal}%
  \AP The \emph{projections} from every pair of objects, $X$ and $Y$, are the
  morphisms $π₁ ፡ X ⊗ Y → X$ and $π₂ ፡ X ⊗ Y → Y$ defined by the equations, $π₁ = (\id{X} ⊗ ε_Y)$ and $π₂ = (ε_X ⊗ \id{Y})$. The \intro{marginals} of a two-output morphism, $f ፡ X → Y ⊗ Z$, are the two morphisms resulting from postcomposition with the two projections, $(f ⨾ π₁) ፡ X → Y$ and $(f ⨾ π₂) ፡ X → Z$.
\end{defi}

Thus, postcomposition with the (second) projection, $\pi = (- ⨾ π₂)$, computes the
(second) \kl{marginal} of a morphism in a \copyDiscardCategory{}. This operation has
a section that brings a morphism $f ፡ X → Y$ into a morphism $φ(f) ፡ X → X ⊗ Y$
and is defined by precomposition with the \emph{copy} morphism, $φ(f) = ν_X ⨾
(\id{X} ⊗ f)$. Indeed, we may check that $\pi(φ(f)) = f$, while, in general,
$φ(\pi(f)) ≠ f$. Still, this section-retraction pair induces a new operation.

\begin{figure}[h!]

\tikzset{every picture/.style={line width=0.75pt}} %

\begin{tikzpicture}[x=0.75pt,y=0.75pt,yscale=-1,xscale=1]
\draw   (130,50) -- (170,50) -- (170,65) -- (130,65) -- cycle ;
\draw    (150,50) -- (150,40) ;
\draw    (235,25) -- (235,15) ;
\draw    (250,100) -- (250,95) ;
\draw  [draw opacity=0] (180,40) -- (210,40) -- (210,65) -- (180,65) -- cycle ;
\draw  [fill={rgb, 255:red, 255; green, 255; blue, 255 }  ,fill opacity=1 ] (260,35) -- (240,35) -- (240,50) -- (260,50) -- cycle ;
\draw  [fill={rgb, 255:red, 255; green, 255; blue, 255 }  ,fill opacity=1 ] (260,80) -- (240,80) -- (240,95) -- (260,95) -- cycle ;
\draw    (235,25) .. controls (225.17,29.43) and (220,24.47) .. (220,35) ;
\draw  [fill={rgb, 255:red, 0; green, 0; blue, 0 }  ,fill opacity=1 ] (232,25) .. controls (232,23.34) and (233.34,22) .. (235,22) .. controls (236.66,22) and (238,23.34) .. (238,25) .. controls (238,26.66) and (236.66,28) .. (235,28) .. controls (233.34,28) and (232,26.66) .. (232,25) -- cycle ;
\draw    (235,25) .. controls (244.87,29.47) and (250,24.72) .. (250,35) ;
\draw    (250,60) -- (250,50) ;
\draw    (250,60) .. controls (240.17,64.43) and (235,59.47) .. (235,70) ;
\draw  [fill={rgb, 255:red, 0; green, 0; blue, 0 }  ,fill opacity=1 ] (247,60) .. controls (247,58.34) and (248.34,57) .. (250,57) .. controls (251.66,57) and (253,58.34) .. (253,60) .. controls (253,61.66) and (251.66,63) .. (250,63) .. controls (248.34,63) and (247,61.66) .. (247,60) -- cycle ;
\draw    (250,60) .. controls (259.88,64.47) and (265.49,60.22) .. (265,70) ;
\draw    (220,60) .. controls (219.75,69.22) and (245,69.97) .. (245,80) ;
\draw    (220,45) -- (220,35) ;
\draw    (220,60) -- (220,45) ;
\draw    (235,70) .. controls (234.75,79.22) and (220,74.97) .. (220,85) ;
\draw    (265,70) .. controls (265.83,75.56) and (254.71,74.03) .. (255,80) ;
\draw    (220,85) -- (220,100) ;
\draw    (140,75) -- (140,65) ;
\draw    (160,75) -- (160,65) ;

\draw (195,52.5) node    {$=$};
\draw (150,57.5) node  [font=\footnotesize]  {$f\lhd g$};
\draw (250,42.5) node  [font=\footnotesize]  {$f$};
\draw (150,36.6) node [anchor=south] [inner sep=0.75pt]  [font=\tiny]  {$X$};
\draw (235,11.6) node [anchor=south] [inner sep=0.75pt]  [font=\tiny]  {$X$};
\draw (140,78.4) node [anchor=north] [inner sep=0.75pt]  [font=\tiny]  {$Y$};
\draw (160,78.4) node [anchor=north] [inner sep=0.75pt]  [font=\tiny]  {$Z$};
\draw (248.5,87.5) node  [font=\footnotesize]  {$g$};
\draw (220,103.4) node [anchor=north] [inner sep=0.75pt]  [font=\tiny]  {$Y$};
\draw (250,103.4) node [anchor=north] [inner sep=0.75pt]  [font=\tiny]  {$Z$};

\end{tikzpicture}
   \caption{\protect\MarginalComposition{} of $f$ and $g$.}%
  \label{diagram-conditionals}%
\end{figure}

\begin{defi}[Conditional composition]%
  \label{def:conditionalCopmosition}%
  \label{def:conditionalExtension}%
  \label{prop:conditionalExtensionAssoc}%
  \defining{linkMarginalComposition}%
  In a \copyDiscardCategory{}, the \intro{conditional composition} of a morphism
  $f ፡ X → Y$ with a morphism $g ፡ X ⊗ Y → Z$ is the morphism $(f ⊲ g) ፡ X → Y ⊗
  Z$ defined by the diagram in \Cref{diagram-conditionals} or, equivalently, by
  the formula
  $$
  f ⊲ g = π(φ(f) ⨾ φ(g)) .
  $$
\end{defi}

\begin{rem}[Type annotations]
  \label{rem:typeannotations}
  Like composition, \kl{conditional composition} ($\condcomp$) determines a
  family of functions \((\condcomp_{X,Y,Z}) ፡ ℂ(X,Y) \times ℂ(X \tensor Y,Z) \to
  ℂ(X,Y \tensor Z)\), where \(ℂ(X,Y)\) denotes the set of morphisms \(X \to Y\)
  in a \copyDiscardCategory{} \(ℂ\). Projections and copy-precomposition, $\pi$
  and $φ$, are also families of functions \(φ_{X,Y} \colon ℂ(X,Y) \to ℂ(X,X
  \tensor Y)\) and \(\pi_{X,Y} \colon ℂ(X,X \tensor Y) \to ℂ(X,Y)\). If we
  explicitly write all type annotations, the definition of conditional
  composition $(\condcomp_{X,Y,Z})$ becomes
  \[f\, ⊲_{X,Y,Z}\, g = π_{X,(Y ⊗ Z)}\,(φ_{X,Y}(f) ⨾ φ_{(X ⊗ Y), Z}(g)).\]
  Because annotating all types is cumbersome, and because they can be inferred
  from their arguments, we omit them.
\end{rem}

\begin{prop}
  \kl{Conditional composition} is associative and unital.
\end{prop}
\begin{proof}
  Let us show that \kl{conditional composition} is unital: since $φ(ε_X) =
  \id{X}$, it follows that $f ⊲ ε_{X ⊗ Y} = f$ and $ε_X ⊲ f = f$.
  Let us show that \kl{conditional composition} is also associative: by associativity of the comonoid,
  we obtain that $φ(f ⊲ g) = φ(f) ⨾ φ(g)$. Using this fact, the definition of
  \((\condcomp)\) and that \(\phi\) is a left inverse to \(\pi\), we may derive
  associativity:
  \[%
    (f ⊲ g) ⊲ h =
    π(φ((f ⊲ g) ⊲ h)) =
    π(φ(f) ⨾ φ(g) ⨾ φ(h)) =
    πφ(f ⊲ (g ⊲ h)) =
    f ⊲ (g ⊲ h). \qedhere
  \]%
\end{proof}

\AP In probability theory, we assume that every morphism can be factored through any of its
\marginals{}: for each $f ፡ X → Y ⊗ Z$, there exists some $c ፡ X ⊗ Y → Z$—its
\intro{conditional}—such that $f = (f ⨾ π₁) ⊲ c$. This assumption defines
\MarkovCategories{} and \partialMarkovCategories{}, which we introduce next.

\subsection{Markov Categories}
\label{subsec:markov-categories}

\MarkovCategories{} are a specialized variant of \copyDiscardCategories{} for probability theory. Synthetic probabilistic inference needs—apart from the axioms of \copyDiscardCategories{}—two extra principles. The first is that every morphism is \total{}: a probabilistic choice that does not affect any outcome could as well not have happened. The second is that every joint morphism can be factored through its \marginal{}: we assume the existence of a morphism such that, when conditionally composed with the \marginal{} of a morphism, yields the original morphism.
In other words, synthetic probabilistic inference uses the existence of \emph{\conditionals{}}~\cite{cho_2019}.\footnote{Markov categories~\cite{fritz_2020} appear in the literature without the requirement of having \conditionals{}—they are defined as \copyDiscardCategories{} of total maps. In this text, we call \emph{Markov categories} only to those with \conditionals{}: the \MarkovCategories{} better suited for synthetic probabilistic inference all have \conditionals{}. Our change of convention makes the parallel with cartesian categories more explicit: Markov categories are cartesian categories with a weaker splitting condition, instead of any category with copy and uniform discard maps.

The following table compares these two ways of assigning nomenclature. For our purposes, the former is more convenient and compact.
\begin{center}
\begin{tabular}{l|l}
``Markov'' assumes conditionals & ``Markov'' does not assume conditionals \\
\hline
Affine copy-discard category & Markov category \\
Copy-discard category & Copy-discard category \\
Copy-discard-compare category & Copy-discard-compare category \\
Markov category & Markov category with conditionals \\
Partial Markov category & Copy-discard category with conditionals \\
Discrete partial Markov category & Copy-discard-compare category with conditionals
\end{tabular}
\end{center}
}

\begin{defi}[Conditionals]
  \label{def:conditional}%
  \defining{linkConditional}%
  A \copyDiscardCategory{} has \intro{conditionals} if each morphism \m{f ፡
  X → Y ⊗ Z} can be factored through its \marginal{}:
  $$f = (f ⨾ π₁) ⊲ c$$%
  for some $c ፡ X ⊗ Y → Z$. In other words, for all \(f ፡ X → Y ⊗ Z\), there exists some morphism \(c ፡ X ⊗ Y → Z\) that satisfies the equation in
  \Cref{fig:conditional-definition}. In this situation, we say that $c ፡ X ⊗ Y → Z$ is a \emph{conditional} of $f$ with respect to $Y$.
\end{defi}
\begin{figure}[h!]

\tikzset{every picture/.style={line width=0.75pt}} %

\begin{tikzpicture}[x=0.75pt,y=0.75pt,yscale=-1,xscale=1]
\draw    (235,20) .. controls (244.88,24.47) and (250,19.72) .. (250,30) ;
\draw   (140,50) -- (160,50) -- (160,65) -- (140,65) -- cycle ;
\draw    (150,50) -- (150,40) ;
\draw    (235,20) -- (235,10) ;
\draw    (250,100) -- (250,95) ;
\draw  [draw opacity=0] (180,40) -- (210,40) -- (210,65) -- (180,65) -- cycle ;
\draw  [fill={rgb, 255:red, 255; green, 255; blue, 255 }  ,fill opacity=1 ] (260,30) -- (240,30) -- (240,45) -- (260,45) -- cycle ;
\draw  [fill={rgb, 255:red, 255; green, 255; blue, 255 }  ,fill opacity=1 ] (260,80) -- (240,80) -- (240,95) -- (260,95) -- cycle ;
\draw    (235,20) .. controls (225.17,24.43) and (220,19.47) .. (220,30) ;
\draw  [fill={rgb, 255:red, 0; green, 0; blue, 0 }  ,fill opacity=1 ] (232,20) .. controls (232,18.34) and (233.34,17) .. (235,17) .. controls (236.66,17) and (238,18.34) .. (238,20) .. controls (238,21.66) and (236.66,23) .. (235,23) .. controls (233.34,23) and (232,21.66) .. (232,20) -- cycle ;
\draw    (245,60) -- (245,45) ;
\draw    (245,60) .. controls (235.17,64.43) and (230,59.47) .. (230,70) ;
\draw  [fill={rgb, 255:red, 0; green, 0; blue, 0 }  ,fill opacity=1 ] (242,60) .. controls (242,58.34) and (243.34,57) .. (245,57) .. controls (246.66,57) and (248,58.34) .. (248,60) .. controls (248,61.66) and (246.66,63) .. (245,63) .. controls (243.34,63) and (242,61.66) .. (242,60) -- cycle ;
\draw    (245,60) .. controls (254.88,64.47) and (260,61.93) .. (260,70) ;
\draw    (220,65) .. controls (219.75,74.22) and (245,69.97) .. (245,80) ;
\draw    (220,45) -- (220,30) ;
\draw    (220,65) -- (220,45) ;
\draw    (230,70) .. controls (229.75,79.22) and (220,74.97) .. (220,85) ;
\draw    (260,70) .. controls (259.75,79.22) and (254.71,74.03) .. (255,80) ;
\draw    (220,85) -- (220,100) ;
\draw    (145,75) -- (145,65) ;
\draw    (155,75) -- (155,65) ;
\draw    (255,55) -- (255,45) ;
\draw  [fill={rgb, 255:red, 0; green, 0; blue, 0 }  ,fill opacity=1 ] (252,55) .. controls (252,53.34) and (253.34,52) .. (255,52) .. controls (256.66,52) and (258,53.34) .. (258,55) .. controls (258,56.66) and (256.66,58) .. (255,58) .. controls (253.34,58) and (252,56.66) .. (252,55) -- cycle ;

\draw (195,52.5) node    {$=$};
\draw (150,57.5) node  [font=\footnotesize]  {$f$};
\draw (250,37.5) node  [font=\footnotesize]  {$f$};
\draw (150,36.6) node [anchor=south] [inner sep=0.75pt]  [font=\tiny]  {$X$};
\draw (235,6.6) node [anchor=south] [inner sep=0.75pt]  [font=\tiny]  {$X$};
\draw (145,78.4) node [anchor=north] [inner sep=0.75pt]  [font=\tiny]  {$Y$};
\draw (155,78.4) node [anchor=north] [inner sep=0.75pt]  [font=\tiny]  {$Z$};
\draw (248.5,87.5) node  [font=\footnotesize]  {$c$};
\draw (220,103.4) node [anchor=north] [inner sep=0.75pt]  [font=\tiny]  {$Y$};
\draw (250,103.4) node [anchor=north] [inner sep=0.75pt]  [font=\tiny]  {$Z$};

\end{tikzpicture}
   \caption{Definition of \protect\conditional{}.}
  \label{fig:conditional-definition}
\end{figure}

\begin{defi}[Markov category]
\label{def:markovCategory}\defining{linkMarkovCategory}{}%
  \AP A \intro{Markov category} is a \copyDiscardCategory{} with \conditionals{}
  where all morphisms are \total{}.
\end{defi}

In general, \conditionals{} are not unique: unique \conditionals{} only exist in
posetal categories (\cite[Proposition 11.15]{fritz2019probability}). However,
\conditionals{} are ``\emph{\almostSurely{}} unique'' with respect to the
\marginal{}. Indeed, many notions in \MarkovCategories{} are better behaved up
to synthetic \almostSureEquality{} \cite[Section
13]{fritz2019probability}.\footnote{There exists a more general notion of
\almostSureEquality{} \cite[Definition 2.1.1]{lorenzin23:absoluteContinuity} and
\conditionalComposition{} \cite[Remark A.14]{monoidalstreams}; for simplicity,
we restrict it to the case that concerns us in this work.}

\begin{defi}[Almost-sure equality]
  \defining{linkAlmostSureEquality}%
  \label{def:almostSureEquality}%
  Two morphisms, $g₁,g₂ ፡ X ⊗ A → B$, are \almostSurelyEqual{} with respect to $f ፡ X → A$ (or, $f$-\almostSurelyEqual{}) whenever $f ⊲ g₁ = f ⊲ g₂$. In this case, we write \(g_{1} =_{f} g_{2}\). This determines an equivalence relation.
\end{defi}

In a certain sense, \conditionals{} generalize how \deterministic{} morphisms split in terms of their projections, $h = ν ⨾ ((h ⨾ π₁) ⊗ (h ⨾ π₂))$. We will denote the \linkdef{linkcond}\conditional{} of a morphism $f ፡ X → Y ⊗ Z$ by $\cond(f) ፡ X ⊗ Y → Z$. Still, note that it is defined merely up to $(f ⨾ π₁)$-\almostSureEquality{}.

\begin{prop}[Conditional of a deterministic morphism]%
  \label{prop:conditionalDeterministic}%
  Let $h ፡ X → Y ⊗ Z$ be a \deterministic{} morphism. Its \conditional{}, $\cond(h)
  ፡ X ⊗ Y → Z$, can be constructed from composing with projections:
  $$\cond(h) = π₁ ⨾ h ⨾ π₂, \mbox{ when } h \mbox{ is deterministic.}$$
\end{prop}

The proposition above shows that cartesian categories are particular
\MarkovCategories{} where all morphisms are \kl{deterministic}. These are
degenerate examples because any nondeterministic behaviour is ruled out by
naturality of the copy morphisms. We now turn to examples that do feature some
kind of nondeterminism.

\subsection{Some Markov categories}

This section recalls examples of \MarkovCategories{}: we later build partial versions of them. Intuitively, morphisms in \MarkovCategories{} encode processes with nondeterministic behaviour, but this nondeterminism is not necessarily probabilistic: the axioms of \MarkovCategories{} do not rule out possibilistic nondeterminism.
For instance, in the categories \(\Stoch\), \(\BorelStoch\), and \(\Gauss\), morphisms represent certain kinds of conditional distributions (\Cref{def:finitary-distributions,def:borelstoch,def:gauss}), while in the category \(\TRel\) morphisms are total relations between sets (\Cref{def:total-relations}).
We refer the reader to Fritz's comprehensive work~\cite{fritz_2020} for more examples and detailed constructions.

\begin{exa}[Finitely supported distributions]%
  \label{def:finitary-distributions}%
  Sets and stochastic channels between them form a \kl{copy-discard category},
  the category $\Stoch{}$. A stochastic channel, \m{f ፡ X → Y}, between two
  sets, \(X\) and \(Y\), is a function \m{f ፡ X × Y → \left[0,1\right]} such
  that, for each \(x ∈ X\), the set \(\{y \in Y : f(x,y) \neq 0\}\) is finite
  and its values add up exactly to the unit, \(\sum_{y \in Y} f(x,y) = 1\). 

  Stochastic channels represent conditional probabilities: \(f(x,y)\) encodes
  the probability of \(y\) given \(x\) according to the channel \(f\), and
  \defining{linkGiven}we usually write \(f(y \given x) = f(x,y)\). In this
  notation, composition and tensors are written as follows.
  \begin{align*}
    (f \dcomp g) (z \given x) &= {\textstyle \sum_{y \in Y} f(y \given x) \cdot g(z \given y)}, \\
    (f ⊗ f') (y,y' \given x,x') &= f(y \given x) \cdot f'(y' \given x').
  \end{align*}

  $\Stoch$ is a \copyDiscardCategory{} where all morphisms are total
  (see~\cite[Example~2.4]{cho_2019} or~\cite[Example~2.5]{fritz_2020}): in fact,
  $\Stoch$ is the Kleisli category of the finitely supported distribution monad
  \(\distr \colon \Set \to \Set\) on the category \(\Set\) of sets and
  functions, whose unit is the \defining{linkDirac}Dirac distribution, \(\dirac{x}\). 
  The copy and discard morphisms are lifted from \(\Set\). 

  \(\Stoch\) is, moreover, a \MarkovCategory{} (see~\cite[Example~3.6]{cho_2019}
  or~\cite[Example~11.6]{fritz_2020}): a \conditional{} of a stochastic
  channel \(f \colon X \to Y \tensor Z\) is given by
  \[\cond(f)(z \given x,y) = \frac{f(y,z \given x)}{\sum_{z' \in Z} f(y,z' \given x)}\]
  whenever \(\sum_{z' \in Z} f(y,z' \given x) \neq 0\); otherwise, if \(\sum_{z'
  \in Z} f(y,z' \given x) =0\), they may be given by \(\cond(f)(z \given x,y) = \sigma(z)\)
  for any arbitrary distribution \(\sigma \in \distr(Z)\). Again, note
  that \conditionals{} are not uniquely defined: they are only uniquely
  determined for pairs \((x,y)\) occurring with nonzero probability, $\sum_{z
  \in Z} f(y, z \given x) ≠ 0$.
\end{exa}

\begin{rem}[Reading string diagrams in Stoch]
  \label{rem:readingStrings}%
  \(\Stoch\) allows an intuitive reading of string diagrams with copy and
  discard: a morphism represented by a string diagram evaluates as the
  multiplication of all its components summed over all of its wires that are not
  inputs nor outputs (cf.\ \cite[Notation~2.8]{fritz_2020}).
  
  \begin{figure}[h!]

\tikzset{every picture/.style={line width=0.75pt}} %

\begin{tikzpicture}[x=0.75pt,y=0.75pt,yscale=-1,xscale=1]
\draw   (140,45) -- (160,45) -- (160,60) -- (140,60) -- cycle ;
\draw    (150,45) -- (150,20) ;
\draw  [draw opacity=0] (175,45) -- (205,45) -- (205,70) -- (175,70) -- cycle ;
\draw  [draw opacity=0] (275,45) -- (305,45) -- (305,70) -- (275,70) -- cycle ;
\draw    (140,90) -- (140,85) ;
\draw    (160,90) -- (160,85) ;
\draw   (230,30.5) -- (250,30.5) -- (250,45.5) -- (230,45.5) -- cycle ;
\draw    (240,30) -- (240,20) ;
\draw    (240,55) -- (240,45) ;
\draw    (240,55) .. controls (230.17,59.43) and (225,59.47) .. (225,70) ;
\draw  [fill={rgb, 255:red, 0; green, 0; blue, 0 }  ,fill opacity=1 ] (237,55) .. controls (237,53.34) and (238.34,52) .. (240,52) .. controls (241.66,52) and (243,53.34) .. (243,55) .. controls (243,56.66) and (241.66,58) .. (240,58) .. controls (238.34,58) and (237,56.66) .. (237,55) -- cycle ;
\draw    (240,55) .. controls (249.88,59.47) and (255,59.72) .. (255,70) ;
\draw  [fill={rgb, 255:red, 255; green, 255; blue, 255 }  ,fill opacity=1 ] (215,65) -- (235,65) -- (235,80) -- (215,80) -- cycle ;
\draw  [fill={rgb, 255:red, 255; green, 255; blue, 255 }  ,fill opacity=1 ] (245,65) -- (265,65) -- (265,80) -- (245,80) -- cycle ;
\draw    (225,90) -- (225,80) ;
\draw    (255,90) -- (255,80) ;
\draw    (155,60) .. controls (155.13,70.22) and (160.13,75.47) .. (160,85) ;
\draw    (145,60) .. controls (145.13,70.22) and (140.13,75.47) .. (140,85) ;

\draw (190,57.5) node    {$=$};
\draw (150,52.5) node  [font=\footnotesize]  {$f$};
\draw (290,57.5) node    {$;$};
\draw (150,16.6) node [anchor=south] [inner sep=0.75pt]  [font=\tiny]  {$X$};
\draw (140,93.4) node [anchor=north] [inner sep=0.75pt]  [font=\tiny]  {$Z_{1}$};
\draw (160,93.4) node [anchor=north] [inner sep=0.75pt]  [font=\tiny]  {$Z_{2}$};
\draw (240,38) node  [font=\footnotesize]  {$g$};
\draw (240,16.6) node [anchor=south] [inner sep=0.75pt]  [font=\tiny]  {$X$};
\draw (225,93.4) node [anchor=north] [inner sep=0.75pt]  [font=\tiny]  {$Z_{1}$};
\draw (225,72.5) node  [font=\footnotesize]  {$h_{1}$};
\draw (255,72.5) node  [font=\footnotesize]  {$h_{2}$};
\draw (255,93.4) node [anchor=north] [inner sep=0.75pt]  [font=\tiny]  {$Z_{2}$};
\draw (258.5,56.6) node [anchor=south] [inner sep=0.75pt]  [font=\tiny]  {$Y$};

\end{tikzpicture}
    \caption{Example of a string diagrammatic equation.}%
    \label{diagram-example-reading}%
  \end{figure}
  \noindent
  For example, \Cref{diagram-example-reading} evaluates to
  \(
    f(z₁, z₂ \given x) = \sum_{y ∈ Y}
      g(y \given x) ⋅ h_{1}(z_{1} \given y) ⋅ h_{2}(z_{2} \given y).
  \)
\end{rem}

\begin{exa}[Distributions on standard Borel spaces]\label{def:borelstoch}
  Stochastic measurable maps between measurable spaces form a
  \copyDiscardCategory{}; however, this category does not have \conditionals{}
  (see~\cite[Example~2.5]{cho_2019},\ \cite[Section~4]{fritz_2020},\ \cite{panangaden09:labelledMarkov}). Instead, its full subcategory
  on standard Borel spaces, \(\BorelStoch\), is a \MarkovCategory{}
  (see~\cite[Theorem~3.11]{cho_2019},\ \cite[Example~11.3]{fritz_2020}).
  Intuitively, a standard Borel space is a measurable space that is measurably isomorphic to either the real numbers with the Borel \(\sigma\)-algebra, the integers with the discrete \(\sigma\)-algebra or a finite discrete space.

  \defining{linkborelstoch}%
  \(\BorelStoch\) has standard Borel spaces—\((X,Σ_{X})\) for \(X\) a set and
  \(Σ_{X}\) a \(σ\)-algebra over it—as objects and Markov kernels as morphisms.
  A Markov kernel \(f ፡ (X, Σ_{X}) → (Y, Σ_{Y})\) is a function \(f ፡ Σ_{Y} × X
  → {[0,1]}\) such that, for each measurable subset of the codomain, \(T ∈
  Σ_{Y}\), its probability given an input, \(f(T \given -) ፡ X → {[0,1]}\)
  defines a measurable function with respect to the Borel \(\sigma\)-algebra on
  \([0,1]\), and such that, for each input \(x ∈ X\), the function \(f(- \given
  x) ፡ Σ_{Y} → {[0,1]}\) is a probability measure.

  Composition of two morphisms \(f ፡ (X, Σ_{X}) → (Y, Σ_{Y})\) and \(g ፡ (Y,
  Σ_{Y}) → (Z, Σ_{Z})\) is given by Lebesgue integration. Identities
  \(\id{(X,\Sigma_{X})}\) are given by Dirac measures, as below.
  \begin{align*} 
    (f \dcomp g) (T \given x) = ∫_{y ∈ Y} g(T \given y) ⋅ f(dy \given x);
    \qquad
    \id{(X,\Sigma_{X})} (T \given x) =
      \begin{cases}
        1 & \text{ if } x ∈ T, \\
        0 & \text{ otherwise. }
      \end{cases}
  \end{align*}
  
  \(\BorelStoch\) is monoidal. The tensor of two standard Borel spaces, \((X,
  Σ_{X})\) and \((Y, Σ_{Y})\), is the cartesian product of the
  underlying sets, \(X × Y\), with the \(\sigma\)-algebra generated by the measurable rectangles, \(Σ_X ⊗ Σ_Y = \langle S₁ × S₂ \given S₁ ∈ Σ_X, S₂ ∈ Σ_Y \rangle\).
  Morphisms are tensored by integration with respect to the product measure, 
  \[%
  (f ⊗ f')(T \given x,x') = \int_{(y,y') ∈ T} f(dy \given x) ⋅ f'(dy' \given x').\]
  \(\BorelStoch\) is also the Kleisli category of the \defining{linkgiry}{Giry
  monad}, \(\Giry \colon \BMeas \to \BMeas\), on the
  \defining{linkBmeas}cartesian category \(\BMeas\) of standard Borel spaces and
  measurable functions between them~\cite{giry82:categorical}. As in the
  finitely supported case, the monoidal and \kl{copy-discard} structure of
  \(\BorelStoch\) is lifted from \(\BMeas\).
\end{exa}

Let us exemplify how not all \kl{Markov categories} are Kleisli categories of monads on their subcategory of deterministic morphisms: the category of stochastic matrices and the category of affine maps with Gaussian noise do not arise this way.

\begin{exa}[Stochastic matrices]\label{def:finstoch}%
  \intro[finstoch]{}
  The full subcategory of \(\Stoch\) on finite sets is a \kl{Markov category}~\cite[Example~2.5]{fritz_2020}, \(\FinStoch\).
  Its morphisms may be described as matrices with nonnegative real entries whose rows sum to \(1\).  
\end{exa}

\begin{exa}[Affine maps with Gaussian noise]\label{def:gauss}%
  \intro[Gauss]{}
  Affine maps with Gaussian noise form a Markov subcategory of \(\BorelStoch\), denoted as \(\Gauss\), that has a convenient presentation~\cite[Section~6]{fritz_2020}: objects are natural numbers, where
  each \(n \in \naturals\) represents the measurable space \(\reals^{n}\) of
  \(n\)-dimensional real vectors with the Borel \(\sigma\)-algebra; morphisms
  \(n \to m\) are specified by a triple \((M,C,s)\) of
  real valued matrices \(M \in \reals^{m} \times \reals^{n}\), \(C \in
  \reals^{m} \times \reals^{m}\) and \(s \in \reals^{m}\), with \(C\) positive semidefinite.

  A morphism \((M,C,s) \colon n \to m\) specifies an affine transformation of an input random variable, \(X \in \reals^{n}\), with added Gaussian noise, \(Z \sim \mathcal{N}(s,C)\), of mean \(s\) and covariance \(C\): the output random variable \(Y \in \reals^{m}\) is obtained by \(Y = MX + Z\); morphisms compose as
  \begin{align*}
    (M,C,s) \dcomp (N,D,t) &= (NM,\, NC\transpose{N}\!+D,\, Ns + t), \\ 
    (M,C,s) \tensor (M',C',s') &= (M \oplus M', C \oplus C', s \oplus s'),
  \end{align*}
   where \(\transpose{(-)}\) denotes the matrix transposition and \((\oplus)\) denotes the biproduct of matrices.
  
  \Conditionals{} can be obtained in terms of familiar matrix
  operations~\cite[Example~11.8]{fritz_2020}: because any morphism \(f \colon n
  \to m + l\) can be written in block form, \(f = \left(\begin{psmallmatrix} M
  \\ N \end{psmallmatrix},
  \begin{psmallmatrix} C & U \\ V & D \end{psmallmatrix},
  \begin{psmallmatrix} s \\ t \end{psmallmatrix} \right)\), the \conditional{} of
  \(f\) can be defined as  
  \[\cond(f) = \left( (N - VC^{-}\!M \mid VC^{-} ),\, D -
  VC^{-}\!U,\, t - VC^{-}\!s \right),\] 
  where \(C^{-}\) denotes the Moore-Penrose pseudoinverse of \(C\), and $(N -
  VC^{-}\!M \mid VC^{-} )$ denotes the block matrix constructed by juxtaposing
  $N - VC^{-}\!M$ and $VC^{-}$~\cite{samuelson23,samuelson24}.
\end{exa}

We finally exemplify how \MarkovCategories{} do not rule out relational nondeterminism: the category of sets and total (or nonempty) relations is a \MarkovCategory{}~\cite[§2.6]{fritz_2020}.

\begin{exa}[Total relations]%
  \label{def:total-relations}%
  The category of total relations, $\TRel$, has sets as objects.
  For two sets, \(X\) and \(Y\), a total relation \(f \colon X \to Y\) is a
  function that maps each element \(x \in X\) to a nonempty subset \(f(x)
  \subseteq Y\) of the codomain. The subset \(f(x)\) represents the set of
  possible outcomes of the process \(f\) without giving any information about
  how likely each outcome is. Composition is given by relational composition and
  the tensor is the point-wise cartesian product.
  \begin{align*}
    f \dcomp g(x) &= \{z ∈ Z : \exists y ∈ Y \ y ∈ f(x) \land z ∈ g(y)\}, \\
    f \tensor f'(x,x') &= f(x) × f'(x').
  \end{align*}
  
  The conditional of a total relation $f ፡ X → Y × Z$ is the morphism
  $\cond(f) ፡ X × Y → Z$ defined by $\cond(f)(x,y) = \{z \in Z : (y,z) \in
  f(x)\}$ whenever it is nonempty, and by \(\cond(f)(x,y) = T\) for any nonempty subset
  \(T \subseteq Z\) otherwise.

  \(\TRel\) is the Kleisli category of the nonempty powerset monad \(\nePowerset
  \colon \Set \to \Set\) on the category \(\Set\) of sets and functions.
\end{exa}

\subsection{Subdistributions}%
\label{ex:non-example-markov}%

Still, not normalized probability theory is not captured by \MarkovCategories{}:
let us present the main counterexample. A \emph{subdistribution} over \(X\) is a
distribution whose total probability mass is allowed to be less than
\(1\)~\cite{jacobs2018probability,cho_2019}. In other words, it is a
distribution over \(\maybe[X]\), where we may interpret the probability of
\(\bot\), the only element of \(1\), as the probability of ``failure''. 

\begin{defi}[Category of subdistributions]
  \defining{linkSubStoch} 
  The category of subdistributions, \(\subStoch\), has sets as objects; morphisms, $f ፡ X → Y$, are given by functions \(X × Y → \left[ 0,1 \right]\) such that, for all elements \(x ∈ X\), their support, \(\{y ∈ Y : f(x,y) ≠ 0\}\), is finite and their total probability mass adds up at most to the unit, \(\sum_{y ∈ Y} f(x,y) ≤ 1\). 
  
  Compositions and identities are defined in the same way as for \(\Stoch\) (\Cref{def:finitary-distributions}).
\end{defi}

\begin{rem}[Subdistributions as a Kleisli category]
  The category of \kl{subdistributions}, $\subStoch$, is a Kleisli category. It
  is both the Kleisli category of the Maybe monad on the category of
  distributions, \((\maybe) ፡ \Stoch → \Stoch\), and the Kleisli category of the
  \defining{linksubdistr}{finitary subdistribution monad}, \(\subdistr =
  \distr(\maybe) \colon \Set \to \Set\). 
\end{rem}

\begin{prop}[e.g.,~{{\cite[Section 4]{jacobs2018probability}}}]
  There exists a distributive law, $\distributivelaw ፡ \maybe[\distr(-)] →
  \distr(\maybe)$, between the Maybe monad and the finitary distribution monad,
  defined by 
  \[\distributivelaw_{X}(σ)(x) = σ(x)\mbox{ with }
  \distributivelaw_{X}(σ)(⊥) = 0, \mbox{ and }
  \distributivelaw_{X}(\bot)(x) = 0 \mbox{ with }
  \distributivelaw_{X}(\bot)(⊥) = 1.\]
\end{prop}

\noindent This construction ensures that the category of \kl{subdistributions},
\(\subStoch\), is a \copyDiscardCategory{}. Moreover, it has 
\conditionals{} given by
\[\cond(f)(z \given x,y) = \frac{f(y,z \given x)}{\sum_{z' \in Z} f(y,z' \given
x)}\] whenever \(\sum_{z' \in Z} f(y,z' \given x) \neq 0\), and by \(\cond(f)(z
\given x,y) = 0\) whenever \(\sum_{z' \in Z} f(y,z' \given x) =0\). 

More generally, many categories of conditional subdistributions are \copyDiscardCategories{} with \conditionals{}, but they are not \MarkovCategories{} because not all their morphisms are \total{}.

\begin{rem}
  The existence of \conditionals{} in the category of \kl{subdistributions} also follows from a general argument that will be the subject of \Cref{sec:KleisliMaybe}. Indeed, similar constructions can be applied to most of the previous examples: in \Cref{sec:examples-partial-markov}, we will obtain that measurable substochastic channels, Gaussian substochastic channels, and different categories of relations all have \conditionals{}. 
\end{rem}

We shall show that totality is not essential for modelling stochastic processes and that dropping this assumption allows us to pursue a synthetic theory of inferential update and observations. This is the main idea behind \emph{\partialMarkovCategories{}}.

\section{Partial Markov Categories}%
\label{sec:PartialMarkovCategories}

Cartesian restriction categories~\cite{cockett2007restriction} extend cartesian categories by encoding
\emph{partiality}: a map may not be defined on all its inputs and fail when
evaluated on inputs outside its domain of definition. %
We introduce \emph{\partialMarkovCategories{}} as a similar extension of
\MarkovCategories{} to encode partial stochastic processes, i.e., stochastic
processes that have a probability of failure on each one of their inputs.
Partiality is obtained by dropping naturality of the discard maps, i.e.\ by
allowing morphisms to be nontotal. 

This extra generality will lead us later to \discretePartialMarkovCategories{}:
\partialMarkovCategories{} with the ability of comparing two outputs. These are
the analogue of \kl{discrete cartesian restriction
categories}~\cite{cockett2012range2,di_liberti_nester_2021}.
\begin{figure}[h!]
  \begin{center}
    \begin{tabular}{l|l}
      Deterministic & Non-deterministic \\
      \hline
      Cartesian category & \MarkovCategory{} \\ 
      \kl{Cartesian restriction category} & \PartialMarkovCategory{} \\
      \kl{Discrete cartesian restriction category} & \DiscretePartialMarkovCategory{} \\
    \end{tabular}
  \end{center}
  \caption{The structures on the left column are obtained by those on the right column by imposing that all morphisms are deterministic (see \Cref{prop:cartesian-restriction}).}%
  \label{fig:cartesian-restriction-partial-markov}
\end{figure}

\begin{defi}[Partial Markov category]
  \defining{linkPartialMarkov}%
  \label{def:partialMarkov}%
  A \intro{partial Markov category} is a \copyDiscardCategory{} with
  \conditionals{}. 
\end{defi}

In a \MarkovCategory{}, all morphisms—and in particular all \conditionals{}—are
\total{}. 
In \partialMarkovCategories{}, we drop the totality assumption while still
obtaining marginals by discarding one of the outputs. %
Instead of being \total{}, \conditionals{} in \partialMarkovCategories{} are \almostSurely{} \total{}.

\begin{prop}%
  \Conditionals{} in \partialMarkovCategories{} are \almostSurely{} \total{}
  with respect to the \marginal{}.
\end{prop}
\begin{proof}
  Let us show that, given $f ፡ X → Y ⊗ Z$, its \conditional{}, $\cond(f) ፡ X ⊗ Y → Z$,
  is \((f ⨾ π₁)\)-\almostSurely{} \total{}, meaning that $(f ⨾ π₁) ⊲ (\cond(f) ⨾ ε) =
  (f ⨾ π₁) ⊲ ε$. By \emph{(i)} the definition of projection, \emph{(ii)} the definition
  of \conditional{}, and \emph{(iii)} counitality, we have that
  \[(f ⨾ π₁) ⊲ (\cond(f) ⨾ ε) \overset{(i)}{=} 
  ((f ⨾ π₁) ⊲ \cond(f)) ⨾ π₁ \overset{(ii)}{=} 
  f ⨾ π₁ \overset{(iii)}{=} 
  (f ⨾ π₁) ⊲ ε.\]

  \begin{figure}[h!]

\tikzset{every picture/.style={line width=0.75pt}} %

\begin{tikzpicture}[x=0.75pt,y=0.75pt,yscale=-1,xscale=1]
\draw    (100,25) .. controls (109.88,29.47) and (115,24.72) .. (115,35) ;
\draw    (100,25) -- (100,15) ;
\draw    (115,110) -- (115,100) ;
\draw  [fill={rgb, 255:red, 255; green, 255; blue, 255 }  ,fill opacity=1 ] (125,35) -- (105,35) -- (105,50) -- (125,50) -- cycle ;
\draw  [fill={rgb, 255:red, 255; green, 255; blue, 255 }  ,fill opacity=1 ] (125,85) -- (105,85) -- (105,100) -- (125,100) -- cycle ;
\draw    (100,25) .. controls (90.17,29.43) and (85,24.47) .. (85,35) ;
\draw  [fill={rgb, 255:red, 0; green, 0; blue, 0 }  ,fill opacity=1 ] (97,25) .. controls (97,23.34) and (98.34,22) .. (100,22) .. controls (101.66,22) and (103,23.34) .. (103,25) .. controls (103,26.66) and (101.66,28) .. (100,28) .. controls (98.34,28) and (97,26.66) .. (97,25) -- cycle ;
\draw    (110,65) -- (110,50) ;
\draw    (110,65) .. controls (100.17,69.43) and (95,64.47) .. (95,75) ;
\draw  [fill={rgb, 255:red, 0; green, 0; blue, 0 }  ,fill opacity=1 ] (107,65) .. controls (107,63.34) and (108.34,62) .. (110,62) .. controls (111.66,62) and (113,63.34) .. (113,65) .. controls (113,66.66) and (111.66,68) .. (110,68) .. controls (108.34,68) and (107,66.66) .. (107,65) -- cycle ;
\draw    (110,65) .. controls (119.88,69.47) and (125,66.93) .. (125,75) ;
\draw    (85,70) .. controls (84.75,79.22) and (110,74.97) .. (110,85) ;
\draw    (85,50) -- (85,35) ;
\draw    (85,70) -- (85,50) ;
\draw    (95,75) .. controls (94.75,84.22) and (85,79.97) .. (85,90) ;
\draw    (125,75) .. controls (124.75,84.22) and (119.71,79.03) .. (120,85) ;
\draw    (85,90) -- (85,115) ;
\draw    (120,60) -- (120,50) ;
\draw  [fill={rgb, 255:red, 0; green, 0; blue, 0 }  ,fill opacity=1 ] (117,60) .. controls (117,58.34) and (118.34,57) .. (120,57) .. controls (121.66,57) and (123,58.34) .. (123,60) .. controls (123,61.66) and (121.66,63) .. (120,63) .. controls (118.34,63) and (117,61.66) .. (117,60) -- cycle ;
\draw  [fill={rgb, 255:red, 0; green, 0; blue, 0 }  ,fill opacity=1 ] (112,110) .. controls (112,108.34) and (113.34,107) .. (115,107) .. controls (116.66,107) and (118,108.34) .. (118,110) .. controls (118,111.66) and (116.66,113) .. (115,113) .. controls (113.34,113) and (112,111.66) .. (112,110) -- cycle ;
\draw  [draw opacity=0] (190,46.5) -- (220,46.5) -- (220,71.5) -- (190,71.5) -- cycle ;
\draw  [fill={rgb, 255:red, 255; green, 255; blue, 255 }  ,fill opacity=1 ] (165,45) -- (185,45) -- (185,60) -- (165,60) -- cycle ;
\draw    (175,45) -- (175,25) ;
\draw    (170,110) -- (170,60) ;
\draw    (180,70) -- (180,60) ;
\draw  [fill={rgb, 255:red, 0; green, 0; blue, 0 }  ,fill opacity=1 ] (177,70) .. controls (177,68.34) and (178.34,67) .. (180,67) .. controls (181.66,67) and (183,68.34) .. (183,70) .. controls (183,71.66) and (181.66,73) .. (180,73) .. controls (178.34,73) and (177,71.66) .. (177,70) -- cycle ;
\draw  [draw opacity=0] (130,46.5) -- (160,46.5) -- (160,71.5) -- (130,71.5) -- cycle ;
\draw    (240,25) .. controls (249.88,29.47) and (255,24.72) .. (255,35) ;
\draw    (240,25) -- (240,15) ;
\draw    (245,90) -- (245,85) ;
\draw  [fill={rgb, 255:red, 255; green, 255; blue, 255 }  ,fill opacity=1 ] (265,35) -- (245,35) -- (245,50) -- (265,50) -- cycle ;
\draw    (240,25) .. controls (230.17,29.43) and (225,24.47) .. (225,35) ;
\draw  [fill={rgb, 255:red, 0; green, 0; blue, 0 }  ,fill opacity=1 ] (237,25) .. controls (237,23.34) and (238.34,22) .. (240,22) .. controls (241.66,22) and (243,23.34) .. (243,25) .. controls (243,26.66) and (241.66,28) .. (240,28) .. controls (238.34,28) and (237,26.66) .. (237,25) -- cycle ;
\draw    (250,65) -- (250,50) ;
\draw    (250,65) .. controls (240.17,69.43) and (235,64.47) .. (235,75) ;
\draw  [fill={rgb, 255:red, 0; green, 0; blue, 0 }  ,fill opacity=1 ] (247,65) .. controls (247,63.34) and (248.34,62) .. (250,62) .. controls (251.66,62) and (253,63.34) .. (253,65) .. controls (253,66.66) and (251.66,68) .. (250,68) .. controls (248.34,68) and (247,66.66) .. (247,65) -- cycle ;
\draw    (250,65) .. controls (259.88,69.47) and (265,66.93) .. (265,75) ;
\draw    (225,70) .. controls (224.75,79.22) and (245,74.97) .. (245,85) ;
\draw    (225,50) -- (225,35) ;
\draw    (225,70) -- (225,50) ;
\draw    (235,75) .. controls (234.75,84.22) and (225,79.97) .. (225,90) ;
\draw    (225,90) -- (225,115) ;
\draw    (260,60) -- (260,50) ;
\draw  [fill={rgb, 255:red, 0; green, 0; blue, 0 }  ,fill opacity=1 ] (257,60) .. controls (257,58.34) and (258.34,57) .. (260,57) .. controls (261.66,57) and (263,58.34) .. (263,60) .. controls (263,61.66) and (261.66,63) .. (260,63) .. controls (258.34,63) and (257,61.66) .. (257,60) -- cycle ;
\draw  [fill={rgb, 255:red, 0; green, 0; blue, 0 }  ,fill opacity=1 ] (242,90) .. controls (242,88.34) and (243.34,87) .. (245,87) .. controls (246.66,87) and (248,88.34) .. (248,90) .. controls (248,91.66) and (246.66,93) .. (245,93) .. controls (243.34,93) and (242,91.66) .. (242,90) -- cycle ;
\draw    (265,80) -- (265,75) ;
\draw  [fill={rgb, 255:red, 0; green, 0; blue, 0 }  ,fill opacity=1 ] (262,80) .. controls (262,78.34) and (263.34,77) .. (265,77) .. controls (266.66,77) and (268,78.34) .. (268,80) .. controls (268,81.66) and (266.66,83) .. (265,83) .. controls (263.34,83) and (262,81.66) .. (262,80) -- cycle ;

\draw (115,42.5) node  [font=\footnotesize]  {$f$};
\draw (100,11.6) node [anchor=south] [inner sep=0.75pt]  [font=\tiny]  {$X$};
\draw (113.5,92.5) node  [font=\footnotesize]  {$c$};
\draw (85,118.4) node [anchor=north] [inner sep=0.75pt]  [font=\tiny]  {$Y$};
\draw (175,52.5) node  [font=\footnotesize]  {$f$};
\draw (170,113.4) node [anchor=north] [inner sep=0.75pt]  [font=\tiny]  {$Y$};
\draw (175,21.6) node [anchor=south] [inner sep=0.75pt]  [font=\tiny]  {$X$};
\draw (145,59) node    {$\overset{\mathit{(i)}}{=}$};
\draw (205,59) node    {$\overset{\mathit{(ii)}}{=}$};
\draw (255,42.5) node  [font=\footnotesize]  {$f$};
\draw (240,11.6) node [anchor=south] [inner sep=0.75pt]  [font=\tiny]  {$X$};
\draw (225,118.4) node [anchor=north] [inner sep=0.75pt]  [font=\tiny]  {$Y$};

\end{tikzpicture}
     \caption{Conditionals are \protect\almostSurely{} \protect\total{}.}
    \label{fig:conditionals-almost-surely-total}
  \end{figure}
  Equivalently, the string diagram reasoning in
  \Cref{fig:conditionals-almost-surely-total} shows that the \conditional{} is
  \((f ⨾ π_1)\)-\almostSurely{} \total. We employ \emph{(i)} the definition of
  \conditional{}, and \emph{(ii)} counitality.
\end{proof}

\Conditionals{} can be computed compositionally: similar results appear in the work of Fritz~\cite[Lemma 11.11]{fritz_2020}, for \MarkovCategories{}, and Jacobs~\cite[Section 5.1]{jacobs_2019}, for \(\Stoch\).
We recast these in the setting of \partialMarkovCategories{}.

\subsection{Bayesian inversions}

The \emph{\BayesianInversion{}} of a stochastic channel \m{g ፡ X → Y} with
respect to a distribution \m{p} over \m{X} is the stochastic channel
\(\bayesinv{g}{p} ፡ Y → X\) defined by 
\mm{%
\bayesinv{g}{p}(x \given y) = \frac{g(y \given x) · p(x)}%
{\sum_{x_• ∈ X} g(y \given x_•) · p(x_•)},%
}%
for any \(y ∈ Y\) with positive probability, meaning that \(\sum_{x_{•} ∈
X} g(y \given x_{•}) · p(x_{•}) > 0\).

\BayesianInversions{} can be defined abstractly in \partialMarkovCategories{},
as they can be in \MarkovCategories{}~\cite[Proposition 11.17]{fritz_2020}:
\BayesianInversions{} are a particular case of \conditionals{}. We state this
result for \partialMarkovCategories{}
(\Cref{prop:bayes-inversions-from-conditionals}) as a straightforward
generalisation of that for \MarkovCategories{}~\cite[Proposition 11.17]{fritz_2020}.

\begin{defi}[Bayesian inversion]
  \label{def:bayes-inversion}%
  \defining{linkbayesinv}%
  A \emph{Bayesian inversion} of a morphism \m{g ፡ X → Y} with respect to \(p ፡
  I → X\) is a morphism \(\bayesinv{g}{p} ፡ Y → X\)
  satisfying the equation in \Cref{diagram-bayesian-inversion}. 
  \begin{figure}[h!]

\tikzset{every picture/.style={line width=0.75pt}} %

\begin{tikzpicture}[x=0.75pt,y=0.75pt,yscale=-1,xscale=1]
\draw   (140,15.5) -- (160,15.5) -- (160,30.5) -- (140,30.5) -- cycle ;
\draw    (135,70) -- (135,65) ;
\draw    (150,40) -- (150,30) ;
\draw    (150,40) .. controls (140.17,44.43) and (135,39.47) .. (135,50) ;
\draw  [fill={rgb, 255:red, 0; green, 0; blue, 0 }  ,fill opacity=1 ] (147,40) .. controls (147,38.34) and (148.34,37) .. (150,37) .. controls (151.66,37) and (153,38.34) .. (153,40) .. controls (153,41.66) and (151.66,43) .. (150,43) .. controls (148.34,43) and (147,41.66) .. (147,40) -- cycle ;
\draw    (150,40) .. controls (159.88,44.47) and (165,39.72) .. (165,50) ;
\draw   (125,50) -- (145,50) -- (145,65) -- (125,65) -- cycle ;
\draw    (165,70) -- (165,50) ;
\draw   (210,15) -- (230,15) -- (230,30) -- (210,30) -- cycle ;
\draw    (220,60) -- (220,50) ;
\draw    (220,35) -- (220,30) ;
\draw    (220,60) .. controls (210.17,64.43) and (205,59.47) .. (205,70) ;
\draw  [fill={rgb, 255:red, 0; green, 0; blue, 0 }  ,fill opacity=1 ] (217,60) .. controls (217,58.34) and (218.34,57) .. (220,57) .. controls (221.66,57) and (223,58.34) .. (223,60) .. controls (223,61.66) and (221.66,63) .. (220,63) .. controls (218.34,63) and (217,61.66) .. (217,60) -- cycle ;
\draw    (220,60) .. controls (229.88,64.47) and (235,59.72) .. (235,70) ;
\draw   (210,35) -- (230,35) -- (230,50) -- (210,50) -- cycle ;
\draw    (205,90) -- (205,70) ;
\draw   (220,70) -- (250,70) -- (250,85) -- (220,85) -- cycle ;
\draw    (235,90) -- (235,85) ;
\draw    (165,90) -- (165,70) ;
\draw    (135,90) -- (135,70) ;
\draw  [draw opacity=0] (165,40) -- (205,40) -- (205,55) -- (165,55) -- cycle ;

\draw (150,23) node  [font=\footnotesize]  {$p$};
\draw (135,57.5) node  [font=\footnotesize]  {$g$};
\draw (220,22.5) node  [font=\footnotesize]  {$p$};
\draw (220,42.5) node  [font=\footnotesize]  {$g$};
\draw (235,77.5) node  [font=\footnotesize]  {$g_{\dagger }( p)$};
\draw (185,52.5) node  [font=\footnotesize]  {$=$};

\end{tikzpicture}
     \caption{\protect\BayesianInversion{}.}%
    \label{diagram-bayesian-inversion}%
  \end{figure}
\end{defi}

\begin{prop}
  \label{prop:bayes-inversions-from-conditionals}
  In a \partialMarkovCategory{}, all \BayesianInversions{} exist. The
  \BayesianInversion{} of $g$ with respect to $p$ is $(p ⨾ g)$-\almostSurely{} unique.
\end{prop}
\begin{proof}
  The \BayesianInversion{} of $g$ with respect to $p$ is given by
  the \conditional{} of the morphism $p ⨾ ν_{X} ⨾ (g ⊗ \id{})$. Explicitly,
  $$\bayesinv{g}{p} = \cond(p ⨾ ν_{X} ⨾ (g ⊗ \id{})).$$
  Note that $(p ⨾ ν_{X} ⨾ (g ⊗ \id{}) ⨾ π₁) = (p ⨾ g),$
  making the \conditional{} 
  $(p ⨾ g)$-\almostSurely{} unique.
\end{proof}

\BayesianInversions{} can be computed compositionally, in the same way
\conditionals{} can
(\Cref{prop:bayes-inversions-from-conditionals,prop:conditional-composite-1}).

\begin{prop}[Conditional of a composite]%
  \label{prop:conditional-composite-1}%
  Let $f ፡ X → Y$ and $g ፡ Y → Z ⊗ W$ be two morphisms of a
  \partialMarkovCategory{}. The \conditional{} of their composition can be
  computed from a \BayesianInversion{}, $\pmb{b}(f,g) = \bayesinv{(g ⨾ \pi_1)}{f}$, and
  the \conditional{} of the second morphism, $\cond(g) ፡ Y ⊗ Z → W$, as in
  \Cref{fig:conditionalComposition}.
\end{prop}
\begin{figure}[h!]

\tikzset{every picture/.style={line width=0.75pt}} %

\begin{tikzpicture}[x=0.75pt,y=0.75pt,yscale=-1,xscale=1]
\draw   (230,45) -- (280,45) -- (280,60) -- (230,60) -- cycle ;
\draw    (255,70) -- (255,60) ;
\draw    (245,45) -- (245,35) ;
\draw    (265,45) -- (265,35) ;
\draw   (335,30) -- (385,30) -- (385,45) -- (335,45) -- cycle ;
\draw    (350,30) -- (350,10) ;
\draw    (385,20) -- (385,10) ;
\draw    (385,20) .. controls (375.17,24.43) and (370,19.47) .. (370,30) ;
\draw  [fill={rgb, 255:red, 0; green, 0; blue, 0 }  ,fill opacity=1 ] (382,20) .. controls (382,18.34) and (383.34,17) .. (385,17) .. controls (386.66,17) and (388,18.34) .. (388,20) .. controls (388,21.66) and (386.66,23) .. (385,23) .. controls (383.34,23) and (382,21.66) .. (382,20) -- cycle ;
\draw    (385,20) .. controls (394.88,24.47) and (400,19.72) .. (400,30) ;
\draw   (365,60) -- (395,60) -- (395,75) -- (365,75) -- cycle ;
\draw    (380,85) -- (380,75) ;
\draw    (360,45) .. controls (360,55.22) and (370,50.72) .. (370,60) ;
\draw    (400,45) .. controls (400,55.22) and (390,50.72) .. (390,60) ;
\draw    (400,45) -- (400,30) ;
\draw  [draw opacity=0] (300,40) -- (325,40) -- (325,65) -- (300,65) -- cycle ;
\draw  [draw opacity=0] (400,40) -- (425,40) -- (425,65) -- (400,65) -- cycle ;

\draw (255,52.5) node  [font=\footnotesize]  {$\cond(f⨾g)$};
\draw (360,37.5) node  [font=\footnotesize]  {$\pmb{b}(f,g)$};
\draw (380,67.5) node  [font=\footnotesize]  {$\cond(g)$};
\draw (312.5,52.5) node  [font=\footnotesize]  {$=_{f ⨾ g ⨾ \pi _{1}}$};
\draw (412.5,52.5) node  [font=\footnotesize]  {$;$};

\end{tikzpicture}
   \caption{\protect\Conditional{} of a composition.}
  \label{fig:conditionalComposition}
\end{figure}
\begin{proof}
  We reason with string diagrams as in \Cref{fig:proof-conditional-composition}. We use \emph{(i)} the definition of
  \conditional{} for \(g\); \emph{(ii)} the definition of \BayesianInversion{} of \(g \dcomp \pi_{1}\) with respect to \(f\); and
  \emph{(iii)} associativity and commutativity.
  \begin{figure}[h!]

\tikzset{every picture/.style={line width=0.75pt}} %

\begin{tikzpicture}[x=0.75pt,y=0.75pt,yscale=-1,xscale=1]
\draw  [draw opacity=0] (210,60) -- (235,60) -- (235,85) -- (210,85) -- cycle ;
\draw  [draw opacity=0] (390.01,60) -- (415.01,60) -- (415.01,85) -- (390.01,85) -- cycle ;
\draw   (110,50) -- (140,50) -- (140,65) -- (110,65) -- cycle ;
\draw   (110,75) -- (140,75) -- (140,90) -- (110,90) -- cycle ;
\draw    (125,50) -- (125,40) ;
\draw    (125,75) -- (125,65) ;
\draw    (115,100) -- (115,90) ;
\draw    (135,100) -- (135,90) ;
\draw   (165,35) -- (195,35) -- (195,50) -- (165,50) -- cycle ;
\draw    (180,35) -- (180,25) ;
\draw    (180,55) -- (180,50) ;
\draw    (180,55) .. controls (170.17,59.43) and (165,54.47) .. (165,65) ;
\draw  [fill={rgb, 255:red, 0; green, 0; blue, 0 }  ,fill opacity=1 ] (177,55) .. controls (177,53.34) and (178.34,52) .. (180,52) .. controls (181.66,52) and (183,53.34) .. (183,55) .. controls (183,56.66) and (181.66,58) .. (180,58) .. controls (178.34,58) and (177,56.66) .. (177,55) -- cycle ;
\draw    (180,55) .. controls (189.88,59.47) and (195,54.72) .. (195,65) ;
\draw   (180,65) -- (210,65) -- (210,80) -- (180,80) -- cycle ;
\draw    (185,95) -- (185,80) ;
\draw    (205,90) -- (205,80) ;
\draw  [fill={rgb, 255:red, 0; green, 0; blue, 0 }  ,fill opacity=1 ] (202,90) .. controls (202,88.34) and (203.34,87) .. (205,87) .. controls (206.66,87) and (208,88.34) .. (208,90) .. controls (208,91.66) and (206.66,93) .. (205,93) .. controls (203.34,93) and (202,91.66) .. (202,90) -- cycle ;
\draw  [fill={rgb, 255:red, 0; green, 0; blue, 0 }  ,fill opacity=1 ] (182,95) .. controls (182,93.34) and (183.34,92) .. (185,92) .. controls (186.66,92) and (188,93.34) .. (188,95) .. controls (188,96.66) and (186.66,98) .. (185,98) .. controls (183.34,98) and (182,96.66) .. (182,95) -- cycle ;
\draw    (185,95) .. controls (170.25,99.43) and (165,94.47) .. (165,105) ;
\draw    (185,95) .. controls (199.81,99.47) and (205,94.72) .. (205,105) ;
\draw   (180,105) -- (210,105) -- (210,120) -- (180,120) -- cycle ;
\draw    (165,90) -- (165,65) ;
\draw    (165,90) .. controls (165,100.22) and (185,95.72) .. (185,105) ;
\draw    (165,130) -- (165,105) ;
\draw    (195,130) -- (195,120) ;
\draw  [draw opacity=0] (135,60) -- (160,60) -- (160,85) -- (135,85) -- cycle ;
\draw   (270,25) -- (300,25) -- (300,40) -- (270,40) -- cycle ;
\draw    (265,15) -- (265,5) ;
\draw    (265,15) .. controls (255.17,19.43) and (245,14.47) .. (245,25) ;
\draw  [fill={rgb, 255:red, 0; green, 0; blue, 0 }  ,fill opacity=1 ] (262,15) .. controls (262,13.34) and (263.34,12) .. (265,12) .. controls (266.66,12) and (268,13.34) .. (268,15) .. controls (268,16.66) and (266.66,18) .. (265,18) .. controls (263.34,18) and (262,16.66) .. (262,15) -- cycle ;
\draw    (265,15) .. controls (274.88,19.47) and (285,14.72) .. (285,25) ;
\draw    (285,45) -- (285,40) ;
\draw  [draw opacity=0] (300,50) -- (325,50) -- (325,75) -- (300,75) -- cycle ;
\draw   (270,45) -- (300,45) -- (300,60) -- (270,60) -- cycle ;
\draw    (275,75) -- (275,60) ;
\draw    (295,70) -- (295,60) ;
\draw  [fill={rgb, 255:red, 0; green, 0; blue, 0 }  ,fill opacity=1 ] (292,70) .. controls (292,68.34) and (293.34,67) .. (295,67) .. controls (296.66,67) and (298,68.34) .. (298,70) .. controls (298,71.66) and (296.66,73) .. (295,73) .. controls (293.34,73) and (292,71.66) .. (292,70) -- cycle ;
\draw    (245,65) -- (245,25) ;
\draw  [fill={rgb, 255:red, 0; green, 0; blue, 0 }  ,fill opacity=1 ] (272,75) .. controls (272,73.34) and (273.34,72) .. (275,72) .. controls (276.66,72) and (278,73.34) .. (278,75) .. controls (278,76.66) and (276.66,78) .. (275,78) .. controls (273.34,78) and (272,76.66) .. (272,75) -- cycle ;
\draw    (275,75) .. controls (260.25,79.43) and (260,74.47) .. (260,85) ;
\draw    (275,75) .. controls (289.81,79.47) and (290,74.72) .. (290,85) ;
\draw    (245,65) .. controls (245,75.22) and (250,75.72) .. (250,85) ;
\draw   (235,85) -- (275,85) -- (275,100) -- (235,100) -- cycle ;
\draw  [fill={rgb, 255:red, 0; green, 0; blue, 0 }  ,fill opacity=1 ] (274,110) .. controls (274,108.34) and (275.34,107) .. (277,107) .. controls (278.66,107) and (280,108.34) .. (280,110) .. controls (280,111.66) and (278.66,113) .. (277,113) .. controls (275.34,113) and (274,111.66) .. (274,110) -- cycle ;
\draw    (277,110) .. controls (267.17,114.43) and (255,109.47) .. (255,120) ;
\draw    (277,110) .. controls (286.88,114.47) and (295,109.72) .. (295,120) ;
\draw   (270,120) -- (300,120) -- (300,135) -- (270,135) -- cycle ;
\draw    (255,100) .. controls (255,110.22) and (275,110.72) .. (275,120) ;
\draw    (290,85) .. controls (290,95.22) and (277,100.72) .. (277,110) ;
\draw    (285,145) -- (285,135) ;
\draw    (255,145) -- (255,120) ;
\draw   (350.01,25) -- (380.01,25) -- (380.01,40) -- (350.01,40) -- cycle ;
\draw    (345.01,15) -- (345.01,5) ;
\draw    (345.01,15) .. controls (335.18,19.43) and (325.01,14.47) .. (325.01,25) ;
\draw  [fill={rgb, 255:red, 0; green, 0; blue, 0 }  ,fill opacity=1 ] (342.01,15) .. controls (342.01,13.34) and (343.35,12) .. (345.01,12) .. controls (346.67,12) and (348.01,13.34) .. (348.01,15) .. controls (348.01,16.66) and (346.67,18) .. (345.01,18) .. controls (343.35,18) and (342.01,16.66) .. (342.01,15) -- cycle ;
\draw    (345.01,15) .. controls (354.88,19.47) and (365.01,14.72) .. (365.01,25) ;
\draw    (365.01,45) -- (365.01,40) ;
\draw   (350.01,45) -- (380.01,45) -- (380.01,60) -- (350.01,60) -- cycle ;
\draw    (355.01,75) -- (355,60) ;
\draw    (375.01,70) -- (375,60) ;
\draw  [fill={rgb, 255:red, 0; green, 0; blue, 0 }  ,fill opacity=1 ] (372.01,70) .. controls (372.01,68.34) and (373.35,67) .. (375.01,67) .. controls (376.67,67) and (378.01,68.34) .. (378.01,70) .. controls (378.01,71.66) and (376.67,73) .. (375.01,73) .. controls (373.35,73) and (372.01,71.66) .. (372.01,70) -- cycle ;
\draw    (325.01,80) -- (325,25) ;
\draw  [fill={rgb, 255:red, 0; green, 0; blue, 0 }  ,fill opacity=1 ] (352.01,75) .. controls (352.01,73.34) and (353.35,72) .. (355.01,72) .. controls (356.67,72) and (358.01,73.34) .. (358.01,75) .. controls (358.01,76.66) and (356.67,78) .. (355.01,78) .. controls (353.35,78) and (352.01,76.66) .. (352.01,75) -- cycle ;
\draw    (355.01,75) .. controls (346.44,77.24) and (340.15,74.1) .. (340.01,80) ;
\draw    (355.01,75) .. controls (369.82,79.47) and (375.01,74.72) .. (375.01,85) ;
\draw    (340.01,80) .. controls (340.15,87.67) and (324.58,90.53) .. (325.01,100) ;
\draw   (330.01,100) -- (370.01,100) -- (370.01,115) -- (330.01,115) -- cycle ;
\draw    (325.01,145) -- (325.01,100) ;
\draw    (375.01,90) .. controls (386.12,94.47) and (390.01,89.72) .. (390.01,100) ;
\draw    (375.01,90) .. controls (363.95,94.43) and (360.01,89.47) .. (360.01,100) ;
\draw  [fill={rgb, 255:red, 0; green, 0; blue, 0 }  ,fill opacity=1 ] (372.01,90) .. controls (372.01,88.34) and (373.35,87) .. (375.01,87) .. controls (376.67,87) and (378.01,88.34) .. (378.01,90) .. controls (378.01,91.66) and (376.67,93) .. (375.01,93) .. controls (373.35,93) and (372.01,91.66) .. (372.01,90) -- cycle ;
\draw    (325.01,80) .. controls (325.01,90.22) and (339.72,90.53) .. (340.01,100) ;
\draw   (355.01,125) -- (385.01,125) -- (385.01,140) -- (355.01,140) -- cycle ;
\draw    (350.01,115) .. controls (350.01,121.82) and (360.01,118.82) .. (360.01,125) ;
\draw    (390.01,115) .. controls (390.01,121.82) and (380.01,118.82) .. (380.01,125) ;
\draw  [draw opacity=0] (390.01,60) -- (415.01,60) -- (415.01,85) -- (390.01,85) -- cycle ;
\draw    (390,115) -- (390,100) ;
\draw    (370.01,150) -- (370.01,140) ;
\draw    (375.01,90) -- (375.01,85) ;
\draw  [draw opacity=0] (300,60) -- (325,60) -- (325,85) -- (300,85) -- cycle ;
\draw  [draw opacity=0] (300,60) -- (325,60) -- (325,85) -- (300,85) -- cycle ;

\draw (222.5,72.5) node  [font=\footnotesize]  {$\overset{\emph{(i)}}{=}$};
\draw (402.51,72.5) node  [font=\footnotesize]  {$;$};
\draw (125,57.5) node  [font=\footnotesize]  {$f$};
\draw (125,82.5) node  [font=\footnotesize]  {$g$};
\draw (180,42.5) node  [font=\footnotesize]  {$f$};
\draw (195,72.5) node  [font=\footnotesize]  {$g$};
\draw (195,112.5) node  [font=\footnotesize]  {$\cond(g)$};
\draw (152.5,68) node  [font=\footnotesize]  {$\overset{\emph{(ii)}}{=}$};
\draw (285,32.5) node  [font=\footnotesize]  {$f$};
\draw (285,52.5) node  [font=\footnotesize]  {$g$};
\draw (255,92.5) node  [font=\footnotesize]  {$\pmb{b}(f,g)$};
\draw (285,127.5) node  [font=\footnotesize]  {$\cond(g)$};
\draw (365.01,32.5) node  [font=\footnotesize]  {$f$};
\draw (365.01,52.5) node  [font=\footnotesize]  {$g$};
\draw (350.01,107.5) node  [font=\footnotesize]  {$\pmb{b}(f,g)$};
\draw (370.01,132.5) node  [font=\footnotesize]  {$\cond(g)$};
\draw (312.5,67.5) node  [font=\footnotesize]  {$\overset{\emph{(iii)}}{=}$};

\end{tikzpicture}
     \caption{Construction of the \protect\conditional{} of a composition.}\label{fig:proof-conditional-composition}
  \end{figure}
\end{proof}

\begin{cor}[Bayesian inversion of a composite]
  \label{prop:composition-bayes-inversions}
  A \BayesianInversion{} of a composite channel \((f ⨾ g) ፡ X → Y\) with respect
  to a distribution \(p ፡ I → X\) can be computed by first inverting \(f\) with
  respect to \(p\) and then inverting \(g\) with respect to \((p ⨾ f)\):
  \[\bayesinv{(f ⨾ g)}{p} = \bayesinv{g}{p ⨾ f} ⨾ \bayesinv{f}{p}.\]
\end{cor}

\begin{exa}[Bayesian inversions via density functions, {{c.f. \cite[Example 3.9]{cho_2019}}}]
  \label{exa:inversionsDensity}
  In \BorelStoch{}, the \BayesianInversion{} of a channel $f ፡ (X,Σ_X) →
  \Giry(Y,Σ_Y)$ with respect to a prior, $p ∈ \Giry(X,Σ_X)$, is a channel
  \(\bayesinv{f}{p} \colon (Y,Σ_Y) \to \Giry(X,Σ_X)\) satisfying
  $$
    ∫_{x ∈ S_X} ∫_{y ∈ S_Y}  f(dy | x) · p(dx) =
    ∫_{x ∈ S_X} ∫_{y ∈ S_Y} \bayesinv{f}{p}(dx | y) ·
    ∫_{x₀ ∈ X} f(dy | x₀) · p(dx₀),
  $$
  for each $S_X ∈ Σ_X$ and $S_Y ∈ Σ_Y$.  These inversions need not to exist in
  the more general category of arbitrary measurable spaces
  \cite{fritz2019probability} and, even for standard Borel spaces, they could be
  difficult to compute. Note how the inversion, $\bayesinv{f}{p}$, cannot depend
  on $dy$ but just on $y$. However, distributions and channels are often given
  by density functions,
$$
p(S_X) = ∫_{x ∈ S_X} \pdf{p}(x)\, dx;
\qquad
f(S_Y|x) = ∫_{y ∈ S_Y} \pdf{f}(y,x)\, dy.
$$
In these cases, we can construct a \BayesianInversion{} as follows: we compute
its density function and integrate it to obtain the following channel,
$$
\pdf{\bayesinv{f}{p}}(x,y) = \frac{\pdf{f}(y,x) · \pdf{p}(x)}{∫_{x₀ ∈ X} \pdf{f}(y,x₀) · \pdf{p}(x₀)\,dx₀};
\quad
\bayesinv{f}{p}(S_X,y) = \int_{x ∈ S_X} \pdf{\bayesinv{f}{p}}(x,y)\, dx,
$$
which does constitute a \BayesianInversion{}: note that
\begin{align*}
  & ∫_{x ∈ S_X}∫_{y ∈ S_Y} \pdf{\bayesinv{f}{p}}(x,y) · \left(
    ∫_{x₀ ∈ X} \pdf{f}(y,x₀) · \pdf{p}(x₀)\,dx₀
  \right)\, = \\
  & ∫_{x ∈ S_X}∫_{y ∈ S_Y} \pdf{f}(y,x) · \pdf{p}(x)\,dx.
\end{align*}
\end{exa}
\subsection{Normalisation}
In some cases, one may not be interested in the probability of failure of a
channel and only look at the relative probabilities of the other events. This
may be the case when updating channels on
observations—\Cref{sec:discretePartialMarkovCategories} studies update
procedures in \partialMarkovCategories{}. In these cases, one would like to work
with \emph{normalised} channels.

The \normalisation{} of a finitely supported stochastic channel \(f ፡ X → Y\) is
usually defined by rescaling the output subdistribution by its total probability
mass.
\begin{equation}%
  \label{eq:normalisation-substoch}
  \normal{f} (y \given x) = \dfrac{f(y \given x)}{1- f(\bot \given x)}
\end{equation}
This definition highlights that \normalisation{} is not well-defined on the
inputs that make the channel fail \almostSurely{}, i.e.\ the elements $x ∈ X$
for which \(f(\bot \given x) = 1\). For these inputs, any distribution may act
as a \normalisation{}. While one may define \normalisation{} as a partial
operation from subdistributions to distributions~\cite{cho_2019}; one may also
define \normalisation{} as an operation on partial stochastic channels defined
up to \almostSureEquality{}. As we did for \conditionals{}, we will choose the
latter.%

\begin{defi}[Normalisation]%
\label{def:normalisation}%
\defining{linkNormalisation}{}%
Let \m{f ፡ X → Y} be a morphism in a \partialMarkovCategory{}.
A \defining{linknormal}{\emph{normalisation}} of \(f\) is any morphism \(\normal{f} ፡ X → Y\) such that \(f = (f ⨾ ε) ⊲ \normal{f}\).
\end{defi}

In \MarkovCategories{}, \normalisation{} trivialises: because all morphisms are required to be \total{}, every morphism is its own \normalisation{}.
In \partialMarkovCategories{}, \normalisation{} is a consequence of conditionals.

\begin{prop}%
  In a \partialMarkovCategory{}, all \normalisations{} exist.
  The \normalisation{} of $f ፡ X → Y$ is \((f ⨾ ε_Y)\)-\almostSurely{} unique.
\end{prop}
\begin{proof}
  Given a morphism $f ፡ X → Y$, consider a conditional $\cond(f) ፡ X ⊗ I → Y$ of
  \(f\) with marginal \(f ⨾ ε_{Y}\). The equation defining conditionals
  instantiated in this case is the equation defining normalisation. As a
  consequence, normalisation is $(f \dcomp \discard_{Y})$-\almostSurely{}
  unique, as conditionals are.
\end{proof}

\begin{rem}
  Multiple authors depict \normalisation{} as a shaded box~\cite{jacobs,tull23}:\footnote{Lorenz and Tull \cite{tull23} introduce an axiomatization of normalization based on properties similar to those discussed in the conference version of the present work \cite{dr:evidentialDecision23}.} the content of the box represents the morphism that is being normalised. We have just seen that \normalisation{} induces an operator taking a morphism $f ፡ X → Y$ to a morphism $\normal{f} ፡ X → Y$ defined up to \((f ⨾ ε_Y)\)-\almostSureEquality{}. For this reason, the shaded box representing \(\normal{f}\) should always appear \kl[conditional composition]{conditionally precomposed} with $f ⨾ ε_{Y}$ as on the right of \Cref{fig:normalisation}. 
  
  Finally, let us remark that the box is not functorial \cite{mellies06}: while
  it may be formalised as a collage of string diagrams~\cite{collages23},
  details are not necessary for this application and are left for further work.
\end{rem}

\begin{figure}[h!]

\tikzset{every picture/.style={line width=0.75pt}} %

\begin{tikzpicture}[x=0.75pt,y=0.75pt,yscale=-1,xscale=1]
\draw   (135,40) -- (155,40) -- (155,55) -- (135,55) -- cycle ;
\draw    (145,40) -- (145,20) ;
\draw    (145,70) -- (145,55) ;
\draw   (195,40) -- (215,40) -- (215,55) -- (195,55) -- cycle ;
\draw   (225,40) -- (245,40) -- (245,55) -- (225,55) -- cycle ;
\draw    (220,30) -- (220,20) ;
\draw    (220,30) .. controls (210.17,34.43) and (205,29.47) .. (205,40) ;
\draw  [fill={rgb, 255:red, 0; green, 0; blue, 0 }  ,fill opacity=1 ] (217,30) .. controls (217,28.34) and (218.34,27) .. (220,27) .. controls (221.66,27) and (223,28.34) .. (223,30) .. controls (223,31.66) and (221.66,33) .. (220,33) .. controls (218.34,33) and (217,31.66) .. (217,30) -- cycle ;
\draw    (220,30) .. controls (229.88,34.47) and (235,29.72) .. (235,40) ;
\draw    (205,70) -- (205,55) ;
\draw    (235,65) -- (235,55) ;
\draw  [draw opacity=0] (160,30) -- (190,30) -- (190,55) -- (160,55) -- cycle ;
\draw  [draw opacity=0] (245,30) -- (275,30) -- (275,55) -- (245,55) -- cycle ;
\draw  [fill={rgb, 255:red, 0; green, 0; blue, 0 }  ,fill opacity=1 ] (232,65) .. controls (232,63.34) and (233.34,62) .. (235,62) .. controls (236.66,62) and (238,63.34) .. (238,65) .. controls (238,66.66) and (236.66,68) .. (235,68) .. controls (233.34,68) and (232,66.66) .. (232,65) -- cycle ;
\draw  [color={rgb, 255:red, 129; green, 161; blue, 193 }  ,draw opacity=0.4 ][fill={rgb, 255:red, 129; green, 161; blue, 193 }  ,fill opacity=0.25 ] (190,40) .. controls (190,37.24) and (192.24,35) .. (195,35) -- (215,35) .. controls (217.76,35) and (220,37.24) .. (220,40) -- (220,55) .. controls (220,57.76) and (217.76,60) .. (215,60) -- (195,60) .. controls (192.24,60) and (190,57.76) .. (190,55) -- cycle ;

\draw (175,42.5) node    {$=$};
\draw (145,47.5) node  [font=\footnotesize]  {$f$};
\draw (260,42.5) node    {$;$};
\draw (205,47.5) node  [font=\footnotesize]  {$f$};
\draw (233.5,47.5) node  [font=\footnotesize]  {$f$};

\end{tikzpicture}

   \caption{\protect\Normalisation{} of a morphism.}%
  \label{fig:normalisation}
\end{figure}

Normalising a channel that is already normalised should not have any effect: in the partial Markov category \(\subStoch\), normalised channels are those for which all the output subdistributions have probability mass \(1\) or \(0\); applying the formula in~\eqref{eq:normalisation-substoch} to these channels has no effect.
In fact, \normalisation{} in \kl{partial Markov categories} is an idempotent operation, up to \almostSureEquality{}.
\Cref{lemma:characterisation-quasi-totality} will precisely relate the notions of idempotency of normalisation and self-normalising morphisms.

\begin{prop}%
  \label{prop:normalisationsIdempotent}%
  The \normalisation{} of the \normalisation{} of a morphism $f ፡ X → Y$ is well-defined up to \((f ⨾ ε_Y)\)-\almostSureEquality{} and \((f ⨾ ε_Y)\)-\almostSurelyEqual{} to the original \normalisation{},
  \[
  \normal{\normal{f}} =_{(f ⨾ ε)} \normal{f}.
  \]
\end{prop}
\begin{proof}
  We reason by the string diagrams of \Cref{diagram:idempotent}: it proves
  $\normal{\normal{f}}$ is well-defined up to $(f ⨾ ε_Y)$-\almostSureEquality{}, and its first four equations prove it is \((f ⨾ ε_Y)\)-\almostSurelyEqual{} to the original \normalisation{}.
  \begin{figure}[h!]

\tikzset{every picture/.style={line width=0.75pt}} %

\begin{tikzpicture}[x=0.75pt,y=0.75pt,yscale=-1,xscale=1]
\draw    (400,85) -- (400,70) ;
\draw    (290,25) -- (290,20) ;
\draw   (185,40) -- (205,40) -- (205,55) -- (185,55) -- cycle ;
\draw   (225,40) -- (245,40) -- (245,55) -- (225,55) -- cycle ;
\draw    (215,25) -- (215,15) ;
\draw    (215,25) .. controls (201.89,29.43) and (195,24.47) .. (195,35) ;
\draw  [fill={rgb, 255:red, 0; green, 0; blue, 0 }  ,fill opacity=1 ] (212,25) .. controls (212,23.34) and (213.34,22) .. (215,22) .. controls (216.66,22) and (218,23.34) .. (218,25) .. controls (218,26.66) and (216.66,28) .. (215,28) .. controls (213.34,28) and (212,26.66) .. (212,25) -- cycle ;
\draw    (215,25) .. controls (228.17,29.47) and (235,24.72) .. (235,35) ;
\draw    (195,70) -- (195,55) ;
\draw    (195,40) -- (195,35) ;
\draw  [draw opacity=0] (245,35) -- (275,35) -- (275,60) -- (245,60) -- cycle ;
\draw  [draw opacity=0] (585,35) -- (615,35) -- (615,60) -- (585,60) -- cycle ;
\draw  [fill={rgb, 255:red, 0; green, 0; blue, 0 }  ,fill opacity=1 ] (232,65) .. controls (232,63.34) and (233.34,62) .. (235,62) .. controls (236.66,62) and (238,63.34) .. (238,65) .. controls (238,66.66) and (236.66,68) .. (235,68) .. controls (233.34,68) and (232,66.66) .. (232,65) -- cycle ;
\draw  [color={rgb, 255:red, 129; green, 161; blue, 193 }  ,draw opacity=0.75 ][fill={rgb, 255:red, 129; green, 161; blue, 193 }  ,fill opacity=0.25 ] (180,40.4) .. controls (180,37.69) and (182.19,35.5) .. (184.9,35.5) -- (205.1,35.5) .. controls (207.81,35.5) and (210,37.69) .. (210,40.4) -- (210,55.1) .. controls (210,57.81) and (207.81,60) .. (205.1,60) -- (184.9,60) .. controls (182.19,60) and (180,57.81) .. (180,55.1) -- cycle ;
\draw    (235,40) -- (235,35) ;
\draw    (235,65) -- (235,55) ;
\draw   (305,55) -- (325,55) -- (325,70) -- (305,70) -- cycle ;
\draw    (330,45) .. controls (320.17,49.43) and (315,44.47) .. (315,55) ;
\draw  [fill={rgb, 255:red, 0; green, 0; blue, 0 }  ,fill opacity=1 ] (307,10) .. controls (307,8.34) and (308.34,7) .. (310,7) .. controls (311.66,7) and (313,8.34) .. (313,10) .. controls (313,11.66) and (311.66,13) .. (310,13) .. controls (308.34,13) and (307,11.66) .. (307,10) -- cycle ;
\draw    (330,45) .. controls (339.88,49.47) and (345,44.72) .. (345,55) ;
\draw    (290,85) -- (290,40) ;
\draw    (315,80) -- (315,70) ;
\draw   (335,55) -- (355,55) -- (355,70) -- (335,70) -- cycle ;
\draw  [color={rgb, 255:red, 129; green, 161; blue, 193 }  ,draw opacity=0.75 ][fill={rgb, 255:red, 129; green, 161; blue, 193 }  ,fill opacity=0.25 ] (175,37) .. controls (175,33.13) and (178.13,30) .. (182,30) -- (208,30) .. controls (211.87,30) and (215,33.13) .. (215,37) -- (215,58) .. controls (215,61.87) and (211.87,65) .. (208,65) -- (182,65) .. controls (178.13,65) and (175,61.87) .. (175,58) -- cycle ;
\draw  [color={rgb, 255:red, 129; green, 161; blue, 193 }  ,draw opacity=0.75 ][fill={rgb, 255:red, 129; green, 161; blue, 193 }  ,fill opacity=0.25 ] (300,55) .. controls (300,52.24) and (302.24,50) .. (305,50) -- (325,50) .. controls (327.76,50) and (330,52.24) .. (330,55) -- (330,70) .. controls (330,72.76) and (327.76,75) .. (325,75) -- (305,75) .. controls (302.24,75) and (300,72.76) .. (300,70) -- cycle ;
\draw    (310,10) .. controls (300.17,14.43) and (290,9.47) .. (290,20) ;
\draw    (310,10) .. controls (319.88,14.47) and (330,9.72) .. (330,20) ;
\draw    (345,80) -- (345,70) ;
\draw  [fill={rgb, 255:red, 0; green, 0; blue, 0 }  ,fill opacity=1 ] (342,80) .. controls (342,78.34) and (343.34,77) .. (345,77) .. controls (346.66,77) and (348,78.34) .. (348,80) .. controls (348,81.66) and (346.66,83) .. (345,83) .. controls (343.34,83) and (342,81.66) .. (342,80) -- cycle ;
\draw    (330,45) -- (330,20) ;
\draw    (310,10) -- (310,0) ;
\draw  [fill={rgb, 255:red, 0; green, 0; blue, 0 }  ,fill opacity=1 ] (327,45) .. controls (327,43.34) and (328.34,42) .. (330,42) .. controls (331.66,42) and (333,43.34) .. (333,45) .. controls (333,46.66) and (331.66,48) .. (330,48) .. controls (328.34,48) and (327,46.66) .. (327,45) -- cycle ;
\draw  [draw opacity=0] (355,35) -- (385,35) -- (385,60) -- (355,60) -- cycle ;
\draw  [fill={rgb, 255:red, 0; green, 0; blue, 0 }  ,fill opacity=1 ] (312,80) .. controls (312,78.34) and (313.34,77) .. (315,77) .. controls (316.66,77) and (318,78.34) .. (318,80) .. controls (318,81.66) and (316.66,83) .. (315,83) .. controls (313.34,83) and (312,81.66) .. (312,80) -- cycle ;
\draw  [draw opacity=0] (455,35) -- (485,35) -- (485,60) -- (455,60) -- cycle ;
\draw   (485,40) -- (505,40) -- (505,55) -- (485,55) -- cycle ;
\draw  [fill={rgb, 255:red, 0; green, 0; blue, 0 }  ,fill opacity=1 ] (507,25) .. controls (507,23.34) and (508.34,22) .. (510,22) .. controls (511.66,22) and (513,23.34) .. (513,25) .. controls (513,26.66) and (511.66,28) .. (510,28) .. controls (508.34,28) and (507,26.66) .. (507,25) -- cycle ;
\draw    (495,70) -- (495,55) ;
\draw   (515,40) -- (535,40) -- (535,55) -- (515,55) -- cycle ;
\draw  [color={rgb, 255:red, 129; green, 161; blue, 193 }  ,draw opacity=0.75 ][fill={rgb, 255:red, 129; green, 161; blue, 193 }  ,fill opacity=0.25 ] (482.5,41.5) .. controls (482.5,39.29) and (484.29,37.5) .. (486.5,37.5) -- (503.5,37.5) .. controls (505.71,37.5) and (507.5,39.29) .. (507.5,41.5) -- (507.5,53.5) .. controls (507.5,55.71) and (505.71,57.5) .. (503.5,57.5) -- (486.5,57.5) .. controls (484.29,57.5) and (482.5,55.71) .. (482.5,53.5) -- cycle ;
\draw    (510,25) .. controls (500.17,29.43) and (495,24.47) .. (495,35) ;
\draw    (510,25) .. controls (519.88,29.47) and (525,24.72) .. (525,35) ;
\draw    (525,65) -- (525,55) ;
\draw  [fill={rgb, 255:red, 0; green, 0; blue, 0 }  ,fill opacity=1 ] (522,65) .. controls (522,63.34) and (523.34,62) .. (525,62) .. controls (526.66,62) and (528,63.34) .. (528,65) .. controls (528,66.66) and (526.66,68) .. (525,68) .. controls (523.34,68) and (522,66.66) .. (522,65) -- cycle ;
\draw    (495,40) -- (495,35) ;
\draw    (510,25) -- (510,15) ;
\draw   (565,35) -- (585,35) -- (585,50) -- (565,50) -- cycle ;
\draw  [draw opacity=0] (535,35) -- (565,35) -- (565,60) -- (535,60) -- cycle ;
\draw    (575,35) -- (575,20) ;
\draw    (575,65) -- (575,50) ;
\draw    (525,40) -- (525,35) ;
\draw   (280,25) -- (300,25) -- (300,40) -- (280,40) -- cycle ;
\draw  [color={rgb, 255:red, 129; green, 161; blue, 193 }  ,draw opacity=0.75 ][fill={rgb, 255:red, 129; green, 161; blue, 193 }  ,fill opacity=0.25 ] (275,25.4) .. controls (275,22.69) and (277.19,20.5) .. (279.9,20.5) -- (300.1,20.5) .. controls (302.81,20.5) and (305,22.69) .. (305,25.4) -- (305,40.1) .. controls (305,42.81) and (302.81,45) .. (300.1,45) -- (279.9,45) .. controls (277.19,45) and (275,42.81) .. (275,40.1) -- cycle ;
\draw  [color={rgb, 255:red, 129; green, 161; blue, 193 }  ,draw opacity=0.75 ][fill={rgb, 255:red, 129; green, 161; blue, 193 }  ,fill opacity=0.25 ] (270,22) .. controls (270,18.13) and (273.13,15) .. (277,15) -- (303,15) .. controls (306.87,15) and (310,18.13) .. (310,22) -- (310,43) .. controls (310,46.87) and (306.87,50) .. (303,50) -- (277,50) .. controls (273.13,50) and (270,46.87) .. (270,43) -- cycle ;
\draw    (435,10) -- (435,0) ;
\draw   (430,55) -- (450,55) -- (450,70) -- (430,70) -- cycle ;
\draw    (435,11) .. controls (425.17,15.43) and (420,10.47) .. (420,21) ;
\draw  [fill={rgb, 255:red, 0; green, 0; blue, 0 }  ,fill opacity=1 ] (417,45) .. controls (417,43.34) and (418.34,42) .. (420,42) .. controls (421.66,42) and (423,43.34) .. (423,45) .. controls (423,46.66) and (421.66,48) .. (420,48) .. controls (418.34,48) and (417,46.66) .. (417,45) -- cycle ;
\draw    (435,11) .. controls (444.88,15.47) and (450,10.72) .. (450,21) ;
\draw    (420,45) -- (420,20) ;
\draw   (440,20) -- (460,20) -- (460,35) -- (440,35) -- cycle ;
\draw  [color={rgb, 255:red, 129; green, 161; blue, 193 }  ,draw opacity=0.75 ][fill={rgb, 255:red, 129; green, 161; blue, 193 }  ,fill opacity=0.25 ] (425,55) .. controls (425,52.24) and (427.24,50) .. (430,50) -- (450,50) .. controls (452.76,50) and (455,52.24) .. (455,55) -- (455,70) .. controls (455,72.76) and (452.76,75) .. (450,75) -- (430,75) .. controls (427.24,75) and (425,72.76) .. (425,70) -- cycle ;
\draw    (420,45) .. controls (410.17,49.43) and (400,44.47) .. (400,55) ;
\draw    (420,45) .. controls (429.88,49.47) and (440,44.72) .. (440,55) ;
\draw    (440,80) -- (440,70) ;
\draw  [fill={rgb, 255:red, 0; green, 0; blue, 0 }  ,fill opacity=1 ] (447,45) .. controls (447,43.34) and (448.34,42) .. (450,42) .. controls (451.66,42) and (453,43.34) .. (453,45) .. controls (453,46.66) and (451.66,48) .. (450,48) .. controls (448.34,48) and (447,46.66) .. (447,45) -- cycle ;
\draw    (420,45) -- (420,40) ;
\draw  [fill={rgb, 255:red, 0; green, 0; blue, 0 }  ,fill opacity=1 ] (432,11) .. controls (432,9.34) and (433.34,8) .. (435,8) .. controls (436.66,8) and (438,9.34) .. (438,11) .. controls (438,12.66) and (436.66,14) .. (435,14) .. controls (433.34,14) and (432,12.66) .. (432,11) -- cycle ;
\draw  [fill={rgb, 255:red, 0; green, 0; blue, 0 }  ,fill opacity=1 ] (437,80) .. controls (437,78.34) and (438.34,77) .. (440,77) .. controls (441.66,77) and (443,78.34) .. (443,80) .. controls (443,81.66) and (441.66,83) .. (440,83) .. controls (438.34,83) and (437,81.66) .. (437,80) -- cycle ;
\draw   (390,55) -- (410,55) -- (410,70) -- (390,70) -- cycle ;
\draw  [color={rgb, 255:red, 129; green, 161; blue, 193 }  ,draw opacity=0.75 ][fill={rgb, 255:red, 129; green, 161; blue, 193 }  ,fill opacity=0.25 ] (385,55.4) .. controls (385,52.69) and (387.19,50.5) .. (389.9,50.5) -- (410.1,50.5) .. controls (412.81,50.5) and (415,52.69) .. (415,55.4) -- (415,70.1) .. controls (415,72.81) and (412.81,75) .. (410.1,75) -- (389.9,75) .. controls (387.19,75) and (385,72.81) .. (385,70.1) -- cycle ;
\draw  [color={rgb, 255:red, 129; green, 161; blue, 193 }  ,draw opacity=0.75 ][fill={rgb, 255:red, 129; green, 161; blue, 193 }  ,fill opacity=0.25 ] (380,52) .. controls (380,48.13) and (383.13,45) .. (387,45) -- (413,45) .. controls (416.87,45) and (420,48.13) .. (420,52) -- (420,73) .. controls (420,76.87) and (416.87,80) .. (413,80) -- (387,80) .. controls (383.13,80) and (380,76.87) .. (380,73) -- cycle ;
\draw    (450,45) -- (450,35) ;

\draw (260,47.5) node    {$=$};
\draw (600,47.5) node    {$;$};
\draw (195,47.5) node  [font=\footnotesize]  {$f$};
\draw (235,47.5) node  [font=\footnotesize]  {$f$};
\draw (315,62.5) node  [font=\footnotesize]  {$f$};
\draw (343.5,62.5) node  [font=\footnotesize]  {$f$};
\draw (370,47.5) node    {$=$};
\draw (470,47.5) node    {$=$};
\draw (495,47.5) node  [font=\footnotesize]  {$f$};
\draw (525,47.5) node  [font=\footnotesize]  {$f$};
\draw (575,42.5) node  [font=\footnotesize]  {$f$};
\draw (550,47.5) node    {$=$};
\draw (290,32.5) node  [font=\footnotesize]  {$f$};
\draw (440,62.5) node  [font=\footnotesize]  {$f$};
\draw (448.5,27.5) node  [font=\footnotesize]  {$f$};
\draw (400,62.5) node  [font=\footnotesize]  {$f$};

\end{tikzpicture}
     \caption{\protect\Normalisation{} is \protect\almostSurely{} idempotent.}%
    \label{diagram:idempotent}
  \end{figure}
\end{proof}

\Normalisation{} is not functorial.
Functoriality, $\normal{f ⨾ g} = \normal{f} ⨾ \normal{g}$, already fails in \SubStoch{} because composing two normalised morphisms does not necessarily yield a normalised morphism: \(\normal{f}\) may have nonzero probability of returning a value outside the domain of \(\normal{g}\).
What is instead true is that normalising at the end of a computation is the same as normalising at any point, $\normal{f ⨾ g} = \normal{\normal{f} ⨾ g}$.

\begin{prop}[Normalisation precomposes]%
  \label{prop:normalisationPrecomposes}%
  Let $f ፡ X → Y$ and $g ፡ Y → Z$ be two morphisms in a
  \partialMarkovCategory{}, then $\normal{f ⨾ g} = \normal{\normal{f} ⨾ g}$.
\end{prop}
\begin{proof}
  We reason by the string diagrams in \Cref{diagram:normaliseWhenever}.
    \begin{figure}[h!]

\tikzset{every picture/.style={line width=0.75pt}} %

\begin{tikzpicture}[x=0.75pt,y=0.75pt,yscale=-1,xscale=1]
\draw  [color={rgb, 255:red, 129; green, 161; blue, 193 }  ,draw opacity=0.75 ][fill={rgb, 255:red, 129; green, 161; blue, 193 }  ,fill opacity=0.25 ] (150,35.4) .. controls (150,32.69) and (152.19,30.5) .. (154.9,30.5) -- (175.1,30.5) .. controls (177.81,30.5) and (180,32.69) .. (180,35.4) -- (180,50.1) .. controls (180,52.81) and (177.81,55) .. (175.1,55) -- (154.9,55) .. controls (152.19,55) and (150,52.81) .. (150,50.1) -- cycle ;
\draw  [color={rgb, 255:red, 129; green, 161; blue, 193 }  ,draw opacity=0.75 ][fill={rgb, 255:red, 129; green, 161; blue, 193 }  ,fill opacity=0.25 ] (10,33) .. controls (10,28.58) and (13.58,25) .. (18,25) -- (42,25) .. controls (46.42,25) and (50,28.58) .. (50,33) -- (50,72) .. controls (50,76.42) and (46.42,80) .. (42,80) -- (18,80) .. controls (13.58,80) and (10,76.42) .. (10,72) -- cycle ;
\draw  [color={rgb, 255:red, 129; green, 161; blue, 193 }  ,draw opacity=0.75 ][fill={rgb, 255:red, 129; green, 161; blue, 193 }  ,fill opacity=0.25 ] (15,35.4) .. controls (15,32.69) and (17.19,30.5) .. (19.9,30.5) -- (40.1,30.5) .. controls (42.81,30.5) and (45,32.69) .. (45,35.4) -- (45,50.1) .. controls (45,52.81) and (42.81,55) .. (40.1,55) -- (19.9,55) .. controls (17.19,55) and (15,52.81) .. (15,50.1) -- cycle ;
\draw   (20,35) -- (40,35) -- (40,50) -- (20,50) -- cycle ;
\draw   (60,35) -- (80,35) -- (80,50) -- (60,50) -- cycle ;
\draw    (50,20) -- (50,10) ;
\draw    (50,20) .. controls (36.89,24.43) and (30,19.47) .. (30,30) ;
\draw  [fill={rgb, 255:red, 0; green, 0; blue, 0 }  ,fill opacity=1 ] (47,20) .. controls (47,18.34) and (48.34,17) .. (50,17) .. controls (51.66,17) and (53,18.34) .. (53,20) .. controls (53,21.66) and (51.66,23) .. (50,23) .. controls (48.34,23) and (47,21.66) .. (47,20) -- cycle ;
\draw    (50,20) .. controls (63.17,24.47) and (70,19.72) .. (70,30) ;
\draw    (30,60) -- (30,50) ;
\draw    (30,35) -- (30,30) ;
\draw  [draw opacity=0] (80,30) -- (110,30) -- (110,55) -- (80,55) -- cycle ;
\draw  [fill={rgb, 255:red, 0; green, 0; blue, 0 }  ,fill opacity=1 ] (67,85) .. controls (67,83.34) and (68.34,82) .. (70,82) .. controls (71.66,82) and (73,83.34) .. (73,85) .. controls (73,86.66) and (71.66,88) .. (70,88) .. controls (68.34,88) and (67,86.66) .. (67,85) -- cycle ;
\draw    (70,35) -- (70,30) ;
\draw    (70,85) -- (70,75) ;
\draw   (20,60) -- (40,60) -- (40,75) -- (20,75) -- cycle ;
\draw    (30,90) -- (30,75) ;
\draw   (60,60) -- (80,60) -- (80,75) -- (60,75) -- cycle ;
\draw    (70,60) -- (70,50) ;
\draw  [color={rgb, 255:red, 129; green, 161; blue, 193 }  ,draw opacity=0.75 ][fill={rgb, 255:red, 129; green, 161; blue, 193 }  ,fill opacity=0.25 ] (115,30.4) .. controls (115,27.69) and (117.19,25.5) .. (119.9,25.5) -- (140.1,25.5) .. controls (142.81,25.5) and (145,27.69) .. (145,30.4) -- (145,45.1) .. controls (145,47.81) and (142.81,50) .. (140.1,50) -- (119.9,50) .. controls (117.19,50) and (115,47.81) .. (115,45.1) -- cycle ;
\draw  [color={rgb, 255:red, 129; green, 161; blue, 193 }  ,draw opacity=0.75 ][fill={rgb, 255:red, 129; green, 161; blue, 193 }  ,fill opacity=0.25 ] (110,28) .. controls (110,23.58) and (113.58,20) .. (118,20) -- (142,20) .. controls (146.42,20) and (150,23.58) .. (150,28) -- (150,67) .. controls (150,71.42) and (146.42,75) .. (142,75) -- (118,75) .. controls (113.58,75) and (110,71.42) .. (110,67) -- cycle ;
\draw   (120,30) -- (140,30) -- (140,45) -- (120,45) -- cycle ;
\draw   (155,35) -- (175,35) -- (175,50) -- (155,50) -- cycle ;
\draw    (155,15) -- (155,10) ;
\draw    (155,15) .. controls (141.89,19.43) and (130,14.47) .. (130,25) ;
\draw  [fill={rgb, 255:red, 0; green, 0; blue, 0 }  ,fill opacity=1 ] (152,15) .. controls (152,13.34) and (153.34,12) .. (155,12) .. controls (156.66,12) and (158,13.34) .. (158,15) .. controls (158,16.66) and (156.66,18) .. (155,18) .. controls (153.34,18) and (152,16.66) .. (152,15) -- cycle ;
\draw    (155,15) .. controls (168.17,19.47) and (180,14.72) .. (180,25) ;
\draw    (130,55) -- (130,45) ;
\draw    (130,30) -- (130,25) ;
\draw  [fill={rgb, 255:red, 0; green, 0; blue, 0 }  ,fill opacity=1 ] (162,85) .. controls (162,83.34) and (163.34,82) .. (165,82) .. controls (166.66,82) and (168,83.34) .. (168,85) .. controls (168,86.66) and (166.66,88) .. (165,88) .. controls (163.34,88) and (162,86.66) .. (162,85) -- cycle ;
\draw    (165,85) -- (165,75) ;
\draw   (120,55) -- (140,55) -- (140,70) -- (120,70) -- cycle ;
\draw    (130,90) -- (130,70) ;
\draw   (155,60) -- (175,60) -- (175,75) -- (155,75) -- cycle ;
\draw    (165,60) -- (165,50) ;
\draw  [fill={rgb, 255:red, 0; green, 0; blue, 0 }  ,fill opacity=1 ] (177,25) .. controls (177,23.34) and (178.34,22) .. (180,22) .. controls (181.66,22) and (183,23.34) .. (183,25) .. controls (183,26.66) and (181.66,28) .. (180,28) .. controls (178.34,28) and (177,26.66) .. (177,25) -- cycle ;
\draw    (180,25) .. controls (170.17,29.43) and (165,24.47) .. (165,35) ;
\draw    (180,25) .. controls (189.88,29.47) and (195,24.72) .. (195,35) ;
\draw   (185,35) -- (205,35) -- (205,50) -- (185,50) -- cycle ;
\draw    (195,60) -- (195,50) ;
\draw  [fill={rgb, 255:red, 0; green, 0; blue, 0 }  ,fill opacity=1 ] (192,60) .. controls (192,58.34) and (193.34,57) .. (195,57) .. controls (196.66,57) and (198,58.34) .. (198,60) .. controls (198,61.66) and (196.66,63) .. (195,63) .. controls (193.34,63) and (192,61.66) .. (192,60) -- cycle ;
\draw  [color={rgb, 255:red, 129; green, 161; blue, 193 }  ,draw opacity=0.75 ][fill={rgb, 255:red, 129; green, 161; blue, 193 }  ,fill opacity=0.25 ] (280,34.9) .. controls (280,32.19) and (282.19,30) .. (284.9,30) -- (305.1,30) .. controls (307.81,30) and (310,32.19) .. (310,34.9) -- (310,49.6) .. controls (310,52.31) and (307.81,54.5) .. (305.1,54.5) -- (284.9,54.5) .. controls (282.19,54.5) and (280,52.31) .. (280,49.6) -- cycle ;
\draw  [color={rgb, 255:red, 129; green, 161; blue, 193 }  ,draw opacity=0.75 ][fill={rgb, 255:red, 129; green, 161; blue, 193 }  ,fill opacity=0.25 ] (240,35.4) .. controls (240,32.69) and (242.19,30.5) .. (244.9,30.5) -- (265.1,30.5) .. controls (267.81,30.5) and (270,32.69) .. (270,35.4) -- (270,50.1) .. controls (270,52.81) and (267.81,55) .. (265.1,55) -- (244.9,55) .. controls (242.19,55) and (240,52.81) .. (240,50.1) -- cycle ;
\draw  [color={rgb, 255:red, 129; green, 161; blue, 193 }  ,draw opacity=0.75 ][fill={rgb, 255:red, 129; green, 161; blue, 193 }  ,fill opacity=0.25 ] (235,33) .. controls (235,28.58) and (238.58,25) .. (243,25) -- (267,25) .. controls (271.42,25) and (275,28.58) .. (275,33) -- (275,72) .. controls (275,76.42) and (271.42,80) .. (267,80) -- (243,80) .. controls (238.58,80) and (235,76.42) .. (235,72) -- cycle ;
\draw   (245,35) -- (265,35) -- (265,50) -- (245,50) -- cycle ;
\draw   (285,34.5) -- (305,34.5) -- (305,49.5) -- (285,49.5) -- cycle ;
\draw    (300,15) -- (300,5) ;
\draw    (300,15) .. controls (286.89,19.43) and (275,14.47) .. (275,25) ;
\draw  [fill={rgb, 255:red, 0; green, 0; blue, 0 }  ,fill opacity=1 ] (297,15) .. controls (297,13.34) and (298.34,12) .. (300,12) .. controls (301.66,12) and (303,13.34) .. (303,15) .. controls (303,16.66) and (301.66,18) .. (300,18) .. controls (298.34,18) and (297,16.66) .. (297,15) -- cycle ;
\draw    (300,15) .. controls (313.17,19.47) and (325,14.72) .. (325,25) ;
\draw    (255,60) -- (255,50) ;
\draw  [fill={rgb, 255:red, 0; green, 0; blue, 0 }  ,fill opacity=1 ] (292,84.5) .. controls (292,82.84) and (293.34,81.5) .. (295,81.5) .. controls (296.66,81.5) and (298,82.84) .. (298,84.5) .. controls (298,86.16) and (296.66,87.5) .. (295,87.5) .. controls (293.34,87.5) and (292,86.16) .. (292,84.5) -- cycle ;
\draw    (295,84.5) -- (295,74.5) ;
\draw   (245,60) -- (265,60) -- (265,75) -- (245,75) -- cycle ;
\draw    (255,90) -- (255,75) ;
\draw   (285,59.5) -- (305,59.5) -- (305,74.5) -- (285,74.5) -- cycle ;
\draw    (295,59.5) -- (295,49.5) ;
\draw  [fill={rgb, 255:red, 0; green, 0; blue, 0 }  ,fill opacity=1 ] (272,25) .. controls (272,23.34) and (273.34,22) .. (275,22) .. controls (276.66,22) and (278,23.34) .. (278,25) .. controls (278,26.66) and (276.66,28) .. (275,28) .. controls (273.34,28) and (272,26.66) .. (272,25) -- cycle ;
\draw    (275,25) .. controls (265.17,29.43) and (255,24.47) .. (255,35) ;
\draw    (275,25) .. controls (284.88,29.47) and (295,24.72) .. (295,35) ;
\draw   (315,35) -- (335,35) -- (335,50) -- (315,50) -- cycle ;
\draw    (325,60) -- (325,50) ;
\draw  [fill={rgb, 255:red, 0; green, 0; blue, 0 }  ,fill opacity=1 ] (322,60) .. controls (322,58.34) and (323.34,57) .. (325,57) .. controls (326.66,57) and (328,58.34) .. (328,60) .. controls (328,61.66) and (326.66,63) .. (325,63) .. controls (323.34,63) and (322,61.66) .. (322,60) -- cycle ;
\draw    (325,35) -- (325,25) ;
\draw  [draw opacity=0] (205,30) -- (235,30) -- (235,55) -- (205,55) -- cycle ;
\draw  [color={rgb, 255:red, 129; green, 161; blue, 193 }  ,draw opacity=0.75 ][fill={rgb, 255:red, 129; green, 161; blue, 193 }  ,fill opacity=0.25 ] (365,35.4) .. controls (365,32.69) and (367.19,30.5) .. (369.9,30.5) -- (390.1,30.5) .. controls (392.81,30.5) and (395,32.69) .. (395,35.4) -- (395,50.1) .. controls (395,52.81) and (392.81,55) .. (390.1,55) -- (369.9,55) .. controls (367.19,55) and (365,52.81) .. (365,50.1) -- cycle ;
\draw   (370,35) -- (390,35) -- (390,50) -- (370,50) -- cycle ;
\draw    (395,25) -- (395,15) ;
\draw    (395,25) .. controls (381.89,29.43) and (380,24.47) .. (380,35) ;
\draw  [fill={rgb, 255:red, 0; green, 0; blue, 0 }  ,fill opacity=1 ] (392,25) .. controls (392,23.34) and (393.34,22) .. (395,22) .. controls (396.66,22) and (398,23.34) .. (398,25) .. controls (398,26.66) and (396.66,28) .. (395,28) .. controls (393.34,28) and (392,26.66) .. (392,25) -- cycle ;
\draw    (395,25) .. controls (408.17,29.47) and (410,24.72) .. (410,35) ;
\draw    (380,85) -- (380,75) ;
\draw   (370,60) -- (390,60) -- (390,75) -- (370,75) -- cycle ;
\draw    (380,60) -- (380,50) ;
\draw   (400,35.5) -- (420,35.5) -- (420,50.5) -- (400,50.5) -- cycle ;
\draw    (410,60.5) -- (410,50.5) ;
\draw  [fill={rgb, 255:red, 0; green, 0; blue, 0 }  ,fill opacity=1 ] (407,60.5) .. controls (407,58.84) and (408.34,57.5) .. (410,57.5) .. controls (411.66,57.5) and (413,58.84) .. (413,60.5) .. controls (413,62.16) and (411.66,63.5) .. (410,63.5) .. controls (408.34,63.5) and (407,62.16) .. (407,60.5) -- cycle ;
\draw  [draw opacity=0] (335,30) -- (365,30) -- (365,55) -- (335,55) -- cycle ;
\draw  [draw opacity=0] (420,30) -- (450,30) -- (450,55) -- (420,55) -- cycle ;
\draw   (450,30) -- (470,30) -- (470,45) -- (450,45) -- cycle ;
\draw    (460,80) -- (460,70) ;
\draw   (450,55) -- (470,55) -- (470,70) -- (450,70) -- cycle ;
\draw    (460,55) -- (460,45) ;
\draw    (460,30) -- (460,20) ;
\draw  [draw opacity=0] (470,30) -- (500,30) -- (500,55) -- (470,55) -- cycle ;

\draw (95,42.5) node    {$=$};
\draw (30,42.5) node  [font=\footnotesize]  {$f$};
\draw (30,67.5) node  [font=\footnotesize]  {$g$};
\draw (70,67.5) node  [font=\footnotesize]  {$g$};
\draw (70,42.5) node  [font=\footnotesize]  {$f$};
\draw (130,37.5) node  [font=\footnotesize]  {$f$};
\draw (130,62.5) node  [font=\footnotesize]  {$g$};
\draw (165,67.5) node  [font=\footnotesize]  {$g$};
\draw (165,42.75) node  [font=\footnotesize]  {$f$};
\draw (195,42.5) node  [font=\footnotesize]  {$f$};
\draw (255,42.5) node  [font=\footnotesize]  {$f$};
\draw (255,67.5) node  [font=\footnotesize]  {$g$};
\draw (295,67) node  [font=\footnotesize]  {$g$};
\draw (295,42.25) node  [font=\footnotesize]  {$f$};
\draw (325,42.5) node  [font=\footnotesize]  {$f$};
\draw (220,42.5) node    {$=$};
\draw (380,67.5) node  [font=\footnotesize]  {$g$};
\draw (380,42.75) node  [font=\footnotesize]  {$f$};
\draw (410,43) node  [font=\footnotesize]  {$f$};
\draw (350,42.5) node    {$=$};
\draw (435,42.5) node    {$=$};
\draw (460,62.5) node  [font=\footnotesize]  {$g$};
\draw (460,37.75) node  [font=\footnotesize]  {$f$};
\draw (485,42.5) node    {$;$};

\end{tikzpicture}
         \caption{\protect\Normalisation{} precomposes.}
        \label{diagram:normaliseWhenever}
      \end{figure}
\end{proof}

Probabilistic programming implicitly relies on the result in
\Cref{prop:normalisationPrecomposes} to renormalize the validity at any point of
the computation without changing semantics
\cite[Section~4.1]{staton_et_al_2016}: ``Whenever there is a score, it is good
to renormalize and resample.'' \Normalisation{} does also not influence
conditioning: the \conditionals{} of a \normalisation{} are still
\conditionals{} for the original morphism. \Cref{prop:conditionals-of-programs},
for exact observations, will rely on this fact.

\begin{prop}%
  \label{prop:conditionals-of-normalisation}
  Let \m{f ፡ X → Y ⊗ Z} be a morphism in a \partialMarkovCategory{}.
  Then, the \conditionals{} of the \normalisation{} of \m{f} are \conditionals{} of \m{f}.
  In other words, given a \conditional{} of the \normalisation{}, $c_{\normal{f}}$, and a
  \conditional{}, $c_f$, we have
  \[
  c_{\normal{f}} =_{(f ⨾ \pi_{1})} c_{f}.
  \]
\end{prop}
\begin{proof}
  Let \m{c_{\normal{f}} ፡ X ⊗ Y → Z} be a \conditional{} of the \normalisation{}
  of \m{f}. Let us first note that the \normalisation{} preserves discarding,
  meaning that $\normal{f ⨾ \pi_1} =_{(f ⨾ ε_{Y ⊗ Z})} \normal{f} ⨾ \pi_{1}$.  We
  use \emph{(i)} that \normalisation{} preserves discarding, \emph{(ii)} the definition
  of \conditional{}, and \emph{(iii)} the definition of \normalisation{}.
  \[
  (f ⨾ \pi_{1}) ⊲ c_{\normal{f}}
  = (f ⨾ ε_{Y ⊗ Z}) ⊲ (\normal{f} ⨾ \pi_{1}) ⊲ c_{\normal{f}}
  = (f ⨾ ε_{Y ⊗ Z}) ⊲ \normal{f}
  = f.
  \]
  This shows that \(c_{\normal{f}}\) is a conditional of \(f\), and it concludes the proof because \conditionals{} are \almostSurely{} unique.
\end{proof}

\subsection{Domains, and quasi-total morphisms}

\kl{Cartesian restriction categories}~\cite{cockett2007restriction} axiomatise
partial deterministic processes—each process is allowed to fail on some inputs.
The \emph{domain of definition} is the \kl{predicate} that records the inputs on
which the process terminates: it runs the morphism and discards the value of the
output, only keeping the information on whether the output was produced.

\begin{defi}[Cartesian restriction category]%
  \label{def:cartesian-restriction}
  A \intro{cartesian restriction category} is a \kl{copy-discard category} where all morphisms are \kl{deterministic}\footnote{This definition is a characterisation~\cite[Theorem~5.2]{cockett2007restriction} of the original definition~\cite[2.2]{cockett2008introduction}.}.
\end{defi}

\PartialMarkovCategories{} axiomatise partial processes with a notion of nondeterminism, generalising \kl{cartesian restriction categories}.

\begin{prop}[{\cite[Proposition~2.13]{2025order}}]%
  \label{prop:cartesian-restriction}
  \kl{Cartesian restriction categories} have \kl{conditionals}.
  In other words, \kl{cartesian restriction categories} are precisely \kl{partial Markov categories} where all morphisms are \kl{deterministic}.
\end{prop}
\begin{proof}[Proof sketch]
  Conditionals follow from determinism: \(f = f \dcomp \cp \dcomp (π_1 \tensor π_2) = \cp \dcomp ((f \dcomp π_1) \tensor (f \dcomp π_2))\). 
\end{proof}

Morphisms in \kl{partial Markov categories} are allowed to fail on some inputs like those in \kl{cartesian restriction categories}, but they may do so nondeterministically.
We record this information in a \kl{predicate}, as
in cartesian restriction categories. This time, the domain predicate is not
necessarily deterministic: it associates to each input a ``probability of termination''.

\begin{defi}[Domain of definition]%
  \label{def:domain-of-definition}%
  \defining{linkDomainOfDefinition}
  The \intro{domain of definition} of a
  morphism \(f \colon X \to Y\) in a \copyDiscardCategory{} is the
  \kl{predicate} $(f \dcomp \discard_{Y}) ፡ X → I$ on the domain of \(f\).
\end{defi}

\begin{exa}[Domains of definition of substochastic channels]
\label{rem:probability-of-failure}%
\label{rem:domainOfDefinition}%
In \(\subStoch\), the \kl{domain of definition} of a  morphism \(f ፡ X → Y\)
yields the probability of ``success'': the probability that it does not fail,
\[%
(f ⨾ ε)(\ast \given x) = \sum_{y ∈ Y} f(y \given x) = 1 - f(\bot \given x).%
\]%
This fuzzy \kl{predicate} is sharp exactly when \((f ⨾ ε)\) is itself
\deterministic{} (\Cref{fig:domain-of-definition}): in this case, \(f\) factors
through the inclusion \(\distr(Y)+1 \into \subdistr(Y)\).
\end{exa}

We could be tempted to impose that conditionals be \total{} morphisms, but this
is not possible even in \(\SubStoch\): the ``always fail'' map \(⊥ ፡ A → 0\)
cannot have a \total{} \conditional{}. However, we may impose a weaker
condition: that \conditionals{} have a \emph{deterministic domain}—in other
words, that they are \emph{\kl{quasi-total}}. This is a route we have explored
in previous work \cite{dr:evidentialDecision23}; in this manuscript, we study
\kl[quasi-total]{quasi-totality} as a separate property.

\begin{defi}[Quasi-total morphism, {{\cite{dr:evidentialDecision23}}}]
  A morphism in a \kl{partial Markov category}, \m{f ፡ X → Y}, is
  \intro{quasi-total} when its \kl{domain of definition} is \kl{deterministic}.
  In other words, a morphism is \kl{quasi-total} when it satisfies the equation in
  \Cref{fig:domain-of-definition}. 
\end{defi}

\begin{figure}[h!]

\tikzset{every picture/.style={line width=0.75pt}} %

\begin{tikzpicture}[x=0.75pt,y=0.75pt,yscale=-1,xscale=1]
\draw    (130,30) .. controls (120.17,34.43) and (115,34.47) .. (115,45) ;
\draw    (130,30) .. controls (139.88,34.47) and (145,34.72) .. (145,45) ;
\draw   (200,40) -- (220,40) -- (220,55) -- (200,55) -- cycle ;
\draw    (210,40) -- (210,20) ;
\draw    (210,64.5) -- (210,54.5) ;
\draw  [fill={rgb, 255:red, 0; green, 0; blue, 0 }  ,fill opacity=1 ] (207,64.5) .. controls (207,62.84) and (208.34,61.5) .. (210,61.5) .. controls (211.66,61.5) and (213,62.84) .. (213,64.5) .. controls (213,66.16) and (211.66,67.5) .. (210,67.5) .. controls (208.34,67.5) and (207,66.16) .. (207,64.5) -- cycle ;
\draw  [fill={rgb, 255:red, 255; green, 255; blue, 255 }  ,fill opacity=1 ] (105,40) -- (125,40) -- (125,55) -- (105,55) -- cycle ;
\draw  [fill={rgb, 255:red, 255; green, 255; blue, 255 }  ,fill opacity=1 ] (140,40) -- (160,40) -- (160,55) -- (140,55) -- cycle ;
\draw    (130,30) -- (130,20) ;
\draw  [fill={rgb, 255:red, 0; green, 0; blue, 0 }  ,fill opacity=1 ] (127,30) .. controls (127,28.34) and (128.34,27) .. (130,27) .. controls (131.66,27) and (133,28.34) .. (133,30) .. controls (133,31.66) and (131.66,33) .. (130,33) .. controls (128.34,33) and (127,31.66) .. (127,30) -- cycle ;
\draw    (115,65) -- (115,55) ;
\draw    (150,65) -- (150,55) ;
\draw  [draw opacity=0] (165,30) -- (195,30) -- (195,55) -- (165,55) -- cycle ;
\draw  [fill={rgb, 255:red, 0; green, 0; blue, 0 }  ,fill opacity=1 ] (112,65) .. controls (112,63.34) and (113.34,62) .. (115,62) .. controls (116.66,62) and (118,63.34) .. (118,65) .. controls (118,66.66) and (116.66,68) .. (115,68) .. controls (113.34,68) and (112,66.66) .. (112,65) -- cycle ;
\draw  [fill={rgb, 255:red, 0; green, 0; blue, 0 }  ,fill opacity=1 ] (147,65) .. controls (147,63.34) and (148.34,62) .. (150,62) .. controls (151.66,62) and (153,63.34) .. (153,65) .. controls (153,66.66) and (151.66,68) .. (150,68) .. controls (148.34,68) and (147,66.66) .. (147,65) -- cycle ;
\draw  [draw opacity=0] (190,5) -- (230,5) -- (230,80) -- (190,80) -- cycle ;
\draw  [draw opacity=0] (100,5) -- (160,5) -- (160,80) -- (100,80) -- cycle ;

\draw (180,42.5) node    {$=$};
\draw (210,47.5) node  [font=\footnotesize]  {$f$};
\draw (115,47.5) node  [font=\footnotesize]  {$f$};
\draw (150,47.5) node  [font=\footnotesize]  {$f$};
\draw (210,16.6) node [anchor=south] [inner sep=0.75pt]  [font=\tiny]  {$X$};
\draw (130,16.6) node [anchor=south] [inner sep=0.75pt]  [font=\tiny]  {$X$};

\end{tikzpicture}
   \caption{\protect\kl{Quasi-total morphism}.}
  \label{fig:domain-of-definition}
\end{figure}

\begin{lem}[Quasi-total means self-normalizing]
  \label{prop:characterisation-quasi-total}%
  \label{lemma:characterisation-quasi-totality}%
  In a \partialMarkovCategory{}, a morphism \(f ፡ X → Y\) is \kl{quasi-total} if
  and only if it is its own \normalisation{}:
  $f = (f ⨾ ε) ⊲ f$ (\Cref{fig:characterisation-quasitotal}).
  \begin{figure}[ht]

\tikzset{every picture/.style={line width=0.75pt}} %

\begin{tikzpicture}[x=0.75pt,y=0.75pt,yscale=-1,xscale=1]
\draw   (135,40) -- (155,40) -- (155,55) -- (135,55) -- cycle ;
\draw    (145,40) -- (145,20) ;
\draw    (145,70) -- (145,55) ;
\draw   (195,40) -- (215,40) -- (215,55) -- (195,55) -- cycle ;
\draw   (225,40) -- (245,40) -- (245,55) -- (225,55) -- cycle ;
\draw    (220,30) -- (220,20) ;
\draw    (220,30) .. controls (210.17,34.43) and (205,29.47) .. (205,40) ;
\draw  [fill={rgb, 255:red, 0; green, 0; blue, 0 }  ,fill opacity=1 ] (217,30) .. controls (217,28.34) and (218.34,27) .. (220,27) .. controls (221.66,27) and (223,28.34) .. (223,30) .. controls (223,31.66) and (221.66,33) .. (220,33) .. controls (218.34,33) and (217,31.66) .. (217,30) -- cycle ;
\draw    (220,30) .. controls (229.88,34.47) and (235,29.72) .. (235,40) ;
\draw    (205,70) -- (205,55) ;
\draw    (235,65) -- (235,55) ;
\draw  [draw opacity=0] (160,30) -- (190,30) -- (190,55) -- (160,55) -- cycle ;
\draw  [draw opacity=0] (245,30) -- (275,30) -- (275,55) -- (245,55) -- cycle ;
\draw  [fill={rgb, 255:red, 0; green, 0; blue, 0 }  ,fill opacity=1 ] (232,65) .. controls (232,63.34) and (233.34,62) .. (235,62) .. controls (236.66,62) and (238,63.34) .. (238,65) .. controls (238,66.66) and (236.66,68) .. (235,68) .. controls (233.34,68) and (232,66.66) .. (232,65) -- cycle ;

\draw (175,42.5) node    {$=$};
\draw (145,47.5) node  [font=\footnotesize]  {$f$};
\draw (260,42.5) node    {$;$};
\draw (205,47.5) node  [font=\footnotesize]  {$f$};
\draw (233.5,47.5) node  [font=\footnotesize]  {$f$};
\draw (145,16.6) node [anchor=south] [inner sep=0.75pt]  [font=\tiny]  {$X$};
\draw (220,16.6) node [anchor=south] [inner sep=0.75pt]  [font=\tiny]  {$X$};
\draw (145,73.4) node [anchor=north] [inner sep=0.75pt]  [font=\tiny]  {$Y$};
\draw (205,73.4) node [anchor=north] [inner sep=0.75pt]  [font=\tiny]  {$Y$};

\end{tikzpicture}
     \caption{Self-normalizing morphism.}
    \label{fig:characterisation-quasitotal}
  \end{figure}
\end{lem}
\begin{proof}
  Let us show the first implication: whenever the equation in
  \Cref{fig:characterisation-quasitotal} holds, we can discard the outputs on
  both sides of the equation and conclude that the \domain{} must be
  \deterministic{}: $f = (f ⨾ ε) ⊲ f$ implies 
  $$(f ⨾ ε) = (f ⨾ ε) ⊲ (f ⨾ ε) = ν ⨾ ((f ⨾ ε) ⊗ (f ⨾ ε)).$$
  For the opposite implication, we employ the existence of a \normalisation{} of $f$.
  $$
  f = (f ⨾ ε) ⊲ \normal{f} = (f ⨾ ε) ⊲ (f ⨾ ε) ⊲ \normal{f} = (f ⨾ ε) ⊲ f.
  $$
  We conclude that any \deterministicdomain{} morphism is its own \normalisation{}.
\end{proof}

\begin{exa}[Quasi-totality of substochastic channels]
  A substochastic channel, $f ፡ X → \subdistr(Y)$, is \kl{quasi-total} in
  \SubStoch{} if and only if it factors through the inclusion of distributions
  and failure into \subdistributions{}, \(\distr(X)+1 \into \subdistr(X)\).
\end{exa}

\begin{rem}
  All \kl{deterministic} morphisms are
  \kl{quasi-total} morphisms:
  $$(f ⨾ ε) = f ⨾ ν ⨾ (ε ⊗ ε) = ν ⨾ ((f ⨾ ε) ⊗ (f ⨾ ε)).$$
  In particular, all morphisms in a cartesian restriction category are \kl{quasi-total}.
  All \kl{total} morphisms are \kl{quasi-total}:
  \[(f ⨾ ε) = ε = ν ⨾ (ε ⊗ ε) = ν ⨾ ((f ⨾ ε) ⊗ (f ⨾ ε)).\]
  And, in particular, all morphisms in a \MarkovCategory{} are \kl{quasi-total}.
\end{rem}

\begin{rem}[Quasi-total morphisms and quasi-Markov]%
  \label{rem:quasi-markov}
  \kl{Quasi-total morphisms} of a \kl{partial Markov category} do not form a
  subcategory in general. Asking them to be closed under composition rules out
  important examples like the Kleisli category of the subdistribution monad. Still, \kl{quasi-total} subdistributions form a category with a different
  composition. In contrast to the \emph{updating semantics} of subdistributions,
  this other composition, ($•$), follows the so-called \emph{black-hole
  semantics} \cite{fritz2009convex,sokolova18:termination}: any possibility of
  failure is equated to failure. 
  \[
  (f • g)(z \given x) = \begin{cases}
    ⊥, & \mbox{ when } \exists y ∈ Y.\ f(y \given x) · g(⊥ \given y) > 0, \\
    f(y \given x) · g(z \given y), & \mbox{ otherwise. }
  \end{cases}
  \]
  The resulting category of partial distributions is the leading example of
  \emph{quasi-Markov categories}: \kl{partial Markov categories} where all
  morphisms are \kl{quasi-total}, and which have found applications in
  categorical probability \cite{shahmohammed25:empirical}.  During this text,
  however, we focus on updating semantics—of which black-hole semantics
  constitute the edge case where all updates fail. For a detailed
  treatment and general construction of stochastic partiality with black-hole
  semantics, we refer the reader to the recent work of Shah Mohammed on
  \emph{quasi-Markov categories} \cite{shahmohammed2025partializations}.
\end{rem}

\begin{rem}
Previous work~\cite{dr:evidentialDecision23} has also investigated a slightly
stronger definition of \kl{conditionals}, on which they are asked to be
\kl{quasi-total}. In that case, not developed here,
\kl{conditionals} may be defined with a weaker notion of marginal.
\end{rem}

\begin{lem}[Marginals from quasi-totality]
  \label{lemma:marginals-qt-conditionals}%
  In a \partialMarkovCategory{}, if a morphism can be factored as the
  conditional composition, $f = (g ⊲ c)$, of an arbitrary $g ፡ X → A$ and a
  \kl{quasi-total} \m{c ፡ A ⊗ X → B}, then this \kl{quasi-total} map is a
  \kl{conditional}, meaning that \(f = (f ⨾ π₁) ⊲ c\).
\end{lem}
\begin{proof}
  Let us prove that it is indeed a \kl{conditional}, $f = (f ⨾ π_1) \mathbin{⊲}
  c$.  We employ \emph{(i)} the factorization of $f$, by assumption; \emph{(ii)}
  associativity; \emph{(iii)} that $c$ is \kl{quasi-total}, as characterized in
  \Cref{lemma:characterisation-quasi-totality}; and \emph{(iv)} again the
  decomposition of $f$.
  \begin{align*}
    & (f ⨾ π₁) ⊲ c \overset{(i)}{=} 
    ((m ⊲ c) ⨾ π₁) ⊲ c \overset{(ii)}{=} 
    m ⊲ (c ⨾ π₁) ⊲ c \overset{(iii)}{=} 
    m ⊲ c = f. \qedhere
  \end{align*}
\end{proof}

After developing the algebra of \kl{conditionals} in \kl{partial Markov
categories}, and considering again the example in \Cref{ex:non-example-markov},
one can conjecture that there may exist a procedure to construct \conditionals{} in Kleisli categories of Maybe monads over \MarkovCategories{}. We show that
this is indeed the case.
\subsection{Kleisli categories of Maybe monads}
\label{sec:KleisliMaybe}

We prove in this section that the Kleisli category of the Maybe monad over a
\MarkovCategory{} has \conditionals{}; these \conditionals{} are inherited from
the base distributive category. In other words, we provide a recipe for constructing
\partialMarkovCategories{} from \MarkovCategories{} with coproducts.

\begin{defi}[{Distributive Markov category, c.f.~{{\cite{kaddar24:interfaces,liell24:imprecise}}}}]%
  \label{def:distributiveMarkov}%
  A \intro{distributive Markov category} is a \MarkovCategory{} with coproducts
  that have \deterministic{} injections, $\iota_1 ፡ X → X + Y$ and $\iota_2 ፡ Y → X + Y$, and such that the canonical distributor, $d ፡ X ⊗ Z + Y ⊗ Z → (X + Y) ⊗ Z$, is an isomorphism.\footnote{Note that we ask
  \protect\MarkovCategories{} to have \conditionals{}: in the literature,
  \distributiveMarkovCategories{} can be found as ``distributive Markov categories with
  conditionals'' \cite{kaddar24:interfaces,liell24:imprecise}.}
\end{defi}

The Maybe monad on a \distributiveMarkovCategory{} is symmetric lax-monoidal with respect to the monoidal tensor $(⊗)$. This implies that its Kleisli category is a \copyDiscardCategory{} because of the identity-on-objects monoidal functor from the base category \cite{guitart1980tenseurs}; let us prove next that it also has \kl{conditionals}, making it a \partialMarkovCategory{}. 

For this proof, we use multiple results that we will detail during the rest of
the section. The main idea is that morphisms of the form \(f \colon X \to (Y
\tensor Z) + 1\) induce morphisms \((f ⨾ 𝔰) ፡ X → (Y+1) ⊗ (Z+1)\). Since the
original category had \kl{conditionals}, we obtain a \kl{marginal} $(f ⨾ 𝔰 ⨾
\pi_{1}) ፡ X → Y + 1$, and a \kl{conditional}, $c ፡ X ⊗ (Y+1) → Z+1$, such that
\((f ⨾ 𝔰 ⨾ \pi_{1}) ⊲ c = f ⨾ 𝔰\). The rest of the proof constructs a
\kl{conditional} in the Kleisli category of the Maybe monad, extending this morphism.

We will need the notation \((-)^{\ast}\) to denote the strong Kleisli
extension~\cite{mcdermott22:strongkleisli} of a strong monad \(T\): this is the
operation lifting a morphism $f ፡ X ⊗ Y → T(Z)$ to a morphism $f^{\ast} ፡ X ⊗
T(Y) → T(Z)$. In particular, we call \emph{strength} the lifting of the unit,
$\mathrm{str} ፡ X ⊗ T(Y) → T(X ⊗ Y)$.

\begin{lem}%
  \label{lemma:conditionalReplace}
  Let $g ፡ X ⊗ (Y + 1) → Z + 1$ and $f ፡ X → Y + 1$ be two morphisms in a
  \kl{distributive Markov category}. Consider the morphism \(g_{1} = (\id{}
  \tensor \iota_1) \dcomp g \colon X \tensor Y \to Z +1\) and its strong
  Kleisli extension \(g_{1}^{\ast} \colon X \tensor (Y+1) \to Z + 1\). Then,
  $$(f ⊲ g) ⨾ ℓ = (f ⊲ g_{1}^{\ast}) ⨾ ℓ,$$
  where \(ℓ_{Y,Z} \colon (Y+1) \tensor (Z+1) \to (Y \tensor Z) +1\) is the natural transformation making the Maybe monad monoidal.
\end{lem}
\begin{proof}
  We must prove an equality between two morphisms of type $X → (Y ⊗ Z) + 1$, but we can prove a slightly stronger claim: there are two equal morphisms of type
  $X ⊗ (Y + 1) → (Y ⊗ Z) + 1$ that, when precomposed by $ν_X ⨾ (\id{} \tensor f)$, give rise to both sides of the equation.
  That is, if we write $ν_{U}^{σ} ፡ X ⊗ U → U ⊗ X ⊗ U$ for copying and symmetry, $ν^{σ}_U = (\id{} ⊗ ν_U) ⨾ (σ_{X,U} ⊗ \id{U})$, we will prove that
  \[ν_{(Y + 1)}^{σ} ⨾ (\id{} ⊗ g) ⨾ ℓ = ν_{(Y + 1)}^{σ} ⨾ (\id{} ⊗ g_{1}^{\ast}) ⨾ ℓ.\]
  From which the original statement follows,
  \begin{align*}
    (f ⊲ g) ⨾ ℓ
    = ν_X ⨾ f ⨾ ν_{(Y + 1)}^{σ} ⨾ (\id{} ⊗ g) ⨾ ℓ 
     = ν_X ⨾ f ⨾ ν_{(Y + 1)}^{σ} ⨾ (\id{} ⊗ g_{1}^{\ast}) ⨾ ℓ
     = (f ⊲ g_{1}^{\ast}) ⨾ ℓ.
  \end{align*}
    
  By the universal property of coproducts, proving an equality between morphisms of type $X ⊗ (Y + 1) → (Y ⊗ Z) + 1$  is the same as proving two equalities between the two factor morphisms, of type $X ⊗ Y → (Y ⊗ Z) + 1$ and $X → (Y ⊗ Z) + 1$.
  In other words, we must only prove the following two equations.
  \begin{align}
    (\id{X} \tensor \iota_1) ⨾ ν_{(Y + 1)}^{σ} ⨾ (\id{} ⊗ g) ⨾ ℓ &= (\id{X} \tensor \iota_1) ⨾ ν_{(Y + 1)}^{σ} ⨾ (\id{} ⊗ g_{1}^{\ast}) ⨾ ℓ \label{eq:proof-first-inclusion}\\
    (\id{X} \tensor \iota_2) ⨾ ν_{(Y + 1)}^{σ} ⨾ (\id{} ⊗ g) ⨾ ℓ &= (\id{X} \tensor \iota_2) ⨾ ν_{(Y + 1)}^{σ} ⨾ (\id{} ⊗ g_{1}^{\ast}) ⨾ ℓ \label{eq:proof-second-inclusion}
  \end{align}
  We start by proving~\eqref{eq:proof-first-inclusion}, \(g\) and \(g_1^{\ast}\) coincide when precomposed with the first inclusion:
  \begin{equation}
    \label{eq:proof-maybe-inclusions}
    (\id{X} \tensor \iota_1) \dcomp g_1^{\ast} = g_1 = (\id{X} \tensor \iota_1) \dcomp g .
  \end{equation}
  Explicitly, we define \(g_1^{\ast} = \mathrm{str}_{X,Y} \dcomp (g_1 + \id{1}) \dcomp \mu_Z\), where \(\mathrm{str}_{X,Y} \colon X \tensor (Y+1) \to (X \tensor Y) +1\) is the strength morphism of the Maybe monad.
  We use that \emph{(i)} the definition of \(g_1^{\ast}\), \emph{(ii)} the definition of the unit of the Maybe monad, \emph{(iii)} the properties of strength, \emph{(iv)} naturality of the unit, \emph{(v)} unitality of the Maybe monad. 
  \begin{align*}
    & (\id{X} \tensor \iota_1) \dcomp g_1^{\ast} & \hint{i}{=} \\
    & (\id{X} \tensor \iota_1) \dcomp \mathrm{str}_{X,Y} \dcomp (g_1 + \id{1}) \dcomp \mu_Z & \hint{ii}{=} \\
    & (\id{X} \tensor \eta_Y) \dcomp \mathrm{str}_{X,Y} \dcomp (g_1 + \id{1}) \dcomp \mu_Z & \hint{iii}{=} \\
    & \eta_{X \tensor Y} \dcomp (g_1 + \id{1}) \dcomp \mu_Z & \hint{iv}{=} \\
    & g_1 \dcomp \eta_{Z+1} \dcomp \mu_Z & \hint{v}{=} \\
    & g_1
  \end{align*}
  Then, we obtain \Cref{eq:proof-first-inclusion}; we use\ \eqref{eq:proof-maybe-1} determinism of the inclusions and naturality of symmetries,\ \eqref{eq:proof-maybe-2} interchange,\ \eqref{eq:proof-maybe-3} the definition of $g_1$ and~\Cref{eq:proof-maybe-inclusions}, and\ \eqref{eq:proof-maybe-4} the definition of strength.  \begin{align}
    & (\id{X} \tensor \iota_1) ⨾ ν_{(Y + 1)}^{σ} ⨾ (\id{} ⊗ g) ⨾ ℓ & & (\id{X} \tensor \iota_1) ⨾ ν_{(Y + 1)}^{σ} ⨾ (\id{} ⊗ g_1^{\ast}) ⨾ ℓ \nonumber\\
    & = ν^{σ}_Y ⨾ (\iota_1 ⊗ \id{} ⊗ \iota_1) ⨾ (\id{} ⊗ g) ⨾ ℓ & & =  ν^{σ}_Y ⨾ (\iota_1 ⊗ \id{} ⊗ \iota_1) ⨾ (\id{} ⊗ g_1^{\ast}) ⨾ ℓ \label{eq:proof-maybe-1} \\
    & = ν^{σ}_Y ⨾ (\id{} ⊗ \iota_1) ⨾ (\id{} ⊗ g) ⨾ (\iota_1 ⊗ \id{}) ⨾ ℓ & & = ν^{σ}_Y ⨾ (\id{} ⊗ \iota_1) ⨾ (\id{} ⊗ g_1^{\ast}) ⨾ (\iota_1 ⊗ \id{}) ⨾ ℓ  \label{eq:proof-maybe-2} \\
    & = ν^{σ}_Y ⨾ (\id{} ⊗ g_1) ⨾ (\iota_1 ⊗ \id{}) ⨾ ℓ & & = ν^{σ}_Y ⨾ (\id{} ⊗ g_1) ⨾ (\iota_1 ⊗ \id{}) ⨾ ℓ \label{eq:proof-maybe-3} \\
    & = ν^{σ}_Y ⨾ (\id{} ⊗ g_1) ⨾ \mathrm{str} & & = ν^{σ}_Y ⨾ (\id{} ⊗ g_1) ⨾ \mathrm{str} \label{eq:proof-maybe-4} 
  \end{align}
  The same steps may be seen in string diagrams.
  \begin{center}
\tikzset{every picture/.style={line width=0.75pt}} %

\begin{tikzpicture}[x=0.75pt,y=0.75pt,yscale=-1,xscale=1]
\draw   (50,28.5) -- (30,28.5) -- (30,43.5) -- (50,43.5) -- cycle ;
\draw    (40,28) -- (40,18) ;
\draw    (40,53) -- (40,43) ;
\draw    (40,53) .. controls (49.83,57.43) and (50,58.47) .. (50,69) ;
\draw  [fill={rgb, 255:red, 0; green, 0; blue, 0 }  ,fill opacity=1 ] (43,53) .. controls (43,51.34) and (41.66,50) .. (40,50) .. controls (38.34,50) and (37,51.34) .. (37,53) .. controls (37,54.66) and (38.34,56) .. (40,56) .. controls (41.66,56) and (43,54.66) .. (43,53) -- cycle ;
\draw    (40,53) .. controls (30.13,57.47) and (30,58.72) .. (30,69) ;
\draw    (15,69) -- (15,19) ;
\draw    (30,69) .. controls (29.8,78.85) and (15,73.72) .. (15,84) ;
\draw    (15,69) .. controls (14.63,78.8) and (35,73.47) .. (35,84) ;
\draw    (50,84) -- (50,69) ;
\draw   (55,84) -- (30,84) -- (30,99) -- (55,99) -- cycle ;
\draw   (30,114) -- (10,114) -- (10,129) -- (30,129) -- cycle ;
\draw    (40,99) .. controls (40.06,108.69) and (25,103.47) .. (25,114) ;
\draw    (15,114) -- (15,84) ;
\draw    (20,139) -- (20,129) ;
\draw   (110,69.5) -- (90,69.5) -- (90,84.5) -- (110,84.5) -- cycle ;
\draw    (120,69) -- (120,59) ;
\draw    (125,29) -- (125,19) ;
\draw    (125,30) .. controls (134.83,34.43) and (135,33.97) .. (135,44.5) ;
\draw  [fill={rgb, 255:red, 0; green, 0; blue, 0 }  ,fill opacity=1 ] (128,29) .. controls (128,27.34) and (126.66,26) .. (125,26) .. controls (123.34,26) and (122,27.34) .. (122,29) .. controls (122,30.66) and (123.34,32) .. (125,32) .. controls (126.66,32) and (128,30.66) .. (128,29) -- cycle ;
\draw    (125,29) .. controls (115.13,33.47) and (115,34.72) .. (115,45) ;
\draw    (100,44) -- (100,19) ;
\draw    (115,45) .. controls (114.8,54.85) and (100,50.72) .. (100,61) ;
\draw    (100,44) .. controls (99.63,53.8) and (120,49.47) .. (120,60) ;
\draw    (135,69) -- (135,59) ;
\draw   (140,69) -- (115,69) -- (115,84) -- (140,84) -- cycle ;
\draw   (115,98) -- (95,98) -- (95,113) -- (115,113) -- cycle ;
\draw    (125,84) .. controls (125.06,93.69) and (110,87.47) .. (110,98) ;
\draw    (100,98) -- (100,84) ;
\draw    (105,139) -- (105,113) ;
\draw   (145,44.5) -- (125,44.5) -- (125,59.5) -- (145,59.5) -- cycle ;
\draw    (100,69.5) -- (100,61) ;
\draw    (200,69) -- (200,59) ;
\draw    (205,29) -- (205,19) ;
\draw    (205,30) .. controls (214.83,34.43) and (215,33.97) .. (215,44.5) ;
\draw  [fill={rgb, 255:red, 0; green, 0; blue, 0 }  ,fill opacity=1 ] (208,29) .. controls (208,27.34) and (206.66,26) .. (205,26) .. controls (203.34,26) and (202,27.34) .. (202,29) .. controls (202,30.66) and (203.34,32) .. (205,32) .. controls (206.66,32) and (208,30.66) .. (208,29) -- cycle ;
\draw    (205,29) .. controls (195.13,33.47) and (195,34.72) .. (195,45) ;
\draw    (180,44) -- (180,19) ;
\draw    (195,45) .. controls (194.8,54.85) and (180,50.72) .. (180,61) ;
\draw    (180,44) .. controls (179.63,53.8) and (200,49.47) .. (200,60) ;
\draw    (215,69) -- (215,44.5) ;
\draw   (220,69) -- (195,69) -- (195,84) -- (220,84) -- cycle ;
\draw   (195,98) -- (175,98) -- (175,113) -- (195,113) -- cycle ;
\draw    (205,84) .. controls (205.06,93.69) and (190,87.47) .. (190,98) ;
\draw    (180,98) -- (180,69) ;
\draw    (185,139) -- (185,113) ;
\draw    (180,69.5) -- (180,61) ;
\draw   (275,69.5) -- (255,69.5) -- (255,84.5) -- (275,84.5) -- cycle ;
\draw    (285,69) -- (285,59) ;
\draw    (290,29) -- (290,19) ;
\draw    (290,30) .. controls (299.83,34.43) and (300,33.97) .. (300,44.5) ;
\draw  [fill={rgb, 255:red, 0; green, 0; blue, 0 }  ,fill opacity=1 ] (293,29) .. controls (293,27.34) and (291.66,26) .. (290,26) .. controls (288.34,26) and (287,27.34) .. (287,29) .. controls (287,30.66) and (288.34,32) .. (290,32) .. controls (291.66,32) and (293,30.66) .. (293,29) -- cycle ;
\draw    (290,29) .. controls (280.13,33.47) and (280,34.72) .. (280,45) ;
\draw    (265,44) -- (265,19) ;
\draw    (280,45) .. controls (279.8,54.85) and (265,50.72) .. (265,61) ;
\draw    (265,44) .. controls (264.63,53.8) and (285,49.47) .. (285,60) ;
\draw    (300,69) -- (300,59) ;
\draw   (305,69) -- (280,69) -- (280,84) -- (305,84) -- cycle ;
\draw   (280,98) -- (260,98) -- (260,113) -- (280,113) -- cycle ;
\draw    (290,84) .. controls (290.06,93.69) and (275,87.47) .. (275,98) ;
\draw    (265,98) -- (265,84) ;
\draw    (270,139) -- (270,113) ;
\draw   (310,44.5) -- (290,44.5) -- (290,59.5) -- (310,59.5) -- cycle ;
\draw    (265,69.5) -- (265,61) ;
\draw   (380,28.5) -- (360,28.5) -- (360,43.5) -- (380,43.5) -- cycle ;
\draw    (370,28) -- (370,18) ;
\draw    (370,53) -- (370,43) ;
\draw    (370,53) .. controls (379.83,57.43) and (380,58.47) .. (380,69) ;
\draw  [fill={rgb, 255:red, 0; green, 0; blue, 0 }  ,fill opacity=1 ] (373,53) .. controls (373,51.34) and (371.66,50) .. (370,50) .. controls (368.34,50) and (367,51.34) .. (367,53) .. controls (367,54.66) and (368.34,56) .. (370,56) .. controls (371.66,56) and (373,54.66) .. (373,53) -- cycle ;
\draw    (370,53) .. controls (360.13,57.47) and (360,58.72) .. (360,69) ;
\draw    (345,69) -- (345,19) ;
\draw    (360,69) .. controls (359.8,78.85) and (345,73.72) .. (345,84) ;
\draw    (345,69) .. controls (344.63,78.8) and (365,73.47) .. (365,84) ;
\draw    (380,84) -- (380,69) ;
\draw   (385,84) -- (360,84) -- (360,99) -- (385,99) -- cycle ;
\draw   (360,114) -- (340,114) -- (340,129) -- (360,129) -- cycle ;
\draw    (370,99) .. controls (370.06,108.69) and (355,103.47) .. (355,114) ;
\draw    (345,114) -- (345,84) ;
\draw    (350,139) -- (350,129) ;

\draw (40,36) node  [font=\footnotesize]  {$\iota_1$};
\draw (42.5,91.5) node  [font=\footnotesize]  {$g$};
\draw (20,121.5) node  [font=\footnotesize]  {$ℓ$};
\draw (135,52) node  [font=\footnotesize]  {$\iota_1$};
\draw (100,77) node  [font=\footnotesize]  {$\iota_1$};
\draw (127.5,76.5) node  [font=\footnotesize]  {$g$};
\draw (105,105.5) node  [font=\footnotesize]  {$ℓ$};
\draw (72.5,74) node    {$=$};
\draw (207.5,76.5) node  [font=\footnotesize]  {$g_1$};
\draw (185,105.5) node  [font=\footnotesize]  {$\mathrm{str}$};
\draw (157.5,74) node    {$=$};
\draw (300,52) node  [font=\footnotesize]  {$\iota_1$};
\draw (265,77) node  [font=\footnotesize]  {$\iota_1$};
\draw (292.5,76.5) node  [font=\footnotesize]  {$g_1^{\ast}$};
\draw (270,105.5) node  [font=\footnotesize]  {$ℓ$};
\draw (237.5,74) node    {$=$};
\draw (322.5,74) node    {$=$};
\draw (370,36) node  [font=\footnotesize]  {$\iota_1$};
\draw (372.5,91.5) node  [font=\footnotesize]  {$g_1^{\ast}$};
\draw (350,121.5) node  [font=\footnotesize]  {$ℓ$};
\draw (20,142.4) node [anchor=north] [inner sep=0.75pt]  [font=\tiny]  {$YZ+1$};
\draw (105,142.4) node [anchor=north] [inner sep=0.75pt]  [font=\tiny]  {$YZ+1$};
\draw (185,142.4) node [anchor=north] [inner sep=0.75pt]  [font=\tiny]  {$YZ+1$};
\draw (270,142.4) node [anchor=north] [inner sep=0.75pt]  [font=\tiny]  {$YZ+1$};
\draw (350,142.4) node [anchor=north] [inner sep=0.75pt]  [font=\tiny]  {$YZ+1$};
\draw (15,15.6) node [anchor=south] [inner sep=0.75pt]  [font=\tiny]  {$X$};
\draw (40,14.6) node [anchor=south] [inner sep=0.75pt]  [font=\tiny]  {$Y$};
\draw (100,15.6) node [anchor=south] [inner sep=0.75pt]  [font=\tiny]  {$X$};
\draw (125,15.6) node [anchor=south] [inner sep=0.75pt]  [font=\tiny]  {$Y$};
\draw (180,15.6) node [anchor=south] [inner sep=0.75pt]  [font=\tiny]  {$X$};
\draw (205,15.6) node [anchor=south] [inner sep=0.75pt]  [font=\tiny]  {$Y$};
\draw (265,15.6) node [anchor=south] [inner sep=0.75pt]  [font=\tiny]  {$X$};
\draw (290,15.6) node [anchor=south] [inner sep=0.75pt]  [font=\tiny]  {$Y$};
\draw (345,15.6) node [anchor=south] [inner sep=0.75pt]  [font=\tiny]  {$X$};
\draw (370,14.6) node [anchor=south] [inner sep=0.75pt]  [font=\tiny]  {$Y$};

\end{tikzpicture}
   \end{center}
  For~\Cref{eq:proof-second-inclusion}, we prove that both sides equal \(\discard_X \dcomp  \iota_2\). 
  We use~\eqref{eq:proof-maybe-5} determinism of the inclusions in a \kl{distributive Markov category} and naturality of the symmetries,\ \eqref{eq:proof-maybe-6} interchange,\ \eqref{eq:proof-maybe-7} definition of the laxator, and~\eqref{eq:proof-maybe-8} totality of morphisms in \MarkovCategories{}.
  \begin{align}
    & (\id{X} \tensor \iota_2) ⨾ ν_{(Y + 1)}^{σ} ⨾ (\id{} ⊗ g) ⨾ ℓ && (\id{X} \tensor \iota_2) ⨾ ν_{(Y + 1)}^{σ} ⨾ (\id{} ⊗ g_1^{\ast}) ⨾ ℓ \nonumber\\
    &= (\iota_2 ⊗ \id{} ⊗ \iota_2) ⨾ (\id{} ⊗ g) ⨾ ℓ &&= (\iota_2 ⊗ \id{} ⊗ \iota_2) ⨾ (\id{} ⊗ g_1^{\ast}) ⨾ ℓ  \label{eq:proof-maybe-5}\\
    &= (\id{} \tensor (\iota_{2} \dcomp g)) \dcomp (\iota_{2} \tensor \id{}) \dcomp ℓ &&= (\id{} \tensor (\iota_{2} \dcomp g_1^{\ast})) \dcomp (\iota_{2} \tensor \id{}) \dcomp ℓ \label{eq:proof-maybe-6}\\
    &= (\id{} \tensor (\iota_{2} \dcomp g)) \dcomp \discard_{Z+1} \dcomp \iota_{2} &&= (\id{} \tensor (\iota_{2} \dcomp g_1^{\ast})) \dcomp \discard_{Z+1} \dcomp \iota_{2} \label{eq:proof-maybe-7}\\
    &= \discard_X \dcomp \iota_{2} &&= \discard_X \dcomp \iota_{2} \label{eq:proof-maybe-8}
  \end{align}
  The same steps may be seen in string diagrams.
  \begin{center}
\tikzset{every picture/.style={line width=0.75pt}} %

\begin{tikzpicture}[x=0.75pt,y=0.75pt,yscale=-1,xscale=1]
\draw   (50,29.5) -- (30,29.5) -- (30,44.5) -- (50,44.5) -- cycle ;
\draw    (40,55) -- (40,45) ;
\draw    (40,55) .. controls (49.83,59.43) and (50,60.47) .. (50,71) ;
\draw  [fill={rgb, 255:red, 0; green, 0; blue, 0 }  ,fill opacity=1 ] (43,55) .. controls (43,53.34) and (41.66,52) .. (40,52) .. controls (38.34,52) and (37,53.34) .. (37,55) .. controls (37,56.66) and (38.34,58) .. (40,58) .. controls (41.66,58) and (43,56.66) .. (43,55) -- cycle ;
\draw    (40,55) .. controls (30.13,59.47) and (30,60.72) .. (30,71) ;
\draw    (15,70) -- (15,20) ;
\draw    (30,71) .. controls (29.8,80.85) and (15,75.72) .. (15,86) ;
\draw    (15,70) .. controls (14.63,79.8) and (35,74.47) .. (35,85) ;
\draw    (50,86) -- (50,71) ;
\draw   (55,85.5) -- (30,85.5) -- (30,100.5) -- (55,100.5) -- cycle ;
\draw   (30,115.5) -- (10,115.5) -- (10,130.5) -- (30,130.5) -- cycle ;
\draw    (40,100.5) .. controls (40.06,110.19) and (25,104.97) .. (25,115.5) ;
\draw    (15,115.5) -- (15,85.5) ;
\draw    (20,140.5) -- (20,130.5) ;
\draw    (120,70) -- (120,20) ;
\draw    (135,70) -- (135,60) ;
\draw   (140,70) -- (115,70) -- (115,85) -- (140,85) -- cycle ;
\draw   (115,99) -- (95,99) -- (95,114) -- (115,114) -- cycle ;
\draw    (125,85) .. controls (125.06,94.69) and (110,88.47) .. (110,99) ;
\draw    (100,99) -- (100,85) ;
\draw    (105,140) -- (105,114) ;
\draw   (145,45.5) -- (125,45.5) -- (125,60.5) -- (145,60.5) -- cycle ;
\draw   (110,69.5) -- (90,69.5) -- (90,84.5) -- (110,84.5) -- cycle ;
\draw    (185,60) -- (185,20) ;
\draw    (200,59.5) -- (200,49.5) ;
\draw   (205,59.5) -- (180,59.5) -- (180,74.5) -- (205,74.5) -- cycle ;
\draw   (204,99) -- (184,99) -- (184,114) -- (204,114) -- cycle ;
\draw    (194,140) -- (194,114) ;
\draw   (210,35) -- (190,35) -- (190,50) -- (210,50) -- cycle ;
\draw    (193,85) -- (193,75) ;
\draw  [fill={rgb, 255:red, 0; green, 0; blue, 0 }  ,fill opacity=1 ] (196,85) .. controls (196,83.34) and (194.66,82) .. (193,82) .. controls (191.34,82) and (190,83.34) .. (190,85) .. controls (190,86.66) and (191.34,88) .. (193,88) .. controls (194.66,88) and (196,86.66) .. (196,85) -- cycle ;
\draw    (255,50) -- (255,20) ;
\draw  [fill={rgb, 255:red, 0; green, 0; blue, 0 }  ,fill opacity=1 ] (258,50) .. controls (258,48.34) and (256.66,47) .. (255,47) .. controls (253.34,47) and (252,48.34) .. (252,50) .. controls (252,51.66) and (253.34,53) .. (255,53) .. controls (256.66,53) and (258,51.66) .. (258,50) -- cycle ;
\draw   (266,99) -- (246,99) -- (246,114) -- (266,114) -- cycle ;
\draw    (256,140) -- (256,114) ;
\draw    (305,60) -- (305,20) ;
\draw    (320,59.5) -- (320,49.5) ;
\draw   (325,59.5) -- (300,59.5) -- (300,74.5) -- (325,74.5) -- cycle ;
\draw   (324,99) -- (304,99) -- (304,114) -- (324,114) -- cycle ;
\draw    (314,140) -- (314,114) ;
\draw   (330,35) -- (310,35) -- (310,50) -- (330,50) -- cycle ;
\draw    (313,85) -- (313,75) ;
\draw  [fill={rgb, 255:red, 0; green, 0; blue, 0 }  ,fill opacity=1 ] (316,85) .. controls (316,83.34) and (314.66,82) .. (313,82) .. controls (311.34,82) and (310,83.34) .. (310,85) .. controls (310,86.66) and (311.34,88) .. (313,88) .. controls (314.66,88) and (316,86.66) .. (316,85) -- cycle ;
\draw    (395,70) -- (395,20) ;
\draw    (410,70) -- (410,60) ;
\draw   (415,70) -- (390,70) -- (390,85) -- (415,85) -- cycle ;
\draw   (390,99) -- (370,99) -- (370,114) -- (390,114) -- cycle ;
\draw    (400,85) .. controls (400.06,94.69) and (385,88.47) .. (385,99) ;
\draw    (380,140) -- (380,114) ;
\draw   (420,45.5) -- (400,45.5) -- (400,60.5) -- (420,60.5) -- cycle ;
\draw   (385,69.5) -- (365,69.5) -- (365,84.5) -- (385,84.5) -- cycle ;
\draw    (375,98.5) -- (375,84.5) ;
\draw   (495,30.5) -- (475,30.5) -- (475,45.5) -- (495,45.5) -- cycle ;
\draw    (485,56) -- (485,46) ;
\draw    (485,56) .. controls (494.83,60.43) and (495,59.47) .. (495,70) ;
\draw  [fill={rgb, 255:red, 0; green, 0; blue, 0 }  ,fill opacity=1 ] (488,56) .. controls (488,54.34) and (486.66,53) .. (485,53) .. controls (483.34,53) and (482,54.34) .. (482,56) .. controls (482,57.66) and (483.34,59) .. (485,59) .. controls (486.66,59) and (488,57.66) .. (488,56) -- cycle ;
\draw    (485,56) .. controls (475.13,60.47) and (475,59.72) .. (475,70) ;
\draw    (460,70) -- (460,20) ;
\draw    (475,70) .. controls (474.8,79.85) and (460,74.72) .. (460,85) ;
\draw    (460,70) .. controls (459.63,79.8) and (480,74.47) .. (480,85) ;
\draw    (495,85) -- (495,70) ;
\draw   (500,84.5) -- (475,84.5) -- (475,99.5) -- (500,99.5) -- cycle ;
\draw   (475,114.5) -- (455,114.5) -- (455,129.5) -- (475,129.5) -- cycle ;
\draw    (485,99.5) .. controls (485.06,109.19) and (470,103.97) .. (470,114.5) ;
\draw    (460,114.5) -- (460,84.5) ;
\draw    (465,140) -- (465,129.5) ;

\draw (40,37) node  [font=\footnotesize]  {$\iota_2$};
\draw (42.5,93) node  [font=\footnotesize]  {$g$};
\draw (20,123) node  [font=\footnotesize]  {$ℓ$};
\draw (127.5,77.5) node  [font=\footnotesize]  {$g$};
\draw (105,106.5) node  [font=\footnotesize]  {$ℓ$};
\draw (72.5,75) node    {$=$};
\draw (162.5,75) node    {$=$};
\draw (20,143.9) node [anchor=north] [inner sep=0.75pt]  [font=\tiny]  {$YZ+1$};
\draw (105,142.4) node [anchor=north] [inner sep=0.75pt]  [font=\tiny]  {$YZ+1$};
\draw (15,16.6) node [anchor=south] [inner sep=0.75pt]  [font=\tiny]  {$X$};
\draw (120,16.6) node [anchor=south] [inner sep=0.75pt]  [font=\tiny]  {$X$};
\draw (135,53) node  [font=\footnotesize]  {$\iota_2$};
\draw (100,77) node  [font=\footnotesize]  {$\iota_2$};
\draw (192.5,67) node  [font=\footnotesize]  {$g$};
\draw (194,142.4) node [anchor=north] [inner sep=0.75pt]  [font=\tiny]  {$YZ+1$};
\draw (185,17.6) node [anchor=south] [inner sep=0.75pt]  [font=\tiny]  {$X$};
\draw (200,42.5) node  [font=\footnotesize]  {$\iota_2$};
\draw (194,106.5) node  [font=\footnotesize]  {$\iota_2$};
\draw (227.5,75) node    {$=$};
\draw (256,142.4) node [anchor=north] [inner sep=0.75pt]  [font=\tiny]  {$YZ+1$};
\draw (256,106.5) node  [font=\footnotesize]  {$\iota_2$};
\draw (255,16.6) node [anchor=south] [inner sep=0.75pt]  [font=\tiny]  {$X$};
\draw (282.5,75) node    {$=$};
\draw (312.5,67) node  [font=\footnotesize]  {$g_1^{\ast}$};
\draw (314,142.4) node [anchor=north] [inner sep=0.75pt]  [font=\tiny]  {$YZ+1$};
\draw (305,17.6) node [anchor=south] [inner sep=0.75pt]  [font=\tiny]  {$X$};
\draw (320,42.5) node  [font=\footnotesize]  {$i$};
\draw (314,106.5) node  [font=\footnotesize]  {$i$};
\draw (347.5,75) node    {$=$};
\draw (402.5,77.5) node  [font=\footnotesize]  {$g_1^{\ast}$};
\draw (380,106.5) node  [font=\footnotesize]  {$ℓ$};
\draw (437.5,75) node    {$=$};
\draw (380,142.4) node [anchor=north] [inner sep=0.75pt]  [font=\tiny]  {$YZ+1$};
\draw (395,16.6) node [anchor=south] [inner sep=0.75pt]  [font=\tiny]  {$X$};
\draw (410,53) node  [font=\footnotesize]  {$\iota_2$};
\draw (375,77) node  [font=\footnotesize]  {$\iota_2$};
\draw (485,38) node  [font=\footnotesize]  {$\iota_2$};
\draw (487.5,92) node  [font=\footnotesize]  {$g_1^{\ast}$};
\draw (465,122) node  [font=\footnotesize]  {$ℓ$};
\draw (465,143.4) node [anchor=north] [inner sep=0.75pt]  [font=\tiny]  {$YZ+1$};
\draw (460,16.6) node [anchor=south] [inner sep=0.75pt]  [font=\tiny]  {$X$};

\end{tikzpicture}
   \end{center}
  This completes the proof.
\end{proof}

The next lemma will rewrite the conditionals in the Kleisli category of the Maybe monad in terms of the conditionals in the base category and strong Kleisli extension.

\begin{lem}
  \label{lemma:conditionalKleisli}
  Let $f ፡ X → Y + 1$ and $g ፡ X ⊗ Y → Z + 1$ be morphisms in the Kleisli
  category of the maybe monad over a \kl{distributive Markov category}. Their
  \conditionalComposition{} in the Kleisli category, $(⊲^{\ast})$, can be
  rewritten as
  \[%
  (f ⊲^{\ast} g) = (f ⊲ g^{\ast}) ⨾ ℓ.
  \]
\end{lem}
\begin{proof}
  Let us first write explicitly what \conditionalComposition{} means in the
  Kleisli category of the maybe monad. We write $φ^{\ast}$ for the section to
  marginalization in the Kleisli category, which is
  \[
  φ^{\ast}(f) = ν_X ⨾ (\id{} ⊗ f) ⨾ \mathrm{str}.
  \]
  With this, we rewrite the left-hand side. We use (\emph{i}) the definition of
  \conditionalComposition{}, (\emph{ii}) the section to marginalization,
  (\emph{iii}) composition in the Kleisli category, and (\emph{iv}) the
  definition of projection in the Kleisli category.
  \begin{align*} 
    & f ⊲^{\ast} g & \hint{i}{=} \\
    & φ^{\ast}(f) ⨾^{\ast} φ^{\ast}(g) ⨾^{\ast} \pi^{\ast} & \hint{ii}{=} \\
    & (ν_X ⨾ (\id{} ⊗ f) ⨾ \mathrm{str}) ⨾^{\ast} (ν_{X ⊗ Y} ⨾ (\id{} ⊗ g) ⨾ \mathrm{str}) ⨾^{\ast} \pi^{\ast} & \hint{iii}{=} \\
    & ν_X ⨾ (\id{} ⊗ f) ⨾ \mathrm{str} ⨾ ((ν_{X ⊗ Y} ⨾ (\id{} ⊗ g) ⨾ \mathrm{str}) + 1) ⨾ μ ⨾^{\ast} \pi^{\ast} & \hint{iv}{=} \\
    & ν_X ⨾ (\id{} ⊗ f) ⨾ \mathrm{str} ⨾ ((ν_{X ⊗ Y} ⨾ (\id{} ⊗ g) ⨾ \mathrm{str}) + 1) ⨾ μ ⨾ (\pi + 1). &
  \end{align*}
  We also partially rewrite the right-hand side, by definition of \conditionalComposition{}.
  \begin{align*} 
    (f ⊲ g^{\ast}) ⨾ ℓ =
     ν_X ⨾ (\id{X} ⊗ f) ⨾ ν^{σ}_{Y+1} ⨾ (id{} ⊗ g^{\ast}) ⨾ ℓ.
  \end{align*}

  We note that both sides of the equation can be obtained by precomposing with
  $ν_X ⨾ (\id{} ⊗ f)$. It thus remains to show that
  \[
  \mathrm{str} ⨾ ((ν_{X ⊗ Y} ⨾ (\id{} ⊗ g) ⨾ \mathrm{str}) + 1) ⨾ μ ⨾ (\pi + 1) =
  ν^{σ}_{Y+1} ⨾ (\id{} ⊗  g^{\ast} ) ⨾ ℓ.
  \]
  Again, by the universal property of the coproduct, proving an equality between morphisms of 
  type $X ⊗ (Y + 1) → Y ⊗ Z + 1$ is the same as proving the two equalities obtained by
  precomposition with inclusions.
  
  Let us address the first of these. On the left-hand side, we use \emph{(i)}
  the definition of strength, \emph{(ii)} naturality of the inclusions,
  \emph{(iii)} the definition of the multiplication of the maybe monad,
  \emph{(iv)} interchange, \emph{(v)} and the interaction between copy and
  discard.
  \begin{align*}
    & \iota_1 ⨾ \mathrm{str} ⨾ ((ν_{X ⊗ Y} ⨾ (\id{} ⊗ g) ⨾ \mathrm{str}) + 1) ⨾ μ ⨾ (\pi + 1) & \hint{i}{=} \\
    & \iota_1 ⨾ ((ν_{X ⊗ Y} ⨾ (\id{} ⊗ g) ⨾ \mathrm{str}) + 1) ⨾ μ ⨾ (\pi + 1) & \hint{ii}{=} \\
    & ν_{X ⊗ Y} ⨾ (\id{} ⊗ g) ⨾ \mathrm{str} ⨾ \iota_1 ⨾ μ ⨾ (\pi + 1) & \hint{iii}{=} \\
    & ν_{X ⊗ Y} ⨾  (\id{} ⊗ g) ⨾ \pi ⨾ \mathrm{str} & \hint{iv}{=} \\
    & ν_{X ⊗ Y} ⨾ \pi ⨾ (\id{} ⊗ g) ⨾ \mathrm{str} & \hint{v}{=} \\
    & ν_{Y}^{\sigma} ⨾ (\id{} ⊗ g) ⨾ \mathrm{str}.
  \end{align*}
  On the right hand side, we use \emph{(i)} determinism of the inclusions in a
  \kl{distributive Markov category}; \emph{(ii)} the definition of $g^{\ast}$;
  \emph{(iii)} interchange, and \emph{(iv)} the definition of strength.
  \begin{align*}
    & \iota_1 ⨾ ν^{σ}_{Y+1} ⨾ (\id{} ⊗ g^{\ast}) ⨾ ℓ & \hint{i}{=} \\
    & ν^{σ}_{Y} ⨾ (\iota_1 ⊗ \id{} ⊗ \iota_1) ⨾ g^{\ast} ⨾ ℓ & \hint{ii}{=} \\
    & ν^{σ}_{Y} ⨾ (\iota_1 ⊗ \id{}) ⨾ (\id{} ⊗ g) ⨾ ℓ & \hint{iii}{=} \\
    & ν^{σ}_{Y} ⨾ (\id{} ⊗ g) ⨾ (\iota_1 ⊗ \id{}) ⨾  ℓ & \hint{iv}{=} \\
    & ν^{σ}_{Y} ⨾  (\id{} ⊗ g) ⨾ \mathrm{str}.
  \end{align*}

  Let us address the second of these equations. On the left-hand side, we use
  that \emph{(i)} the definition of strength, \emph{(ii)} naturality of
  inclusions, and \emph{(iii)} the definition of multiplication of the maybe
  monad.
  \begin{align*}
    & \iota_⊥ ⨾ \mathrm{str} ⨾ ((ν_{X ⊗ Y} ⨾ (\id{} ⊗ g) ⨾ \mathrm{str}) + 1) ⨾ μ ⨾ (\pi + 1) & \hint{i}{=} \\
    & \iota_⊥ ⨾ ((ν_{X ⊗ Y} ⨾ (\id{} ⊗ g) ⨾ \mathrm{str}) + 1) ⨾ μ ⨾ (\pi + 1) & \hint{ii}{=} \\
    & \iota_⊥ ⨾ μ ⨾ (\pi + 1) & \hint{iii}{=} \\
    & \iota_⊥.
  \end{align*}
  On the right-hand side, we use \emph{(i)} naturality of the inclusions,
  \emph{(ii)} interchange, and \emph{(iii)} the definition of the laxator.
  \begin{align*}
    & \iota_⊥ ⨾ ν^{σ}_{Y+1} ⨾ (\id{} ⊗ g^{\ast}) ⨾ ℓ  & \hint{i}{=} \\
    & ν^{σ}_{Y} ⨾ (\iota_⊥ ⊗ \id{} ⊗ \iota_⊥) ⨾ (\id{} ⊗ g^{\ast}) ⨾ ℓ  & \hint{ii}{=} \\
    & ν^{σ}_{Y} ⨾ (\iota_⊥ ⊗ \id{}) ⨾ (\id{} ⊗ g^{\ast}) ⨾ (\id{} ⊗ \iota_⊥) ⨾ ℓ  & \hint{iii}{=} \\
    & \iota_⊥.
  \end{align*}

  These two equations conclude the proof.  
\end{proof}

\begin{thm}[Partial Markov categories from Maybe monads]%
  \label{thm:partialMarkovMaybeMonad}
  The Kleisli category of the maybe monad, $(• + 1)$, over a
  \distributiveMarkovCategory{} is a \partialMarkovCategory{}.
\end{thm}
\begin{proof}
  From the deterministic inclusions of a \distributiveMarkovCategory{}, we build
  a \deterministic{} laxator for the maybe monad: $ℓ_{X,Y} ፡ (X + 1) ⊗ (Y + 1) →
  X ⊗ Y + 1.$ This laxator has a section, $𝔰_{X,Y} ⨾ ℓ_{X,Y} = \id{}$.
  
  Given any morphism in the Kleisli category of the maybe monad, $f ፡ X → T(Y ⊗
  Z)$, we can construct a \conditional{} for $(f ⨾ 𝔰_{Y,Z})$: there exists some
  $c ፡ X ⊗ T(Y) → T(Z)$ such that
  $$f ⨾ 𝔰_{Y,Z} = (f ⨾ 𝔰_{Y,Z} ⨾ π₁) ⊲ c.$$%
  Postcomposing with the laxator, we obtain the following equation; by \Cref{lemma:conditionalReplace}, we know
  that the conditional can be replaced by the Kleisli extension of some $d ፡ X ⊗
  Y → T(Z)$
  \begin{align*}
    f & = ((f ⨾ 𝔰_{Y,Z} ⨾ π₁) ⊲ c) ⨾ ℓ_{Y,Z}  \\
    & = ((f ⨾ 𝔰_{Y,Z} ⨾ π₁) ⊲ d^{\ast}) ⨾ ℓ_{Y,Z}.
  \end{align*}
  Finally, by \Cref{lemma:conditionalKleisli}, this is the same as saying that
  $f = (f ⨾ \overline{π}^{\ast}) ⊲^{\ast} d$; in other words, we have
  constructed a \conditional{} in the Kleisli category.
\end{proof}

\subsection{Examples of partial Markov categories}
\label{sec:examples-partial-markov}

\Cref{thm:partialMarkovMaybeMonad} is a source of \kl{partial Markov categories}: from finitely supported distributions, we recover finitely supported subdistributions;
from distribution on standard Borel spaces, we recover subdistributions on standard Borel spaces; from sets and functions, we recover partial functions; from total relations, we recover the category of \emph{may-must} relations~\cite{bonchiSokolovaVignudelli22}.
To illustrate this process, let us detail the example of standard Borel spaces.

\begin{exa}[Standard Borel spaces]
The analogue of \(\subStoch\) for Borel spaces is \(\subBorelStoch\). A morphism
\m{f ፡ X → Y} in \(\subBorelStoch\) represents a stochastic channel that has
some probability of failure, i.e. \(f(x)\) is a \subprobabilityMeasure{}.
\end{exa}

\begin{defi}[Subprobability measure]%
  \defining{linkSubprobabilityMeasure}%
  A \emph{subprobability measure} \(p\) on a measurable space \((X,σ_{X})\) is a measure such that \(p(X) ≤ 1\).
\end{defi}

\begin{defi}[Category of substochastic kernels between standard Borel spaces]%
\label{def:subborelstoch}%
\defining{linksubborelstoch}%
  The objects of \(\subBorelStoch\) are standard Borel spaces, \((X,Σ_{X})\),
  where \(X\) is a set and \(Σ_{X}\) is a \(\sigma\)-algebra on \(X\).
  A morphism \(f ፡ (X,σ_{X}) → (Y,σ_{Y})\) in \(\subBorelStoch\) is a function
  \m{f ፡ Σ_{Y} × X → {[0,1]}} such that
  the conditional \m{f(T \given -) ፡ X → {[0,1]}} is a measurable function for each measurable set \(T ∈ σ_{Y}\), and the function \m{f(- | x) ፡ Σ_{Y} → {[0,1]}} forms
  a \subprobabilityMeasure{}
  for each element \m{x ∈ X}. Identities and composition behave as in \m{\BorelStoch}.
\end{defi}

Many \kl{distributive Markov categories}, like \(\Stoch\) and \(\BorelStoch\), arise as Kleisli categories of suitable monads; some others, like \(\FinStoch\), do not.

\begin{rem}[Panangaden's monad~\cite{panangaden1999}]%
\label{rem:subBorelStoch-kleisli-cat}%
\defining{linksubgiry}%
  The category \(\subBorelStoch\) arises as the Kleisli category of
  Panangaden's monad, \(\subGiry\). Its underlying
  functor is the composition of the Giry functor and the Maybe functor:
  \(\subGiry = \Giry(\maybe)\). There is a candidate distributive law
  \(\distributivelaw ፡ \maybe[\Giry(-)] → \Giry(\maybe)\) defined by
  \(\distributivelaw_{X}(\sigma) = \extenddomain{σ}\) and
  \(\distributivelaw_{X}(\bot) = \dirac{\bot}\), where \(\extenddomain{σ}\) is
  the extension of \m{σ} to \(\maybe[X]\) by \(\extenddomain{σ}(\{⊥\}) = 0\).
  By a general result~\cite[Theorem~4.2.20]{manes2007monad}, this transformation is a distributive law of the Giry monad over the Maybe monad.
  \Cref{th:distributive-law-giry-maybe} shows that the composite monad obtained via this distributive law is indeed the Panangaden monad, \(\subGiry\).
\end{rem}

\begin{prop}%
\label{th:distributive-law-giry-maybe}%
  The Panangaden's monad, \(\subGiry\), arises as the composition of the Giry monad, \(\Giry\), with the Maybe monad, \((\maybe)\): as monads, \(\subGiry = \Giry(\maybe)\).
\end{prop}
\begin{proof}
  Any commutative monad distributes over the Maybe monad by a general result~\cite[Theorem~4.2.20]{manes2007monad}.
  We check that the composition of the Giry monad, \(\Giry\), with the Maybe monad via this distributive law gives the Panangaden monad, \(\subGiry\).

  The functor underlying the Panangaden monad is the composition of the functors underlying the other two monads, \(\Giry(\maybe) = \subGiry\).
  The unit of \(\Giry(\maybe)\) is the composition of the two units: the unit \(\eta_X \colon X \to \Giry(\maybe[X])\) is the inclusion of the unit of \(\Giry\) in \(\Giry(\maybe[X])\), and this inclusion, \(\Giry \to \Giry(\maybe)\), is the unit of \((\maybe)\), \(\eta = \eta_1 \dcomp \Giry(\eta_2)\).
  For the multiplications,
  we can check that \(μ = (\id{} ⊗ \distributivelaw ⊗ \id{}) ⨾ (μ_{1} ⊗
  μ_{2})\). In fact, let \m{p ∈ \Giry(\maybe[\Giry(\maybe[X])])}. Then,
  \begin{align*}
    & ((\id{} ⊗ \distributivelaw_{X} ⊗ \id{}) ⨾ (\monadmultiplication_{1} ⊗ \monadmultiplication_{2}))(p) =
    ∫_{τ ∈ \Giry(\maybe[X])} τ(-) ⋅ p(dτ),
  \end{align*}
  which corresponds to the definition of the multiplication of \(\subGiry\).
\end{proof}

\begin{prop}[Subprobability kernels form a partial Markov category]
  By \Cref{thm:partialMarkovMaybeMonad}, $\subBorelStoch$ is a \kl{partial Markov
  category}.
\end{prop}

\begin{exa}[Relations]
  The category of sets and relations is a partial Markov category, but not as an instance of \Cref{thm:partialMarkovMaybeMonad}.
  In fact, all cartesian bicategories of relations~\cite{CARBONI198711} are partial Markov categories~\cite[Section~2.3]{2025order}.
  For a relation \(f \colon X \to Y \times Z\), its conditional on \(Y\) is the relation \(c \colon X \times Y \to Z\) defined by \(((x,y),z) \in c\) if and only if \((x,(y,z)) \in f\).
\end{exa}

\begin{exa}[May-must relations]%
  \label{ex:rel-may-must}
  The application of \Cref{thm:partialMarkovMaybeMonad} to the category \(\TRel\) of total relations obtains a partial Markov category where morphisms represent nondeterministic computations that return a ``failure'' element if the computation may fail~\cite{bonchi2022traces}.
  Morphisms are functions \(f \colon X \to \nePowerset(Y+1)\); we interpret the set \(f(x)\) as giving the possible outputs; one of the possible outputs can be the only element \(\bot\) of \(1\), which we interpret as failure.
  The difference with the category \(\Rel\) of relations is that, in \(\Rel\), failure is forgotten as long as there is some value that can be returned: there is no special element signalling failure; failure only happens when there are no other possible outputs, i.e.\ when an element is mapped to the empty set.
\end{exa}

The next example describes another \kl{partial Markov category} that does not arise as a Kleisli category of a Maybe monad on its subcategory of total morphisms.

\begin{exa}[Affine maps with partial Gaussian noise]%
  \label{exa:gaussians}
  Consider the subcategory of \(\subBorelStoch\) generated by affine maps with Gaussian noise, \(\Gauss\), described in \Cref{def:gauss}; this is a copy-discard subcategory but is not closed under coproducts.
  Consider the scalars in \(\subBorelStoch\); these are the morphisms of type \(1 \to 1\) and they are in bijection with real numbers in the closed interval \([0,1]\).
  Finally, consider the monoidal subcategory of \(\subBorelStoch\) generated by \(\Gauss\) and all the scalars, which we denote \(\parGauss\).
  Morphisms in this subcategory are of the form \(f \tensor p\) for a morphism \(f\) of \(\Gauss\) and a real \(p \in [0,1]\) because \(\subBorelStoch\) is symmetric monoidal; if \(p=0\), then \(f \tensor p\) is the failure map for any choice of \(f\).
  For these reasons, \(\parGauss\) is equivalent to a prop similar to the one describing \(\Gauss\).
  Its morphisms \(n \to m\) are either a pair \((f,p)\) of a morphism \(f \colon n \to m\) in \(\Gauss\) and a \emph{validity} \(p \in (0,1]\), or \(\bot\) representing the subdistribution that is always \(0\).
  Compositions and monoidal products of two nonzero morphisms are inherited from \(\Gauss\) and multiply the validities: \((f,p) \dcomp (g,q) = (f \dcomp g, p \cdot q)\) and \((f,p) \tensor (g,q) = (f \tensor g, p \cdot q)\).
  Composition and monoidal product with \(\bot\) on either side equal \(\bot\).
  The copy-discard structure and conditionals are also inherited from \(\Gauss\): the copy and discard morphisms are \((\cp, 1)\) and \((\discard, 1)\), where \(\cp\) and \(\discard\) are the copy and discard morphisms of \(\Gauss\); for a morphism \((f,p) \colon n \to m+l\), its marginal is \((f \dcomp \pi_1, p) \colon n \to m\) and its conditional is \((c,1) \colon m+n \to l\), for a conditional \(c\) of \(f\) in \(\Gauss\).
  Then, \(\parGauss\) is a \kl{partial Markov category}.
  
  The construction of \(\parGauss\) could be generalised to any other \kl{Markov category} \(\cat{C}\): one would consider the product of \(\cat{C}\) and the interval \([0,1]\) seen as a monoidal category, and quotient by \((f,0) \sim (g,0)\).
  Intuitively, one only adds a ``probability of termination'' that does not depend on the inputs to the morphisms in \(\cat{C}\); this is an interesting construction for affine maps with Gaussian noise because the Gaussian noise is also independent of the inputs, but it may be less relevant for a generic \kl{Markov category}.
\end{exa}
\section{Discrete Partial Markov Categories}%
\label{sec:discretePartialMarkovCategories}

We shall now introduce \emph{\discretePartialMarkovCategories{}} (\Cref{def:discrete-partial-markov-cat}), a blend of \partialMarkovCategories{} and \kl{discrete cartesian restriction categories}.

\kl{Discrete cartesian restriction categories}~\cite{cockett2012range2} are a refinement of \kl{cartesian restriction categories} that allows \emph{constraints}: a map may fail if some conditions are not satisfied. Our observation is that a similar refinement can be applied to \partialMarkovCategories{} to obtain \discretePartialMarkovCategories{}. These provide a setting in which it is possible to \emph{(i)} constrain the behaviour of a morphism, via an abstract analogue of Bayesian updates; and \emph{(ii)} reason with these constrained morphisms.

Encoding constraints requires \emph{compare morphisms} that interact nicely with copy and discard (\Cref{diagram-partial-frobenius}, see also \Cref{rem:reading-partial-frobenius}).
A compare morphism declares that some constraint—usually coming from an observation—must be satisfied by a probabilistic process.

\begin{defi}[Copy-discard-compare category]
\defining{linkcompare}%
\defining{linkcomparators}%
\AP 
A \intro{copy-discard-compare category} is a \kl{copy-discard category} endowed where every object \(X\) is endowed with a ``compare'' morphism, \m{\compare_X ፡ X ⊗ X → X}, that is uniform (meaning $μ_{X ⊗ Y} = (\id{X} ⊗ σ_{X,Y} ⊗ \id{Y}) ⨾ (μ_X ⊗ μ_Y)$ and $μ_I = \id{I}$), commutative, associative, satisfies the Frobenius axiom, and the special axiom.

In other words, objects are endowed with a uniform partial Frobenius monoid structure \cite{di_liberti_nester_2021,stein_thesis_2021} satisfying the
axioms in \Cref{diagram-partial-frobenius}.
\begin{figure}[h!]

\tikzset{every picture/.style={line width=0.75pt}} %

\begin{tikzpicture}[x=0.75pt,y=0.75pt,yscale=-1,xscale=1]
\draw    (70,85) -- (70,100) ;
\draw    (85,115) -- (85,125) ;
\draw    (85,115) .. controls (75.17,110.57) and (70,110.53) .. (70,100) ;
\draw  [fill={rgb, 255:red, 0; green, 0; blue, 0 }  ,fill opacity=1 ] (82,115) .. controls (82,116.66) and (83.34,118) .. (85,118) .. controls (86.66,118) and (88,116.66) .. (88,115) .. controls (88,113.34) and (86.66,112) .. (85,112) .. controls (83.34,112) and (82,113.34) .. (82,115) -- cycle ;
\draw    (85,115) .. controls (94.88,110.53) and (100,110.28) .. (100,100) ;
\draw    (100,100) .. controls (90.17,95.57) and (85,95.53) .. (85,85) ;
\draw  [fill={rgb, 255:red, 0; green, 0; blue, 0 }  ,fill opacity=1 ] (97,100) .. controls (97,101.66) and (98.34,103) .. (100,103) .. controls (101.66,103) and (103,101.66) .. (103,100) .. controls (103,98.34) and (101.66,97) .. (100,97) .. controls (98.34,97) and (97,98.34) .. (97,100) -- cycle ;
\draw    (100,100) .. controls (109.88,95.53) and (115,95.28) .. (115,85) ;
\draw    (175,115) -- (175,125) ;
\draw    (175,115) .. controls (165.17,110.57) and (160,110.53) .. (160,100) ;
\draw  [fill={rgb, 255:red, 0; green, 0; blue, 0 }  ,fill opacity=1 ] (172,115) .. controls (172,116.66) and (173.34,118) .. (175,118) .. controls (176.66,118) and (178,116.66) .. (178,115) .. controls (178,113.34) and (176.66,112) .. (175,112) .. controls (173.34,112) and (172,113.34) .. (172,115) -- cycle ;
\draw    (175,115) .. controls (184.88,110.53) and (190,110.28) .. (190,100) ;
\draw    (190,85) -- (190,100) ;
\draw    (160,100) .. controls (150.17,95.57) and (145,95.53) .. (145,85) ;
\draw  [fill={rgb, 255:red, 0; green, 0; blue, 0 }  ,fill opacity=1 ] (157,100) .. controls (157,101.66) and (158.34,103) .. (160,103) .. controls (161.66,103) and (163,101.66) .. (163,100) .. controls (163,98.34) and (161.66,97) .. (160,97) .. controls (158.34,97) and (157,98.34) .. (157,100) -- cycle ;
\draw    (160,100) .. controls (169.88,95.53) and (175,95.28) .. (175,85) ;
\draw    (235,115) -- (235,125) ;
\draw    (235,115) .. controls (225.17,110.57) and (220.29,110.69) .. (220,105) ;
\draw  [fill={rgb, 255:red, 0; green, 0; blue, 0 }  ,fill opacity=1 ] (232,115) .. controls (232,116.66) and (233.34,118) .. (235,118) .. controls (236.66,118) and (238,116.66) .. (238,115) .. controls (238,113.34) and (236.66,112) .. (235,112) .. controls (233.34,112) and (232,113.34) .. (232,115) -- cycle ;
\draw    (235,115) .. controls (244.88,110.53) and (249.86,110.4) .. (250,105) ;
\draw    (220,105) .. controls (220.5,94.57) and (249.83,95.23) .. (250,85) ;
\draw    (250,105) .. controls (250.5,94.57) and (219.83,95.23) .. (220,85) ;
\draw  [draw opacity=0] (250.01,115) -- (280.01,115) -- (280.01,90) -- (250.01,90) -- cycle ;
\draw    (295.01,110) -- (295.01,125) ;
\draw    (295.01,110) .. controls (285.17,105.57) and (280.01,105.53) .. (280.01,95) ;
\draw  [fill={rgb, 255:red, 0; green, 0; blue, 0 }  ,fill opacity=1 ] (292.01,110) .. controls (292.01,111.66) and (293.35,113) .. (295.01,113) .. controls (296.66,113) and (298.01,111.66) .. (298.01,110) .. controls (298.01,108.34) and (296.66,107) .. (295.01,107) .. controls (293.35,107) and (292.01,108.34) .. (292.01,110) -- cycle ;
\draw    (295.01,110) .. controls (304.88,105.53) and (310.01,105.28) .. (310.01,95) ;
\draw    (280.01,85) -- (280.01,95) ;
\draw    (310.01,85) -- (310.01,95) ;
\draw    (580,95) -- (580,85) ;
\draw    (580,95) .. controls (570.17,99.43) and (570,99.47) .. (570,110) ;
\draw  [fill={rgb, 255:red, 0; green, 0; blue, 0 }  ,fill opacity=1 ] (577,95) .. controls (577,93.34) and (578.34,92) .. (580,92) .. controls (581.66,92) and (583,93.34) .. (583,95) .. controls (583,96.66) and (581.66,98) .. (580,98) .. controls (578.34,98) and (577,96.66) .. (577,95) -- cycle ;
\draw    (580,95) .. controls (589.88,99.47) and (590,94.72) .. (590,105) ;
\draw    (600,115) .. controls (590.17,110.57) and (590,115.53) .. (590,105) ;
\draw  [fill={rgb, 255:red, 0; green, 0; blue, 0 }  ,fill opacity=1 ] (597,115) .. controls (597,116.66) and (598.34,118) .. (600,118) .. controls (601.66,118) and (603,116.66) .. (603,115) .. controls (603,113.34) and (601.66,112) .. (600,112) .. controls (598.34,112) and (597,113.34) .. (597,115) -- cycle ;
\draw    (600,115) .. controls (609.88,110.53) and (610,110.28) .. (610,100) ;
\draw    (610,100) -- (610,85) ;
\draw    (570,125) -- (570,110) ;
\draw    (600,125) -- (600,115) ;
\draw    (470,95) -- (470,85) ;
\draw    (470,95) .. controls (479.83,99.43) and (480,99.47) .. (480,110) ;
\draw  [fill={rgb, 255:red, 0; green, 0; blue, 0 }  ,fill opacity=1 ] (473,95) .. controls (473,93.34) and (471.66,92) .. (470,92) .. controls (468.34,92) and (467,93.34) .. (467,95) .. controls (467,96.66) and (468.34,98) .. (470,98) .. controls (471.66,98) and (473,96.66) .. (473,95) -- cycle ;
\draw    (470,95) .. controls (460.13,99.47) and (460,94.72) .. (460,105) ;
\draw    (450,115) .. controls (459.83,110.57) and (460,115.53) .. (460,105) ;
\draw  [fill={rgb, 255:red, 0; green, 0; blue, 0 }  ,fill opacity=1 ] (453,115) .. controls (453,116.66) and (451.66,118) .. (450,118) .. controls (448.34,118) and (447,116.66) .. (447,115) .. controls (447,113.34) and (448.34,112) .. (450,112) .. controls (451.66,112) and (453,113.34) .. (453,115) -- cycle ;
\draw    (450,115) .. controls (440.13,110.53) and (440,110.28) .. (440,100) ;
\draw    (440,100) -- (440,85) ;
\draw    (480,125) -- (480,110) ;
\draw    (450,125) -- (450,115) ;
\draw    (525,100) .. controls (515.17,95.57) and (510,95.53) .. (510,85) ;
\draw  [fill={rgb, 255:red, 0; green, 0; blue, 0 }  ,fill opacity=1 ] (522,100) .. controls (522,101.66) and (523.34,103) .. (525,103) .. controls (526.66,103) and (528,101.66) .. (528,100) .. controls (528,98.34) and (526.66,97) .. (525,97) .. controls (523.34,97) and (522,98.34) .. (522,100) -- cycle ;
\draw    (525,100) .. controls (534.88,95.53) and (540,95.28) .. (540,85) ;
\draw    (525,110) .. controls (515.17,114.43) and (510,114.47) .. (510,125) ;
\draw  [fill={rgb, 255:red, 0; green, 0; blue, 0 }  ,fill opacity=1 ] (522,110) .. controls (522,108.34) and (523.34,107) .. (525,107) .. controls (526.66,107) and (528,108.34) .. (528,110) .. controls (528,111.66) and (526.66,113) .. (525,113) .. controls (523.34,113) and (522,111.66) .. (522,110) -- cycle ;
\draw    (525,110) .. controls (534.88,114.47) and (540,114.72) .. (540,125) ;
\draw    (525,110) -- (525,100) ;
\draw  [draw opacity=0] (480,95) -- (510,95) -- (510,115) -- (480,115) -- cycle ;
\draw  [draw opacity=0] (540,95) -- (570,95) -- (570,115) -- (540,115) -- cycle ;
\draw  [draw opacity=0] (115,115) -- (145,115) -- (145,90) -- (115,90) -- cycle ;
\draw    (355,115) .. controls (345.17,110.57) and (340,115.53) .. (340,105) ;
\draw  [fill={rgb, 255:red, 0; green, 0; blue, 0 }  ,fill opacity=1 ] (352,115) .. controls (352,116.66) and (353.34,118) .. (355,118) .. controls (356.66,118) and (358,116.66) .. (358,115) .. controls (358,113.34) and (356.66,112) .. (355,112) .. controls (353.34,112) and (352,113.34) .. (352,115) -- cycle ;
\draw    (355,115) .. controls (364.88,110.53) and (370,115.28) .. (370,105) ;
\draw    (355,95) .. controls (345.17,99.43) and (340,94.47) .. (340,105) ;
\draw  [fill={rgb, 255:red, 0; green, 0; blue, 0 }  ,fill opacity=1 ] (352,95) .. controls (352,93.34) and (353.34,92) .. (355,92) .. controls (356.66,92) and (358,93.34) .. (358,95) .. controls (358,96.66) and (356.66,98) .. (355,98) .. controls (353.34,98) and (352,96.66) .. (352,95) -- cycle ;
\draw    (355,95) .. controls (364.88,99.47) and (370,94.72) .. (370,105) ;
\draw    (355,125) -- (355,115) ;
\draw    (355,85) -- (355,95) ;
\draw  [draw opacity=0] (370,115) -- (400,115) -- (400,90) -- (370,90) -- cycle ;
\draw    (405,85) -- (405,125) ;
\draw  [draw opacity=0] (310,120) -- (340,120) -- (340,90) -- (310,90) -- cycle ;
\draw  [draw opacity=0] (410,120) -- (440,120) -- (440,90) -- (410,90) -- cycle ;

\draw (440,81.6) node [anchor=south] [inner sep=0.75pt]  [font=\tiny]  {$X$};
\draw (470,81.6) node [anchor=south] [inner sep=0.75pt]  [font=\tiny]  {$X$};
\draw (510,81.6) node [anchor=south] [inner sep=0.75pt]  [font=\tiny]  {$X$};
\draw (265.01,102.5) node    {$=$};
\draw (450,128.4) node [anchor=north] [inner sep=0.75pt]  [font=\tiny]  {$X$};
\draw (480,128.4) node [anchor=north] [inner sep=0.75pt]  [font=\tiny]  {$X$};
\draw (495,105) node    {$=$};
\draw (555,105) node    {$=$};
\draw (540,81.6) node [anchor=south] [inner sep=0.75pt]  [font=\tiny]  {$X$};
\draw (580,81.6) node [anchor=south] [inner sep=0.75pt]  [font=\tiny]  {$X$};
\draw (610,81.6) node [anchor=south] [inner sep=0.75pt]  [font=\tiny]  {$X$};
\draw (540,128.4) node [anchor=north] [inner sep=0.75pt]  [font=\tiny]  {$X$};
\draw (510,128.4) node [anchor=north] [inner sep=0.75pt]  [font=\tiny]  {$X$};
\draw (570,128.4) node [anchor=north] [inner sep=0.75pt]  [font=\tiny]  {$X$};
\draw (600,128.4) node [anchor=north] [inner sep=0.75pt]  [font=\tiny]  {$X$};
\draw (295.01,128.4) node [anchor=north] [inner sep=0.75pt]  [font=\tiny]  {$X$};
\draw (235,128.4) node [anchor=north] [inner sep=0.75pt]  [font=\tiny]  {$X$};
\draw (220,81.6) node [anchor=south] [inner sep=0.75pt]  [font=\tiny]  {$X$};
\draw (250,81.6) node [anchor=south] [inner sep=0.75pt]  [font=\tiny]  {$X$};
\draw (280.01,81.6) node [anchor=south] [inner sep=0.75pt]  [font=\tiny]  {$X$};
\draw (310.01,81.6) node [anchor=south] [inner sep=0.75pt]  [font=\tiny]  {$X$};
\draw (130,102.5) node    {$=$};
\draw (70,81.6) node [anchor=south] [inner sep=0.75pt]  [font=\tiny]  {$X$};
\draw (85,81.6) node [anchor=south] [inner sep=0.75pt]  [font=\tiny]  {$X$};
\draw (115,81.6) node [anchor=south] [inner sep=0.75pt]  [font=\tiny]  {$X$};
\draw (145,81.6) node [anchor=south] [inner sep=0.75pt]  [font=\tiny]  {$X$};
\draw (175,81.6) node [anchor=south] [inner sep=0.75pt]  [font=\tiny]  {$X$};
\draw (190,81.6) node [anchor=south] [inner sep=0.75pt]  [font=\tiny]  {$X$};
\draw (175,128.4) node [anchor=north] [inner sep=0.75pt]  [font=\tiny]  {$X$};
\draw (85,128.4) node [anchor=north] [inner sep=0.75pt]  [font=\tiny]  {$X$};
\draw (385,102.5) node    {$=$};
\draw (405,128.4) node [anchor=north] [inner sep=0.75pt]  [font=\tiny]  {$X$};
\draw (355,128.4) node [anchor=north] [inner sep=0.75pt]  [font=\tiny]  {$X$};
\draw (355,81.6) node [anchor=south] [inner sep=0.75pt]  [font=\tiny]  {$X$};
\draw (405,81.6) node [anchor=south] [inner sep=0.75pt]  [font=\tiny]  {$X$};

\end{tikzpicture}
   \caption{Partial Frobenius axioms: associativity, commutativity, special, and Frobenius.}%
  \label{diagram-partial-frobenius}%
\end{figure}
\end{defi}

\begin{defi}[Discrete partial Markov category]
  \label{def:discrete-partial-markov-cat}\defining{linkDiscretePartialMarkov}%
  \AP A \intro{discrete partial Markov category} is a \kl{copy-discard-compare category} with \kl{conditionals}.
\end{defi}

\begin{rem}%
  \label{rem:discrete-cartesian-restriction}
  \intro{Discrete cartesian restriction categories} are \kl{cartesian restriction categories} with \comparators{}~\cite{cockett2012range2,di_liberti_nester_2021}.
  Thanks to \Cref{prop:cartesian-restriction}, \kl{discrete cartesian restriction categories} are characterised as \discretePartialMarkovCategories{} where all morphisms are deterministic.
  The paradigmatic example of \kl{discrete cartesian restriction category} is the category of sets and partial functions, \intro[Par]{}\(\Par\); it can be described as the Kleisli category of the Maybe monad on \(\Set\), \(\Par = \kleisli{\maybe}\).
  The compare morphism in \(\Par\) is the partial function
  \[μ_{X}(x_{1},x_{2}) =
    \begin{cases} 
      x, & x=x_{1}=x_{2}; \\ 
      \text{undefined}, & \text{otherwise}.
    \end{cases}\]%
  One can check that it satisfies the axioms in
  \Cref{diagram-partial-frobenius}: copying a resource and then checking that the two copies coincide should be the identity process; checking equality of two resources and then copying them is the same as copying one resource if it
  coincides with the other one.
\end{rem}

\begin{exa}[\protect\SubStoch{} is a \protect\discretePartialMarkovCategory{}]%
  \label{ex:substoch-discrete-partial-markov}
  The Kleisli category of the finitary subdistribution monad, \(\subStoch\), is a \discretePartialMarkovCategory{}.
  The compare morphism \(μ_{X} ፡ X ⊗ X → X\) is given by
  \[μ_{X}(x \given x_{1},x_{2}) =
    \begin{cases} 
      1, & x=x_{1}=x_{2}; \\ 
      0, & \text{otherwise}.
    \end{cases}\]%
  The \kl{compare}, copy and discard morphisms of \(\subStoch\) are inherited from the category of partial functions, \(\Par\), via a symmetric monoidal functor \(\inclusion \colon \Par \to
  \subStoch\).
  This functor arises from the monoidal distributive law of the finitary distributions monad over the Maybe monad: we have that \(\Par = \kleisli{\maybe}\) and \(\subStoch = \kleisli{\distr(\maybe)}\); then, the functor is defined by post-composition in \(\Set\) with the unit of the finitary distribution monad, \(\inclusion(f) \defn f \dcomp \eta^{\distr}_{\maybe[Y]}\); this functor is identity on objects and, since the unit \(\eta^{\distr}\) is monoidal and monic, it is also monoidal and faithful.
  The compare morphisms in \(\subStoch\) are defined as the images under \(\inclusion\) of the comparre morphisms in \(\Par\); the copy and discard morphisms in the two categories are also related this way.
  Then, the compare, copy and discard morphisms satisfy the axioms in \Cref{diagram-partial-frobenius} in \(\subStoch\) because they do so in \(\Par\).
\end{exa}

\begin{rem}%
  \label{rem:reading-partial-frobenius}
  Thanks to the partial Frobenius axioms (\Cref{diagram-partial-frobenius}),
  string diagrams in \(\subStoch\) keep the same intuitive reading as
  in $\Stoch$ (\Cref{rem:readingStrings}): the value of a morphism is obtained by multiplying
  the values of all its components and summing on the wires that are not inputs
  nor outputs. For example, the value of the morphism
  \begin{center}
    \smallskip

\tikzset{every picture/.style={line width=0.75pt}} %

\begin{tikzpicture}[x=0.75pt,y=0.75pt,yscale=-1,xscale=1]
\draw   (245,35) -- (265,35) -- (265,50) -- (245,50) -- cycle ;
\draw    (255,35) -- (255,25) ;
\draw    (255,60) -- (255,50) ;
\draw   (350,20) -- (370,20) -- (370,35) -- (350,35) -- cycle ;
\draw    (360,20) -- (360,15) ;
\draw    (360,40) -- (360,35) ;
\draw    (360,40) .. controls (350.17,44.43) and (345,39.47) .. (345,50) ;
\draw  [fill={rgb, 255:red, 0; green, 0; blue, 0 }  ,fill opacity=1 ] (357,40) .. controls (357,38.34) and (358.34,37) .. (360,37) .. controls (361.66,37) and (363,38.34) .. (363,40) .. controls (363,41.66) and (361.66,43) .. (360,43) .. controls (358.34,43) and (357,41.66) .. (357,40) -- cycle ;
\draw    (360,40) .. controls (369.88,44.47) and (375,39.72) .. (375,50) ;
\draw    (330,70) -- (330,60) ;
\draw    (330,60) .. controls (320.17,55.57) and (315,60.53) .. (315,50) ;
\draw  [fill={rgb, 255:red, 0; green, 0; blue, 0 }  ,fill opacity=1 ] (327,60) .. controls (327,61.66) and (328.34,63) .. (330,63) .. controls (331.66,63) and (333,61.66) .. (333,60) .. controls (333,58.34) and (331.66,57) .. (330,57) .. controls (328.34,57) and (327,58.34) .. (327,60) -- cycle ;
\draw    (330,60) .. controls (339.88,55.53) and (345,60.28) .. (345,50) ;
\draw  [fill={rgb, 255:red, 0; green, 0; blue, 0 }  ,fill opacity=1 ] (327,70) .. controls (327,68.34) and (328.34,67) .. (330,67) .. controls (331.66,67) and (333,68.34) .. (333,70) .. controls (333,71.66) and (331.66,73) .. (330,73) .. controls (328.34,73) and (327,71.66) .. (327,70) -- cycle ;
\draw   (305,35) -- (325,35) -- (325,50) -- (305,50) -- cycle ;
\draw   (365,50) -- (385,50) -- (385,65) -- (365,65) -- cycle ;
\draw    (375,75) -- (375,65) ;
\draw  [draw opacity=0] (265,30) -- (305,30) -- (305,55) -- (265,55) -- cycle ;

\draw (255,21.6) node [anchor=south] [inner sep=0.75pt]  [font=\tiny]  {$X$};
\draw (285,42.5) node  [font=\footnotesize]  {$=$};
\draw (255,42.5) node  [font=\footnotesize]  {$f$};
\draw (315,42.5) node  [font=\footnotesize]  {$k$};
\draw (360,27.5) node  [font=\footnotesize]  {$g$};
\draw (375,57.5) node  [font=\footnotesize]  {$h$};
\draw (360,11.6) node [anchor=south] [inner sep=0.75pt]  [font=\tiny]  {$X$};
\draw (255,63.4) node [anchor=north] [inner sep=0.75pt]  [font=\tiny]  {$Z$};
\draw (375,78.4) node [anchor=north] [inner sep=0.75pt]  [font=\tiny]  {$Z$};

\end{tikzpicture}
     \smallskip
  \end{center}
  is given by the formula 
  \[f(z \given x) = \sum_{y ∈ Y} 
      g(y \given x) · h(z \given y) · k(y).\]
\end{rem}

\begin{rem}
  \label{rem:stein}
  Stein first recognized this same partial Frobenius structure in the category of projectivized substochastic kernels: the localization of $\SubStoch$ that equates every non-zero scalar to the identity \cite[Proposition 20.1]{stein_thesis_2021}.
\end{rem}

\begin{exa}[Naïve compare morphisms in measurable probability]%
\label{ex:borel-subdistributions-no-comparator}
  The category of standard Borel spaces and substochastic kernels, \(\subBorelStoch\), can be endowed with naive compare morphisms,
  \(\compare_{X} ፡ (X, Σ_{X}) ⊗ (X, Σ_{X}) → (X, Σ_{X})\), that satisfy the
  special Frobenius axioms:
  \[μ_{X}(A, x, y) =
  \begin{cases}
    1, & \mbox{ if } x=y \in A; \\
    0, & \text{otherwise}. %
  \end{cases}
  \]%
  This definition gives a measurable function \(μ_{X}(A \given -,-) ፡ X × X →
  {[0,1]}\) if and only if the diagonal \(Δ_{X} = \{(x,x) \st x ∈ X\}\) belongs
  to the product \(σ\)-algebra, \(Σ_{X} × Σ_{X}\), which holds for standard Borel spaces.
  However, this naive compare morphism does not behave as we would like it to: for any point \(x_0 \in X\) and any continuous measure on \(X\), the singleton set \(\{x_{0}\}\) has measure \(0\), so comparing with \(x_{0}\) yields the subdistribution with measure \(0\), which cannot be renormalised.
  We address this problem in \Cref{sec:exactObservations}.
\end{exa}

\subsection{Bayes' Theorem}

\BayesTheorem{} prescribes how to update one's belief in light of new evidence. Briefly, one observes evidence \(y ∈ Y\) from a prior distribution \(σ\) on \(X\) through a channel \(c ፡ X → Y\). The updated distribution is given by evaluating the Bayesian inversion of the channel \(c\) on the new observation \(y\). This formulation suggests that Bayes' theorem is intrinsically about exact observations.
\Cref{th:bayes-exact-observations} in the next section supports this claim.

\begin{thm}[Synthetic Bayes' Theorem]%
  \label{th:bayes}%
  \defining{linkBayesTheorem}%
  In a \discretePartialMarkovCategory{}, observing a deterministic \m{y ፡ I → Y}
  from a prior distribution \m{p ፡ I → X} through a channel \m{f ፡ X → Y} (LHS of \Cref{fig:bayesTheorem}) is the
  same, up to scalar, as evaluating the \BayesianInversion{} on the observation,
  \m{\bayesinv{f}{p}(y)} (RHS of \Cref{fig:bayesTheorem}).
  \begin{figure}[h!]

\tikzset{every picture/.style={line width=0.75pt}} %

\begin{tikzpicture}[x=0.75pt,y=0.75pt,yscale=-1,xscale=1]
\draw   (130,15.5) -- (150,15.5) -- (150,30.5) -- (130,30.5) -- cycle ;
\draw    (110,85) -- (110,75) ;
\draw    (140,40) -- (140,30) ;
\draw    (140,40) .. controls (130.17,44.43) and (125,39.47) .. (125,50) ;
\draw  [fill={rgb, 255:red, 0; green, 0; blue, 0 }  ,fill opacity=1 ] (137,40) .. controls (137,38.34) and (138.34,37) .. (140,37) .. controls (141.66,37) and (143,38.34) .. (143,40) .. controls (143,41.66) and (141.66,43) .. (140,43) .. controls (138.34,43) and (137,41.66) .. (137,40) -- cycle ;
\draw    (140,40) .. controls (149.88,44.47) and (155,39.72) .. (155,50) ;
\draw   (115,50) -- (135,50) -- (135,65) -- (115,65) -- cycle ;
\draw    (155,70) -- (155,50) ;
\draw   (220,40) -- (240,40) -- (240,55) -- (220,55) -- cycle ;
\draw    (230,65) -- (230,55) ;
\draw   (215,65) -- (245,65) -- (245,80) -- (215,80) -- cycle ;
\draw    (230,90) -- (230,80) ;
\draw    (155,90) -- (155,70) ;
\draw  [draw opacity=0] (165,40) -- (205,40) -- (205,65) -- (165,65) -- cycle ;
\draw    (110,75) .. controls (100.17,70.57) and (95,75.53) .. (95,65) ;
\draw  [fill={rgb, 255:red, 0; green, 0; blue, 0 }  ,fill opacity=1 ] (107,75) .. controls (107,76.66) and (108.34,78) .. (110,78) .. controls (111.66,78) and (113,76.66) .. (113,75) .. controls (113,73.34) and (111.66,72) .. (110,72) .. controls (108.34,72) and (107,73.34) .. (107,75) -- cycle ;
\draw    (110,75) .. controls (119.88,70.53) and (125,75.28) .. (125,65) ;
\draw  [fill={rgb, 255:red, 0; green, 0; blue, 0 }  ,fill opacity=1 ] (107,85) .. controls (107,83.34) and (108.34,82) .. (110,82) .. controls (111.66,82) and (113,83.34) .. (113,85) .. controls (113,86.66) and (111.66,88) .. (110,88) .. controls (108.34,88) and (107,86.66) .. (107,85) -- cycle ;
\draw   (85,50) -- (105,50) -- (105,65) -- (85,65) -- cycle ;
\draw   (280,20) -- (300,20) -- (300,35) -- (280,35) -- cycle ;
\draw    (275,74.5) -- (275,64.5) ;
\draw    (290,40) -- (290,35) ;
\draw   (280,40) -- (300,40) -- (300,55) -- (280,55) -- cycle ;
\draw    (275,64.5) .. controls (265.17,60.07) and (260,65.03) .. (260,54.5) ;
\draw  [fill={rgb, 255:red, 0; green, 0; blue, 0 }  ,fill opacity=1 ] (272,64.5) .. controls (272,66.16) and (273.34,67.5) .. (275,67.5) .. controls (276.66,67.5) and (278,66.16) .. (278,64.5) .. controls (278,62.84) and (276.66,61.5) .. (275,61.5) .. controls (273.34,61.5) and (272,62.84) .. (272,64.5) -- cycle ;
\draw    (275,64.5) .. controls (284.88,60.03) and (290,64.78) .. (290,54.5) ;
\draw  [fill={rgb, 255:red, 0; green, 0; blue, 0 }  ,fill opacity=1 ] (272,74.5) .. controls (272,72.84) and (273.34,71.5) .. (275,71.5) .. controls (276.66,71.5) and (278,72.84) .. (278,74.5) .. controls (278,76.16) and (276.66,77.5) .. (275,77.5) .. controls (273.34,77.5) and (272,76.16) .. (272,74.5) -- cycle ;
\draw   (250,40) -- (270,40) -- (270,55) -- (250,55) -- cycle ;

\draw (140,23) node  [font=\footnotesize]  {$p$};
\draw (125,57.5) node  [font=\footnotesize]  {$f$};
\draw (230,47.5) node  [font=\footnotesize]  {$y$};
\draw (230,72.5) node  [font=\footnotesize]  {$f_{\dagger }( p)$};
\draw (185,52.5) node  [font=\footnotesize]  {$=$};
\draw (95,57.5) node  [font=\footnotesize]  {$y$};
\draw (290,27.5) node  [font=\footnotesize]  {$p$};
\draw (290,47.5) node  [font=\footnotesize]  {$f$};
\draw (260,47) node  [font=\footnotesize]  {$y$};

\end{tikzpicture}
     \caption{Synthetic \protect\BayesTheorem{}.}%
    \label{fig:bayesTheorem}
  \end{figure}
\end{thm}
\begin{proof}
  The equalities follow from: (\emph{i}) the definition of Bayesian inversion
  (\Cref{def:bayes-inversion}), (\emph{ii}) the partial
  Frobenius axioms (\Cref{diagram-partial-frobenius}), and (\emph{iii}) the fact
  that $y$ is \deterministic{}.
  \begin{figure}[h!]

\tikzset{every picture/.style={line width=0.75pt}} %

\begin{tikzpicture}[x=0.75pt,y=0.75pt,yscale=-1,xscale=1]
\draw   (130,15.5) -- (150,15.5) -- (150,30.5) -- (130,30.5) -- cycle ;
\draw    (110,85) -- (110,75) ;
\draw    (140,40) -- (140,30) ;
\draw    (140,40) .. controls (130.17,44.43) and (125,39.47) .. (125,50) ;
\draw  [fill={rgb, 255:red, 0; green, 0; blue, 0 }  ,fill opacity=1 ] (137,40) .. controls (137,38.34) and (138.34,37) .. (140,37) .. controls (141.66,37) and (143,38.34) .. (143,40) .. controls (143,41.66) and (141.66,43) .. (140,43) .. controls (138.34,43) and (137,41.66) .. (137,40) -- cycle ;
\draw    (140,40) .. controls (149.88,44.47) and (155,39.72) .. (155,50) ;
\draw   (115,50) -- (135,50) -- (135,65) -- (115,65) -- cycle ;
\draw    (155,70) -- (155,50) ;
\draw   (470,40) -- (490,40) -- (490,55) -- (470,55) -- cycle ;
\draw    (480,65) -- (480,55) ;
\draw   (465,65) -- (495,65) -- (495,80) -- (465,80) -- cycle ;
\draw    (480,90) -- (480,80) ;
\draw    (155,90) -- (155,70) ;
\draw  [draw opacity=0] (160,40) -- (200,40) -- (200,65) -- (160,65) -- cycle ;
\draw    (110,75) .. controls (100.17,70.57) and (95,75.53) .. (95,65) ;
\draw  [fill={rgb, 255:red, 0; green, 0; blue, 0 }  ,fill opacity=1 ] (107,75) .. controls (107,76.66) and (108.34,78) .. (110,78) .. controls (111.66,78) and (113,76.66) .. (113,75) .. controls (113,73.34) and (111.66,72) .. (110,72) .. controls (108.34,72) and (107,73.34) .. (107,75) -- cycle ;
\draw    (110,75) .. controls (119.88,70.53) and (125,75.28) .. (125,65) ;
\draw  [fill={rgb, 255:red, 0; green, 0; blue, 0 }  ,fill opacity=1 ] (107,85) .. controls (107,83.34) and (108.34,82) .. (110,82) .. controls (111.66,82) and (113,83.34) .. (113,85) .. controls (113,86.66) and (111.66,88) .. (110,88) .. controls (108.34,88) and (107,86.66) .. (107,85) -- cycle ;
\draw   (85,50) -- (105,50) -- (105,65) -- (85,65) -- cycle ;
\draw   (530,20) -- (550,20) -- (550,35) -- (530,35) -- cycle ;
\draw    (525,74.5) -- (525,64.5) ;
\draw    (540,40) -- (540,35) ;
\draw   (530,40) -- (550,40) -- (550,55) -- (530,55) -- cycle ;
\draw    (525,64.5) .. controls (515.17,60.07) and (510,65.03) .. (510,54.5) ;
\draw  [fill={rgb, 255:red, 0; green, 0; blue, 0 }  ,fill opacity=1 ] (522,64.5) .. controls (522,66.16) and (523.34,67.5) .. (525,67.5) .. controls (526.66,67.5) and (528,66.16) .. (528,64.5) .. controls (528,62.84) and (526.66,61.5) .. (525,61.5) .. controls (523.34,61.5) and (522,62.84) .. (522,64.5) -- cycle ;
\draw    (525,64.5) .. controls (534.88,60.03) and (540,64.78) .. (540,54.5) ;
\draw  [fill={rgb, 255:red, 0; green, 0; blue, 0 }  ,fill opacity=1 ] (522,74.5) .. controls (522,72.84) and (523.34,71.5) .. (525,71.5) .. controls (526.66,71.5) and (528,72.84) .. (528,74.5) .. controls (528,76.16) and (526.66,77.5) .. (525,77.5) .. controls (523.34,77.5) and (522,76.16) .. (522,74.5) -- cycle ;
\draw   (500,40) -- (520,40) -- (520,55) -- (500,55) -- cycle ;
\draw   (250,15) -- (270,15) -- (270,30) -- (250,30) -- cycle ;
\draw    (230,89.5) -- (230,79.5) ;
\draw   (250,35) -- (270,35) -- (270,50) -- (250,50) -- cycle ;
\draw    (275,95) -- (275,85) ;
\draw    (230,79.5) .. controls (220.17,75.07) and (215,80.03) .. (215,69.5) ;
\draw  [fill={rgb, 255:red, 0; green, 0; blue, 0 }  ,fill opacity=1 ] (227,79.5) .. controls (227,81.16) and (228.34,82.5) .. (230,82.5) .. controls (231.66,82.5) and (233,81.16) .. (233,79.5) .. controls (233,77.84) and (231.66,76.5) .. (230,76.5) .. controls (228.34,76.5) and (227,77.84) .. (227,79.5) -- cycle ;
\draw    (230,79.5) .. controls (239.88,75.03) and (245,79.78) .. (245,69.5) ;
\draw  [fill={rgb, 255:red, 0; green, 0; blue, 0 }  ,fill opacity=1 ] (227,89.5) .. controls (227,87.84) and (228.34,86.5) .. (230,86.5) .. controls (231.66,86.5) and (233,87.84) .. (233,89.5) .. controls (233,91.16) and (231.66,92.5) .. (230,92.5) .. controls (228.34,92.5) and (227,91.16) .. (227,89.5) -- cycle ;
\draw   (205,54.5) -- (225,54.5) -- (225,69.5) -- (205,69.5) -- cycle ;
\draw    (260,35) -- (260,30) ;
\draw    (260,60) -- (260,50) ;
\draw    (260,60) .. controls (250.17,64.43) and (245,59.47) .. (245,70) ;
\draw  [fill={rgb, 255:red, 0; green, 0; blue, 0 }  ,fill opacity=1 ] (257,60) .. controls (257,58.34) and (258.34,57) .. (260,57) .. controls (261.66,57) and (263,58.34) .. (263,60) .. controls (263,61.66) and (261.66,63) .. (260,63) .. controls (258.34,63) and (257,61.66) .. (257,60) -- cycle ;
\draw    (260,60) .. controls (269.88,64.47) and (275,59.72) .. (275,70) ;
\draw   (260,70) -- (290,70) -- (290,85) -- (260,85) -- cycle ;
\draw  [draw opacity=0] (290,40.5) -- (330,40.5) -- (330,65.5) -- (290,65.5) -- cycle ;
\draw   (395,15) -- (415,15) -- (415,30) -- (395,30) -- cycle ;
\draw    (390,90) -- (390,80) ;
\draw   (395,35) -- (415,35) -- (415,50) -- (395,50) -- cycle ;
\draw    (345,95) -- (345,85) ;
\draw    (390,80) .. controls (380.17,75.57) and (375,80.53) .. (375,70) ;
\draw  [fill={rgb, 255:red, 0; green, 0; blue, 0 }  ,fill opacity=1 ] (387,80) .. controls (387,81.66) and (388.34,83) .. (390,83) .. controls (391.66,83) and (393,81.66) .. (393,80) .. controls (393,78.34) and (391.66,77) .. (390,77) .. controls (388.34,77) and (387,78.34) .. (387,80) -- cycle ;
\draw    (390,80) .. controls (399.88,75.53) and (405,80.28) .. (405,70) ;
\draw  [fill={rgb, 255:red, 0; green, 0; blue, 0 }  ,fill opacity=1 ] (387,90) .. controls (387,88.34) and (388.34,87) .. (390,87) .. controls (391.66,87) and (393,88.34) .. (393,90) .. controls (393,91.66) and (391.66,93) .. (390,93) .. controls (388.34,93) and (387,91.66) .. (387,90) -- cycle ;
\draw   (350,35) -- (370,35) -- (370,50) -- (350,50) -- cycle ;
\draw    (405,35) -- (405,30) ;
\draw   (330,70) -- (360,70) -- (360,85) -- (330,85) -- cycle ;
\draw    (360,60) -- (360,50) ;
\draw    (360,60) .. controls (350.17,64.43) and (345,59.47) .. (345,70) ;
\draw  [fill={rgb, 255:red, 0; green, 0; blue, 0 }  ,fill opacity=1 ] (357,60) .. controls (357,58.34) and (358.34,57) .. (360,57) .. controls (361.66,57) and (363,58.34) .. (363,60) .. controls (363,61.66) and (361.66,63) .. (360,63) .. controls (358.34,63) and (357,61.66) .. (357,60) -- cycle ;
\draw    (360,60) .. controls (369.88,64.47) and (375,59.72) .. (375,70) ;
\draw    (405,70) -- (405,50) ;
\draw  [draw opacity=0] (420,40) -- (460,40) -- (460,65) -- (420,65) -- cycle ;

\draw (260,22.5) node  [font=\footnotesize]  {$p$};
\draw (125,57.5) node  [font=\footnotesize]  {$f$};
\draw (480,47.5) node  [font=\footnotesize]  {$y$};
\draw (480,72.5) node  [font=\footnotesize]  {$f_{\dagger }( p)$};
\draw (180,52.5) node  [font=\footnotesize]  {$\overset{( i)}{=}$};
\draw (95,57.5) node  [font=\footnotesize]  {$y$};
\draw (540,27.5) node  [font=\footnotesize]  {$p$};
\draw (540,47.5) node  [font=\footnotesize]  {$f$};
\draw (510,47) node  [font=\footnotesize]  {$y$};
\draw (260,42.5) node  [font=\footnotesize]  {$f$};
\draw (215,62) node  [font=\footnotesize]  {$y$};
\draw (275,77.5) node  [font=\footnotesize]  {$f_{\dagger }( p)$};
\draw (140,23) node  [font=\footnotesize]  {$p$};
\draw (405,22.5) node  [font=\footnotesize]  {$p$};
\draw (405,42.5) node  [font=\footnotesize]  {$f$};
\draw (360,42.5) node  [font=\footnotesize]  {$y$};
\draw (345,77.5) node  [font=\footnotesize]  {$f_{\dagger }( p)$};
\draw (310,53) node  [font=\footnotesize]  {$\overset{( ii)}{=}$};
\draw (440,52.5) node  [font=\footnotesize]  {$\overset{( iii)}{=}$};

\end{tikzpicture}
     \caption{Proof of the synthetic \protect\BayesTheorem{}.}%
    \label{fig:bayesTheoremProof}
  \end{figure}
\end{proof}

\begin{exa}%
  \label{ex:bayes-formula}
  In the \kl{discrete partial Markov category} \(\subStoch\), \Cref{th:bayes} instantiates to the formula
  \[p(x) \cdot f(y \given x) = \bayesinv{f}{p}(x \given y) \cdot \sum_{x' \in X} p(x') \cdot f(y \given x'),\]
  which, with the \(P\)-notation of probability theory and dividing both sides by \(\sum_{x' \in X} p(x') \cdot f(y \given x')\), becomes the usual Bayes formula:
  \[
    P(X = x \given Y = y) =
    \frac{P(X=x) \cdot P(Y=y \given X = x)}{\sum_{x' \in X} P(X=x') \cdot P(Y=y \given X=x')}.
  \]  
\end{exa}

\subsection{Pearl's and Jeffrey's updates}%
\label{sec:pearl-jeffrey}

The process for updating a belief on new evidence may depend on the type of
evidence given. \Pearls{}~\cite{pearl1988probabilistic,pearl1990jeffrey} and
\Jeffreys{}~\cite{jeffrey1990logic,shafer1981jeffrey,halpern2017reasoning}
updates are two possible belief updates in light of a new piece of evidence that is not necessarily a single \deterministic{} observation \cite{jacobs_2019}.\footnote{Note that Jacobs and Stein propose a different formulation of Jeffrey's update, using multisets \cite{jacobstein}.} Updating a prior belief according to \Pearls{} rule increases
\emph{validity}, i.e.\ the probability of the new evidence being true according
to our belief~\cite{cho2015introduction}. On the other hand, updating with
\Jeffreys{} rule reduces ``how far'' the new evidence is from our prediction,
i.e.\ it decreases \emph{Kullback-Leibler
divergence}~\cite{jacobs_2019,jacobs2021learning}. 

The difference between these two update rules comes from the fact that they are
based on different types of evidence. \Pearls{} evidence comes as a probabilistic
predicate, i.e.\ a morphism \m{q ፡ Y → I} in a \discretePartialMarkovCategory{}.
\Pearls{} update coincides with the update prescribed by \BayesTheorem{}.

\begin{defi}[Pearl's update]
  \label{def:pearl-update}
  \defining{linkPearls}{}
  Let \m{p ፡ I → X} be a prior distribution and \m{q ፡ Y → I} be a predicate in
  a \discretePartialMarkovCategory{} \m{ℂ}, which is observed through a channel
  \m{f ፡ X → Y}. Pearl's updated prior is defined to be $p ⊲ (f ⨾ q)$ or,
  equivalently, \(\bayesinv{(f ⨾ q)}{p}\), the \BayesianInversion{} of \(f ⨾ q\)
  with respect to \(p\).
  \begin{figure}[h!]

\tikzset{every picture/.style={line width=0.75pt}} %

\begin{tikzpicture}[x=0.75pt,y=0.75pt,yscale=-1,xscale=1]
\draw   (130,15.5) -- (150,15.5) -- (150,30.5) -- (130,30.5) -- cycle ;
\draw    (140,40) -- (140,30) ;
\draw    (140,40) .. controls (130.17,44.43) and (125,39.47) .. (125,50) ;
\draw  [fill={rgb, 255:red, 0; green, 0; blue, 0 }  ,fill opacity=1 ] (137,40) .. controls (137,38.34) and (138.34,37) .. (140,37) .. controls (141.66,37) and (143,38.34) .. (143,40) .. controls (143,41.66) and (141.66,43) .. (140,43) .. controls (138.34,43) and (137,41.66) .. (137,40) -- cycle ;
\draw    (140,40) .. controls (149.88,44.47) and (155,39.72) .. (155,50) ;
\draw   (115,50) -- (135,50) -- (135,65) -- (115,65) -- cycle ;
\draw    (155,70) -- (155,50) ;
\draw    (155,90) -- (155,70) ;
\draw  [draw opacity=0] (160,40) -- (200,40) -- (200,65) -- (160,65) -- cycle ;
\draw   (115,70) -- (135,70) -- (135,85) -- (115,85) -- cycle ;
\draw    (125,70) -- (125,65) ;
\draw   (200,15) -- (220,15) -- (220,30) -- (200,30) -- cycle ;
\draw   (200,35) -- (220,35) -- (220,50) -- (200,50) -- cycle ;
\draw   (200,55) -- (220,55) -- (220,70) -- (200,70) -- cycle ;
\draw    (210,35) -- (210,30) ;
\draw    (210,55) -- (210,50) ;
\draw    (255,90) -- (255,50) ;
\draw   (225,35) -- (285,35) -- (285,50) -- (225,50) -- cycle ;

\draw (140,23) node  [font=\footnotesize]  {$p$};
\draw (125,57.5) node  [font=\footnotesize]  {$f$};
\draw (180,52.5) node  [font=\footnotesize]  {$=$};
\draw (125,77.5) node  [font=\footnotesize]  {$q$};
\draw (210,22.5) node  [font=\footnotesize]  {$p$};
\draw (210,42.5) node  [font=\footnotesize]  {$f$};
\draw (210,62.5) node  [font=\footnotesize]  {$q$};
\draw (255,42.5) node  [font=\footnotesize]  {$(f ⨾ q)_{\dagger }( p)$};

\end{tikzpicture}
     \caption{\protect\PearlsUpdate{}.}
  \end{figure}
\end{defi}

\Jeffreys{} evidence, on the other hand, is given by a distribution on \m{Y}.

\begin{defi}[Jeffrey's update]
\label{def:jeffrey-update}
\defining{linkJeffreys}{}
Let \m{t ፡ I → Y} be a state in \m{ℂ}, representing the evidence. \emph{Jeffrey's updated prior} is \m{t ⨾ \bayesinv{f}{p}}, the composition of the evidence with a \BayesianInversion{} of \(f\) with respect to \(p\), which is defined only up to $p$-\almostSureEquality{}.
\end{defi}

When \Jeffreys{} evidence state \(t\) is \deterministic{}---that is, its
probability mass is concentrated in just one point \(y ∈ Y\)---then the corresponding \Pearls{} evidence becomes a
comparison.
\begin{equation}%
  \label{eq:map-jeffrey-to-pearl}
\tikzset{every picture/.style={line width=0.75pt}}
\begin{tikzpicture}[x=0.75pt,y=0.75pt,yscale=-1,xscale=1]
\draw  [draw opacity=0] (160,40) -- (200,40) -- (200,65) -- (160,65) -- cycle ;
\draw   (135,45) -- (155,45) -- (155,60) -- (135,60) -- cycle ;
\draw    (145,45) -- (145,35) ;
\draw    (230,65) -- (230,55) ;
\draw    (230,55) .. controls (220.17,50.57) and (215,55.53) .. (215,45) ;
\draw  [fill={rgb, 255:red, 0; green, 0; blue, 0 }  ,fill opacity=1 ] (227,55) .. controls (227,56.66) and (228.34,58) .. (230,58) .. controls (231.66,58) and (233,56.66) .. (233,55) .. controls (233,53.34) and (231.66,52) .. (230,52) .. controls (228.34,52) and (227,53.34) .. (227,55) -- cycle ;
\draw    (230,55) .. controls (239.88,50.53) and (245,55.28) .. (245,45) ;
\draw  [fill={rgb, 255:red, 0; green, 0; blue, 0 }  ,fill opacity=1 ] (227,65) .. controls (227,63.34) and (228.34,62) .. (230,62) .. controls (231.66,62) and (233,63.34) .. (233,65) .. controls (233,66.66) and (231.66,68) .. (230,68) .. controls (228.34,68) and (227,66.66) .. (227,65) -- cycle ;
\draw   (205,30) -- (225,30) -- (225,45) -- (205,45) -- cycle ;
\draw    (245,45) -- (245,25) ;

\draw (180,52.5) node  [font=\footnotesize]  {$=$};
\draw (145,52.5) node  [font=\footnotesize]  {$q$};
\draw (215,37.5) node  [font=\footnotesize]  {$y$};
\end{tikzpicture}
\end{equation} In this case, there is no difference between the two update rules. This result
was shown by Jacobs~\cite[Proposition 5.3]{jacobs_2019} in the case of the
Kleisli category of the finitary distribution monad: let us prove it in any
\discretePartialMarkovCategory{}.

\begin{prop}%
  \label{prop:pearlandjeffrey}
  \Pearls{} and \Jeffreys{} updates coincide for \deterministic{} observations.
  Let \m{y ፡ I → Y} be \deterministic{}, then \PearlsUpdate{} on the predicate
  \m{q ፡ Y → I}, as defined in \Cref{eq:map-jeffrey-to-pearl}, is
  \JeffreysUpdate{} on \m{y}.
\end{prop}
\begin{proof}
  The result follows from exactly the same string diagrammatic reasoning as
  \BayesTheorem{} (\Cref{th:bayes}). 
\end{proof}

\section{Exact Observations}%
\label{sec:exactObservations}

Bayesian inference in \discretePartialMarkovCategories{} relies on the existence of \comparators{}. However, many \partialMarkovCategories{} may not have \comparators{} or, even worse, \kl{compare morphisms} might not behave as expected, providing a poor semantics of Bayesian inference (see \Cref{ex:borel-subdistributions-no-comparator}). 

Most inference problems, however, do not need the whole structure of compare morphisms but only a mean of comparing with exact values: in other words, we do not need to be able to compare two samples from a continuous distribution, but just to compare a continuous sample with a discrete value. %

These are called \emph{exact observations}~\cite{stein2021compositional, stein_thesis_2021}. This section shall construct a \partialMarkovCategory{} allowing exact observations over an arbitrary \MarkovCategory{}; we immediately obtain a synthetic Bayes' Theorem (\Cref{th:bayes-exact-observations}) for exact observations.

\subsection{Exact observations}
We define the category \(\exact(ℂ)\) of \emph{\exactObservations{}} over a
\copyDiscardCategory{}, \(ℂ\), by syntactically adding observations of
\deterministic{} evidence. A different but similar construction comes from
the work of Stein and Staton~\cite{stein2021compositional, stein_thesis_2021}, which uses precise supports to define a stronger quotient on morphisms with observations (see \Cref{rem:comparison-cond} for a detailed comparison).
We show that, when \(ℂ\) is a \MarkovCategory{}, processes with \exactObservations{} form a
\partialMarkovCategory{}; \conditionals{} and \normalisations{} for them can be
computed with the \conditionals{} from the original \MarkovCategory{} (\Cref{prop:conditionals-of-programs}). A normal form theorem enables this computation (\Cref{prop:normalisation-of-programs}).

\begin{defi}[Exact observations]%
\label{def:exact-conditioning-cat}%
\defining{linkconsprocC}%
\defining{linkobsv}%
  \AP The category of \intro{processes with exact observations}, \(\exact(ℂ)\),
  over a strict \copyDiscardCategory{} $ℂ$, is the \kl{copy-discard category}
  with the same objects, $\exact(ℂ)_{obj} = ℂ_{obj}$, and morphisms freely generated by
  \begin{itemize}
    \item a morphism $\pur{f} ∈ \exact(ℂ)(X;Y)$ for each morphism $f ∈ ℂ(X;Y)$;
    \item a morphism, $\obsv{x} ∈ \exact(ℂ)(X;I)$ for each \kl{deterministic}
    morphism $x ∈ ℂ(I;X)$, i.e.~such that $x ⨾ ν = x ⊗ x$;
  \end{itemize}
  quotiented by compatibility with the tensors, $\pur{f} ⊗ \pur{g} = \pur{(f ⊗
  g)}$; composition, $\pur{f} ⨾ \pur{g} = \pur{(f ⨾ g)}$; identities,
  $\pur{\id{}} = \id{}$; copy, $\pur{ν} = ν$; and discard, $\pur{ε} = ε$; this ensures that there is an identity-on-objects strict monoidal functor preserving the
  \kl{copy-discard} structure, $\pur{(-)} ፡ ℂ → \exact(ℂ)$. Moreover, it is
  quotiented by compatibility of \kl{exact observations} with the tensor,
  $\obsv{(x ⊗ y)} = \obsv{x} ⊗ \obsv{y}$ and $\obsv{\id{I}} = \id{I}$, and by the
  observation axiom, $\obsv{x} ⊲ \id{} = \obsv{x} ⨾ \pur{x}$ (in
  \Cref{fig:axiom-exact-conditioning-cat}), 
  \begin{figure}[h!]

\tikzset{every picture/.style={line width=0.75pt}} %

\begin{tikzpicture}[x=0.75pt,y=0.75pt,yscale=-1,xscale=1]
\draw    (150,40) -- (150,25) ;
\draw    (150,40) .. controls (140.17,44.43) and (135,39.47) .. (135,50) ;
\draw  [fill={rgb, 255:red, 0; green, 0; blue, 0 }  ,fill opacity=1 ] (147,40) .. controls (147,38.34) and (148.34,37) .. (150,37) .. controls (151.66,37) and (153,38.34) .. (153,40) .. controls (153,41.66) and (151.66,43) .. (150,43) .. controls (148.34,43) and (147,41.66) .. (147,40) -- cycle ;
\draw    (150,40) .. controls (159.88,44.47) and (165,39.72) .. (165,50) ;
\draw    (165,70) -- (165,50) ;
\draw   (200,50) -- (220,50) -- (220,65) -- (200,65) -- cycle ;
\draw  [draw opacity=0] (165,45) -- (200,45) -- (200,60) -- (165,60) -- cycle ;
\draw   (125,62) .. controls (125,63.66) and (126.34,65) .. (128,65) -- (142,65) .. controls (143.66,65) and (145,63.66) .. (145,62) -- (145,50) .. controls (145,50) and (145,50) .. (145,50) -- (125,50) .. controls (125,50) and (125,50) .. (125,50) -- cycle ;
\draw    (210,30) -- (210,25) ;
\draw   (200,42) .. controls (200,43.66) and (201.34,45) .. (203,45) -- (217,45) .. controls (218.66,45) and (220,43.66) .. (220,42) -- (220,30) .. controls (220,30) and (220,30) .. (220,30) -- (200,30) .. controls (200,30) and (200,30) .. (200,30) -- cycle ;
\draw    (210,70) -- (210,65) ;

\draw (135,57.5) node  [font=\footnotesize]  {$x^{\circ }$};
\draw (182.5,52.5) node  [font=\footnotesize]  {$=$};
\draw (210,57.5) node  [font=\footnotesize]  {$\pur{x}$};
\draw (150,21.6) node [anchor=south] [inner sep=0.75pt]  [font=\tiny]  {$X$};
\draw (210,37.5) node  [font=\footnotesize]  {$x^{\circ }$};
\draw (210,21.6) node [anchor=south] [inner sep=0.75pt]  [font=\tiny]  {$X$};
\draw (210,73.4) node [anchor=north] [inner sep=0.75pt]  [font=\tiny]  {$X$};
\draw (165,73.4) node [anchor=north] [inner sep=0.75pt]  [font=\tiny]  {$X$};

\end{tikzpicture}
     \caption{Observation axiom.}%
    \label{fig:axiom-exact-conditioning-cat}%
  \end{figure}
\end{defi}

We interpret the generator, \(\obsv{x} ፡ X → I\), as the observation of the
corresponding \kl{deterministic} evidence, \(x ፡ I → X\). Under this interpretation,
the following equivalent formulation of the axiom explains that 
observing $x$ and computing some $f$ is the same as observing $x$ and passing it to the
input of that computation,
$$x^{∘} ⊲ \pur{f} = x^{∘} ⨾ \pur{x} ⨾ \pur{f}.$$

\ExactObservations{} on a \MarkovCategory{} give a syntax for stochastic
processes with deterministic evidence. In principle, it is not clear how to
compute the semantics of these \exactObservations{} and, in particular, how to
compute conditionals of them. We show that we can give semantics to
\exactObservations{} in the original \MarkovCategory{} by computing their
\normalisations{}. A consequence of this result is that \conditionals{} of
processes with \exactObservations{} can be computed by \conditionals{} in the
original \MarkovCategory{}.

\begin{thm}%
  \label{prop:normalisation-of-programs}%
  Given a strict \kl{Markov category} $ℂ$,
  any \kl{process with exact observations} in 
  $\exact(ℂ)$, can be factored as $(\pur{g} ⨾ \obsv{z}) ⊲ \pur{f}$, for some \total{} morphisms $f$
  and $g$, and some \total{} and \deterministic{} $z$.
  Moreover, $f$ is, \almostSurely{}, its \normalisation{}: 
  $$\norm{((\pur{g} ⨾ \obsv{z}) ⊲ \pur{f})} =_{(\pur{g} ⨾ \obsv{z})} \pur{f}.$$
\end{thm}
\begin{proof}[Proof]
  The proof proceeds by structural induction on the \kl{process with exact
  observations}. We can rewrite any morphism from the original category as
  $\pur{f} = (\pur{ε} ⨾ \obsv{\id{I}}) ⊲ \pur{f}$; and we can rewrite any
  \kl{exact observation} as $\obsv{x} = (\id{} ⨾ \obsv{x}) ⊲ ε$.  
  The tensor of two \kl{processes with
  exact observations}, $(\pur{g}_1 ⨾ \obsv{x}_1) ⊲ \pur{f}_1$ and $(\pur{g}_2 ⨾
  \obsv{x}_2) ⊲ \pur{f}_2$, can be rewritten as follows.
  $$
  ((\pur{g}_1 ⨾ \obsv{x}_1) ⊲ \pur{f}_1) ⊗ ((\pur{g}_2 ⨾ \obsv{x}_2) ⊲ \pur{f}_2) =
  (\pur{(g_1 ⊗ g₂)} ⨾ \obsv{(x₁ ⊗ x₂)}) ⊲ \pur{(f₁ ⊗ f₂)}.
  $$
  The composition of two \kl{processes with exact observations}, $(\pur{g}_1 ⨾
  \obsv{x}_1) ⊲ \pur{f}_1$ and $(\pur{g}_2 ⨾ \obsv{x}_2) ⊲ \pur{f}_2$, can be
  rewritten, using \BayesianInversions{}, as follows.
  $$
  ((\pur{g}_1 ⨾
  \obsv{x}_1) ⊲ \pur{f}_1) ⨾ ((\pur{g}_2 ⨾ \obsv{x}_2) ⊲ \pur{f}_2)
  = (\pur{(g_1 ⊗ (f_1 ⨾ g_2))} ⨾ \obsv{(x_1 ⊗ x_2)}) ⊲ \pur{((\id{} ⊗ x_2) ⨾ \bayesinv{g_2}{f_1} ⨾ f_2)}.
  $$
  We prove this by string diagrammatic reasoning (\Cref{fig:compositionNormalForm}).
  \begin{figure}[h!]

\tikzset{every picture/.style={line width=0.75pt}} %

\begin{tikzpicture}[x=0.75pt,y=0.75pt,yscale=-1,xscale=1]
\draw    (180,35) -- (180,25) ;
\draw    (180,35) .. controls (170.17,39.43) and (165,34.47) .. (165,45) ;
\draw  [fill={rgb, 255:red, 0; green, 0; blue, 0 }  ,fill opacity=1 ] (177,35) .. controls (177,33.34) and (178.34,32) .. (180,32) .. controls (181.66,32) and (183,33.34) .. (183,35) .. controls (183,36.66) and (181.66,38) .. (180,38) .. controls (178.34,38) and (177,36.66) .. (177,35) -- cycle ;
\draw    (180,35) .. controls (189.88,39.47) and (195,34.72) .. (195,45) ;
\draw   (185,45) -- (205,45) -- (205,60) -- (185,60) -- cycle ;
\draw   (185,77) .. controls (185,78.66) and (186.34,80) .. (188,80) -- (202,80) .. controls (203.66,80) and (205,78.66) .. (205,77) -- (205,65) .. controls (205,65) and (205,65) .. (205,65) -- (185,65) .. controls (185,65) and (185,65) .. (185,65) -- cycle ;
\draw    (195,65) -- (195,60) ;
\draw   (155,45) -- (175,45) -- (175,60) -- (155,60) -- cycle ;
\draw    (165,75) -- (165,60) ;
\draw    (165,75) .. controls (155.17,79.43) and (150,74.47) .. (150,85) ;
\draw    (165,75) .. controls (174.88,79.47) and (180,74.72) .. (180,85) ;
\draw   (170,85) -- (190,85) -- (190,100) -- (170,100) -- cycle ;
\draw   (170,117) .. controls (170,118.66) and (171.34,120) .. (173,120) -- (187,120) .. controls (188.66,120) and (190,118.66) .. (190,117) -- (190,105) .. controls (190,105) and (190,105) .. (190,105) -- (170,105) .. controls (170,105) and (170,105) .. (170,105) -- cycle ;
\draw    (180,105) -- (180,100) ;
\draw  [fill={rgb, 255:red, 0; green, 0; blue, 0 }  ,fill opacity=1 ] (162,75) .. controls (162,73.34) and (163.34,72) .. (165,72) .. controls (166.66,72) and (168,73.34) .. (168,75) .. controls (168,76.66) and (166.66,78) .. (165,78) .. controls (163.34,78) and (162,76.66) .. (162,75) -- cycle ;
\draw   (140,85) -- (160,85) -- (160,100) -- (140,100) -- cycle ;
\draw    (150,125) -- (150,100) ;
\draw    (270,25) -- (270,15) ;
\draw    (255,35) .. controls (245.17,39.43) and (240,34.47) .. (240,45) ;
\draw  [fill={rgb, 255:red, 0; green, 0; blue, 0 }  ,fill opacity=1 ] (252,35) .. controls (252,33.34) and (253.34,32) .. (255,32) .. controls (256.66,32) and (258,33.34) .. (258,35) .. controls (258,36.66) and (256.66,38) .. (255,38) .. controls (253.34,38) and (252,36.66) .. (252,35) -- cycle ;
\draw    (255,35) .. controls (264.88,39.47) and (270,34.72) .. (270,45) ;
\draw   (290,45) -- (310,45) -- (310,60) -- (290,60) -- cycle ;
\draw   (290,77) .. controls (290,78.66) and (291.34,80) .. (293,80) -- (307,80) .. controls (308.66,80) and (310,78.66) .. (310,77) -- (310,65) .. controls (310,65) and (310,65) .. (310,65) -- (290,65) .. controls (290,65) and (290,65) .. (290,65) -- cycle ;
\draw    (300,65) -- (300,60) ;
\draw   (260,45) -- (280,45) -- (280,60) -- (260,60) -- cycle ;
\draw    (270,90) -- (270,80) ;
\draw    (270,90) .. controls (260.17,94.43) and (255,89.47) .. (255,100) ;
\draw    (270,90) .. controls (279.88,94.47) and (285,89.72) .. (285,100) ;
\draw   (275,112) .. controls (275,113.66) and (276.34,115) .. (278,115) -- (292,115) .. controls (293.66,115) and (295,113.66) .. (295,112) -- (295,100) .. controls (295,100) and (295,100) .. (295,100) -- (275,100) .. controls (275,100) and (275,100) .. (275,100) -- cycle ;
\draw    (270,65) -- (270,60) ;
\draw  [fill={rgb, 255:red, 0; green, 0; blue, 0 }  ,fill opacity=1 ] (267,90) .. controls (267,88.34) and (268.34,87) .. (270,87) .. controls (271.66,87) and (273,88.34) .. (273,90) .. controls (273,91.66) and (271.66,93) .. (270,93) .. controls (268.34,93) and (267,91.66) .. (267,90) -- cycle ;
\draw   (260,65) -- (280,65) -- (280,80) -- (260,80) -- cycle ;
\draw    (270,25) .. controls (260.17,29.43) and (255,24.47) .. (255,35) ;
\draw  [fill={rgb, 255:red, 0; green, 0; blue, 0 }  ,fill opacity=1 ] (267,25) .. controls (267,23.34) and (268.34,22) .. (270,22) .. controls (271.66,22) and (273,23.34) .. (273,25) .. controls (273,26.66) and (271.66,28) .. (270,28) .. controls (268.34,28) and (267,26.66) .. (267,25) -- cycle ;
\draw    (270,25) .. controls (279.88,29.47) and (285,24.72) .. (285,35) ;
\draw    (285,35) .. controls (285.25,44.72) and (299.75,34.72) .. (300,45) ;
\draw   (230,100) -- (270,100) -- (270,115) -- (230,115) -- cycle ;
\draw    (240,100) -- (240,45) ;
\draw    (250,120) -- (250,115) ;
\draw   (240,120) -- (260,120) -- (260,135) -- (240,135) -- cycle ;
\draw    (250,140) -- (250,135) ;
\draw    (380,40) -- (380,30) ;
\draw    (410,50) .. controls (400.17,54.43) and (395,49.47) .. (395,60) ;
\draw  [fill={rgb, 255:red, 0; green, 0; blue, 0 }  ,fill opacity=1 ] (407,50) .. controls (407,48.34) and (408.34,47) .. (410,47) .. controls (411.66,47) and (413,48.34) .. (413,50) .. controls (413,51.66) and (411.66,53) .. (410,53) .. controls (408.34,53) and (407,51.66) .. (407,50) -- cycle ;
\draw    (410,50) .. controls (419.88,54.47) and (425,49.72) .. (425,60) ;
\draw   (415,60) -- (435,60) -- (435,75) -- (415,75) -- cycle ;
\draw   (415,92) .. controls (415,93.66) and (416.34,95) .. (418,95) -- (432,95) .. controls (433.66,95) and (435,93.66) .. (435,92) -- (435,80) .. controls (435,80) and (435,80) .. (435,80) -- (415,80) .. controls (415,80) and (415,80) .. (415,80) -- cycle ;
\draw    (425,80) -- (425,75) ;
\draw   (385,60) -- (405,60) -- (405,75) -- (385,75) -- cycle ;
\draw    (395,100) -- (395,95) ;
\draw   (385,112) .. controls (385,113.66) and (386.34,115) .. (388,115) -- (402,115) .. controls (403.66,115) and (405,113.66) .. (405,112) -- (405,100) .. controls (405,100) and (405,100) .. (405,100) -- (385,100) .. controls (385,100) and (385,100) .. (385,100) -- cycle ;
\draw    (395,80) -- (395,75) ;
\draw   (385,80) -- (405,80) -- (405,95) -- (385,95) -- cycle ;
\draw    (380,40) .. controls (370.17,44.43) and (350,39.47) .. (350,50) ;
\draw  [fill={rgb, 255:red, 0; green, 0; blue, 0 }  ,fill opacity=1 ] (377,40) .. controls (377,38.34) and (378.34,37) .. (380,37) .. controls (381.66,37) and (383,38.34) .. (383,40) .. controls (383,41.66) and (381.66,43) .. (380,43) .. controls (378.34,43) and (377,41.66) .. (377,40) -- cycle ;
\draw    (380,40) .. controls (389.88,44.47) and (410,39.72) .. (410,50) ;
\draw   (340,80) -- (380,80) -- (380,95) -- (340,95) -- cycle ;
\draw    (350,80) -- (350,50) ;
\draw   (360,60) -- (380,60) -- (380,75) -- (360,75) -- cycle ;
\draw    (370,80) -- (370,75) ;
\draw  [draw opacity=0] (205,75) -- (235,75) -- (235,90) -- (205,90) -- cycle ;
\draw  [draw opacity=0] (310,70) -- (340,70) -- (340,85) -- (310,85) -- cycle ;
\draw    (360,100) -- (360,95) ;
\draw   (350,100) -- (370,100) -- (370,115) -- (350,115) -- cycle ;
\draw    (360,120) -- (360,115) ;

\draw (195,52.5) node  [font=\footnotesize]  {$\pur{g}_{1}$};
\draw (195,72.5) node  [font=\footnotesize]  {$\obsv{x}_{1}$};
\draw (165,52.5) node  [font=\footnotesize]  {$\pur{f}_{1}$};
\draw (180,92.5) node  [font=\footnotesize]  {$\pur{g}_{2}$};
\draw (180,112.5) node  [font=\footnotesize]  {$\obsv{x}_{2}$};
\draw (150,92.5) node  [font=\footnotesize]  {$\pur{f}_{2}$};
\draw (300,52.5) node  [font=\footnotesize]  {$\pur{g}_{1}$};
\draw (300,72.5) node  [font=\footnotesize]  {$\obsv{x}_{1}$};
\draw (270,52.5) node  [font=\footnotesize]  {$\pur{f}_{1}$};
\draw (285,107.5) node  [font=\footnotesize]  {$\obsv{x}_{2}$};
\draw (270,72.5) node  [font=\footnotesize]  {$\pur{g}_{2}$};
\draw (250,107.5) node  [font=\footnotesize]  {$g_{2}^{\dagger}\pur{(f_{1})}$};
\draw (250,127.5) node  [font=\footnotesize]  {$\pur{f}_{2}$};
\draw (425,67.5) node  [font=\footnotesize]  {$\pur{g}_{1}$};
\draw (425,87.5) node  [font=\footnotesize]  {$\obsv{x}_{1}$};
\draw (395,67.5) node  [font=\footnotesize]  {$\pur{f}_{1}$};
\draw (395,107.5) node  [font=\footnotesize]  {$\obsv{x}_{2}$};
\draw (395,87.5) node  [font=\footnotesize]  {$\pur{g}_{2}$};
\draw (360,87.5) node  [font=\footnotesize]  {$g_{2}^{\dagger }\pur{(f_{1})}$};
\draw (370,68.5) node  [font=\footnotesize]  {$\pur{x}_{2}$};
\draw (220,77.5) node  [font=\footnotesize]  {$=$};
\draw (325,77.5) node  [font=\footnotesize]  {$=$};
\draw (360,107.5) node  [font=\footnotesize]  {$\pur{f}_{2}$};

\end{tikzpicture}

     \caption{Composition of processes with \protect\exactObservations{} in normal form.}
    \label{fig:compositionNormalForm}
  \end{figure}

  Finally, let us show that this rewriting computes a \normalisation{},
  $\normal{(\pur{g} ⨾ \obsv{x}) ⊲ \pur{f}} = \pur{f}$; that is, that $f$ is,
  \almostSurely{}, the \normalisation{} of the \kl{process with exact
  observations} $(\pur{g} ⨾ \obsv{x}) ⊲ \pur{f}$. We only apply \kl{totality} of $f$ to obtain
  $
  (\pur{g} ⨾ \obsv{x}) ⊲ \pur{f} =
  (\pur{g} ⨾ \obsv{x}) ⊲ (\pur{f} ⨾ ε) ⊲ \pur{f}. 
  $
\end{proof}

\begin{thm}%
\label{prop:conditionals-of-programs}%
  The category of \exactObservations{} over a strict \MarkovCategory{}, \(\exact(ℂ)\), is a \partialMarkovCategory{}. Its \conditionals{} can be obtained from the \conditionals{} of the base \MarkovCategory{}.
\end{thm}
\begin{proof}   
  By \Cref{prop:normalisation-of-programs}, we can compute a \normalisation{}
  \(\normal{h}\) of a process with \exactObservations{} \(h\), only assuming
  \conditionals{} in \(ℂ\). By \Cref{prop:conditionals-of-normalisation}, a
  \kl{conditional} of \(\normal{h}\) is a \kl{conditional} of \(h\), and this
  \kl{conditional} exists because $ℂ$ is a \kl{Markov category}.
\end{proof}

\PartialMarkovCategories{} with exact observations are a minimal structure where Bayes' theorem holds.
This result suggests that Bayes' theorem may intrinsically be about exact observations.

\begin{thm}[Synthetic Bayes' Theorem]%
  \label{th:bayes-exact-observations}
  For a \MarkovCategory{} \(ℂ\), consider the \partialMarkovCategory{} \(\exact(ℂ)\).
  An exact observation, \(\obsv{y} \colon Y \to I\), from a prior distribution, \(p \colon I \to X\), through a channel, \(f \colon X \to Y\), determines an update that is proportional to evaluating the \BayesianInversion{} of \(f\) on the observation, \(\bayesinv{f}{p}(y)\).
\end{thm}
\begin{proof}
  The equalities follow from: (\emph{i}) the definition of Bayesian inversion
  (\Cref{def:bayes-inversion}), (\emph{ii}) the axiom of exact observations (\Cref{def:exact-conditioning-cat}).
  \begin{figure}[h!]

\tikzset{every picture/.style={line width=0.75pt}} %

\begin{tikzpicture}[x=0.75pt,y=0.75pt,yscale=-1,xscale=1]
\draw    (80,50) -- (80,40) ;
\draw    (80,50) .. controls (70.17,54.43) and (65,49.47) .. (65,60) ;
\draw  [fill={rgb, 255:red, 0; green, 0; blue, 0 }  ,fill opacity=1 ] (77,50) .. controls (77,48.34) and (78.34,47) .. (80,47) .. controls (81.66,47) and (83,48.34) .. (83,50) .. controls (83,51.66) and (81.66,53) .. (80,53) .. controls (78.34,53) and (77,51.66) .. (77,50) -- cycle ;
\draw    (80,50) .. controls (89.88,54.47) and (95,49.72) .. (95,60) ;
\draw    (65,100) -- (65,60) ;
\draw    (95,80) -- (95,75) ;
\draw  [draw opacity=0] (105,45) -- (135,45) -- (135,60) -- (105,60) -- cycle ;
\draw    (170,65) -- (170,55) ;
\draw    (170,65) .. controls (156.89,69.43) and (150,64.47) .. (150,75) ;
\draw  [fill={rgb, 255:red, 0; green, 0; blue, 0 }  ,fill opacity=1 ] (167,65) .. controls (167,63.34) and (168.34,62) .. (170,62) .. controls (171.66,62) and (173,63.34) .. (173,65) .. controls (173,66.66) and (171.66,68) .. (170,68) .. controls (168.34,68) and (167,66.66) .. (167,65) -- cycle ;
\draw    (170,65) .. controls (183.17,69.47) and (190,64.72) .. (190,75) ;
\draw    (150,100) -- (150,90) ;
\draw    (170,40.5) -- (170,35.5) ;
\draw   (70,25) -- (90,25) -- (90,40) -- (70,40) -- cycle ;
\draw   (85,60) -- (105,60) -- (105,75) -- (85,75) -- cycle ;
\draw   (130,75) -- (170,75) -- (170,90) -- (130,90) -- cycle ;
\draw   (160,20.5) -- (180,20.5) -- (180,35.5) -- (160,35.5) -- cycle ;
\draw   (160,40.5) -- (180,40.5) -- (180,55.5) -- (160,55.5) -- cycle ;
\draw   (85,92) .. controls (85,93.66) and (86.34,95) .. (88,95) -- (102,95) .. controls (103.66,95) and (105,93.66) .. (105,92) -- (105,80) .. controls (105,80) and (105,80) .. (105,80) -- (85,80) .. controls (85,80) and (85,80) .. (85,80) -- cycle ;
\draw   (180,87) .. controls (180,88.66) and (181.34,90) .. (183,90) -- (197,90) .. controls (198.66,90) and (200,88.66) .. (200,87) -- (200,75) .. controls (200,75) and (200,75) .. (200,75) -- (180,75) .. controls (180,75) and (180,75) .. (180,75) -- cycle ;
\draw  [draw opacity=0] (200,45) -- (230,45) -- (230,60) -- (200,60) -- cycle ;
\draw    (285,65) -- (285,60) ;
\draw    (245,100) -- (245,90) ;
\draw    (285,45) -- (285,40) ;
\draw   (225,75) -- (265,75) -- (265,90) -- (225,90) -- cycle ;
\draw   (275,25) -- (295,25) -- (295,40) -- (275,40) -- cycle ;
\draw   (275,45) -- (295,45) -- (295,60) -- (275,60) -- cycle ;
\draw   (275,77) .. controls (275,78.66) and (276.34,80) .. (278,80) -- (292,80) .. controls (293.66,80) and (295,78.66) .. (295,77) -- (295,65) .. controls (295,65) and (295,65) .. (295,65) -- (275,65) .. controls (275,65) and (275,65) .. (275,65) -- cycle ;
\draw   (235,55) -- (255,55) -- (255,70) -- (235,70) -- cycle ;
\draw    (245,75) -- (245,70) ;

\draw (120,52.5) node  [font=\footnotesize]  {$\overset{\emph{(i)}}{=}$};
\draw (80,32.5) node  [font=\footnotesize]  {$p$};
\draw (95,67.5) node  [font=\footnotesize]  {$f$};
\draw (150,82.5) node  [font=\footnotesize]  {$f^{\dagger }( p)$};
\draw (170,28) node  [font=\footnotesize]  {$p$};
\draw (170,48) node  [font=\footnotesize]  {$f$};
\draw (215,52.5) node  [font=\footnotesize]  {$\overset{\emph{(ii)}}{=}$};
\draw (245,82.5) node  [font=\footnotesize]  {$f^{\dagger }( p)$};
\draw (285,32.5) node  [font=\footnotesize]  {$p$};
\draw (285,52.5) node  [font=\footnotesize]  {$f$};
\draw (245,62.5) node  [font=\footnotesize]  {$y$};
\draw (95,87.5) node  [font=\footnotesize]  {$y^{\circ }$};
\draw (190,82.5) node  [font=\footnotesize]  {$y^{\circ }$};
\draw (285,72.5) node  [font=\footnotesize]  {$y^{\circ }$};

\end{tikzpicture}
     \caption{Synthetic \protect\BayesTheorem{}.}%
    \label{fig:bayesTheorem-exact}
  \end{figure}
\end{proof}

\begin{rem}[Comparison with the {\(\mathsf{Cond}\)} construction]%
  \label{rem:comparison-cond}
  The \(\mathsf{Cond}\) construction~\cite{stein2021compositional,stein2024probabilistic}
  is similar but different to \(\exact(\cat{C})\).
  For a \kl{Markov category} with precise supports \(\cat{C}\), the copy-discard category \(\mathsf{Cond}(\cat{C})\) is obtained as a quotient of a copy-discard premonoidal category \(\mathsf{Obs}(\cat{C})\):
  morphisms \(X \to Y\) in \(\mathsf{Obs}(\cat{C})\) are triples \((K,f,o)\) of an object \(K\) of \(\cat{C}\), a morphism \(f \colon X \to Y \tensor K\) and a deterministic morphism \(o \colon I \to K\)~\cite[Definition~5.1]{stein2024probabilistic}; morphisms in \(\mathsf{Cond}(\cat{C})\) are morphisms in \(\mathsf{Obs}(\cat{C})\) quotiented by a notion of observational equivalence~\cite[Definition~5.9 and Proposition~5.7]{stein2024probabilistic}.
  We obtain that \(\exact(\cat{C})\) lies between these two—there are two identity-on-objects and full functors \(\mathsf{Obs}(\cat{C}) \to \exact(\cat{C}) \to \mathsf{Cond}(\cat{C})\)—and that \(\mathsf{Cond}(\cat{C})\) is a \kl{partial Markov category}.
\end{rem}

\begin{rem}[Consecutive observations]%
  \label{rem:consecutive-observations}
  One could further quotient \(\exact(\cat{C})\) by the axiom \(\pur{x} \dcomp \obsv{x} = \id{I}\), for all deterministic \(x \colon I \to X\) in \(\cat{C}\), and obtain that consecutive observations of the same value are idempotent: \(\cp_{X} \dcomp (\obsv{x} \tensor \obsv{x}) = \obsv{x} \dcomp \pur{x} \dcomp \obsv{x} = \obsv{x}\).
  However, note that the next proof of faithfulness of the construction (\Cref{thm:embedding}) works without this extra axiom.
\end{rem}

Adding \exactObservations{} to a \kl{Markov category} does not add new equalities to it.

\begin{thm}
  \label{thm:embedding}
  Any strict \kl{Markov category} embeds via a faithful and strict monoidal functor
  into its \kl{partial Markov category} of \kl{processes with exact
  observations}, $\pur{(-)} ፡ ℂ → \exact(ℂ)$.
\end{thm}
\begin{proof}
  The functor is strict monoidal by the definition of $\exact(ℂ)$. Let us show
  it is a faithful embedding. Let us start by distinguishing two types of
  \kl{exact observations}: \emph{(i)} unital \kl{exact observations}, $\obsv{u}
  ∈ \exact(ℂ)(U;I)$, such that there exists an isomorphism $ψ ፡ U ≅ I$; and
  \emph{(ii)} nonunital \kl{exact observations}, $\obsv{x} ∈ \exact(ℂ)(X;I)$
  for any object not isomorphic to the monoidal unit, $X \not\cong I$.

  We now prove that the existence of nonunital \kl{exact observations} remains
  constant under the axioms of the category of \kl{processes with exact
  observations} (\Cref{def:exact-conditioning-cat}). The existence of nonunital
  \kl{exact observations} remains constant under the axioms involving the base
  category, because these do not equate \kl{exact observations}; it also remains
  constant under the unital axiom $\obsv{\id{I}} = \id{I}$, because it only
  involves unital \kl{exact observations}; and it remains constant under the
  observation axiom, $\obsv{x} ⊲ \id{} = \obsv{x} ⨾ \pur{x}$, because the observation
  is the same on both sides of the equation. The only axiom that requires more care
  is the tensor axiom $\obsv{(x ⊗ y)} = \obsv{x} ⊗ \obsv{y}$: this follows from the
  fact that, in a \kl{Markov category}, $X ⊗ Y ≅ I$ implies $X ≅ I$ and $Y ≅ I$.
    
  Now, assume we prove that $\pur{f}_1 = \pur{f}_2$ from the presentation of
  $\exact(ℂ)$ in \Cref{def:exact-conditioning-cat}: there must exist a chain of
  equations $\pur{f}_1 = h_1 = ... = h_k = \pur{f}_2$ that apply the axioms of
  \kl{exact observations}. Because the existence of nonunital observations
  remains constant, we know that no morphism on that chain of equations contains
  nonunital \kl{exact observations}. 
  
  Consider the category $\exact(ℂ)_{\mathsf{unital}}$ freely generated in the
  same way as $\exact(ℂ)$ but only with unital \kl{exact observations}: there
  must exist a chain of equations $\pur{f}_1 = h_1 = ... = h_k = \pur{f}_2$ that
  apply the axioms of this new presentation. However, we can now define a strict
  monoidal functor that erases unital observations, $E ፡
  \exact(ℂ)_{\mathsf{unital}} → ℂ$, defined by $E(\pur{f}) = f$ and
  $E(\obsv{u}) = ε$. This is well-defined under the observation axiom because
  $\obsv{u} ⊲ \id{} = \obsv{u} ⊲ (ε ⨾ \pur{ψ}) =
  \obsv{u} ⨾ \pur{ψ} = \obsv{u} ⨾ \pur{u} ⨾ ε ⨾ \pur{ψ} = \obsv{u} ⨾ \pur{u}$.

  Finally, $f_1 = E(\pur{f}_1) = E(\pur{f}_2) = f_2$, concluding the proof.
\end{proof}

\section{Conclusions and future work}

We have introduced \partialMarkovCategories{}: a common generalisation of \MarkovCategories{} with \conditionals{} and cartesian restriction categories.
\Conditionals{} enable reasoning both about probabilistic problems and about properties of \normalisation{} and \bayesianInversion{}, which appear as derived operations.
We constructed a class of \partialMarkovCategories{}: the Kleisli categories for the Maybe monad on a distributive Markov category (\Cref{thm:partialMarkovMaybeMonad}). Still, we observed that not all \partialMarkovCategories{} are of this form (\Cref{exa:gaussians}).

Some \partialMarkovCategories{} additionally support \exactObservations{}: \discretePartialMarkovCategories{} do so via equality constraints, extending discrete cartesian restriction categories; the construction of \exactObservations{} over a \MarkovCategory{} adds them syntactically.
In both cases, a synthetic Bayes' theorem holds (\Cref{th:bayes,th:bayes-exact-observations}).

\DiscretePartialMarkovCategories{} express the difference between Pearl's and Jeffrey's update procedures (\Cref{sec:pearl-jeffrey}). While updating with Pearl's procedure coincides with applying Bayes' theorem, Jeffrey's update remains more mysterious. This discrepancy and the more recent version of Jeffrey's update~\cite{jacobstein} warrant further work.

Distributive Markov categories were recently introduced for semantics of probabilistic programs. Distributive partial Markov categories enable a semantics for imperative probabilistic programs that additionally support Bayesian updates.
 
\section*{Acknowledgements}

We are thankful for multiple helpful discussions, corrections,
and feedback from
Clémence Chanavat, %
Bart Jacobs, %
Paolo Perrone, %
David Spivak, %
Priyaa Srinivasan, %
Sam Staton, %
Dario Stein, %
Márk Széles, and %
Asier Zubizarreta. %
We thank multiple constructive suggestions from the anonymous reviewers at Logical Methods in Computer Science and the Symposium on Logic in Computer Science.

\newpage
\bibliographystyle{alphaurl}
\bibliography{bibliography.bib}
\end{document}